\documentclass[a4paper,11pt,twoside,reqno,english,fabfinal]{fabsmfart}

\usepackage[T1]{fontenc}
\usepackage[utf8]{inputenc}
\usepackage[style=french]{csquotes}

\usepackage{amsmath,fabalphabets,fabmacromath}
\usepackage{amssymb}
\usepackage{mathrsfs}
\usepackage{bbm}
\usepackage{mathtools}
\usepackage{upgreek}
\usepackage{dsfont}
\usepackage{stmaryrd}
\usepackage{gensymb}

\usepackage{footnote}
\usepackage[backend=bibtex,style=alphabetic,giveninits=true,maxnames=99,maxalphanames=99]{biblatex}
\usepackage{hyperref}
\hypersetup{colorlinks=true, citecolor=ForestGreen, breaklinks, urlcolor=black, linkcolor=Black}
\usepackage{cleveref}

\usepackage[svgnames]{xcolor}
\usepackage{graphicx}
\usepackage{subfig}
\DeclareCiteCommand{\cite}[\color{ForestGreen}\mkbibbrackets] 
  {\usebibmacro{prenote}}
  {\usebibmacro{cite}}
  {\multicitedelim}
  {\usebibmacro{postnote}}

  \usepackage{enumerate}

\theoremstyle{break}

\DeclareMathOperator{\Pb}{\mathbb{P}}
\newcommand{\Prob}[1]{\Pb{\hspace{-0.22em}\left(#1\right)}}
\renewcommand{\indic}{\mathds{1}}
\newcommand\Indic[1]{\indic{_{\{#1\}}}}

\renewcommand{\theenumi}{\roman{enumi}}

\renewcommand{\Z}{\mathbb{Z}}

\usepackage{xargs} 
\usepackage[colorinlistoftodos,prependcaption]{todonotes}
\newcommandx{\rajouter}[2][1=]{\todo[inline,linecolor=red,backgroundcolor=red!25,bordercolor=red,#1]{#2}}
\newcommandx{\thiswillnotshow}[2][1=]{\todo[disable,#1]{#2}} 
\tikzset{notestyle/.append style={align=center}}

\title{Convergence of Eulerian triangulations}
\shorttitle{Convergence of Eulerian triangulations}
\author{A.~Carrance}

\begin{document}
\maketitle

\begin{fabsmfabstract}
  We prove that properly rescaled large planar Eulerian triangulations converge to the Brownian map. This result requires more than a standard application of the methods that have been used to obtain the convergence of other families of planar maps to the Brownian map, as the natural distance for Eulerian triangulations is a canonical oriented pseudo-distance. To circumvent this difficulty, we adapt the layer decomposition method, as formalized by Curien and Le Gall in \cite{curien-legall}, which yields asymptotic proportionality between three natural distances on planar Eulerian triangulations: the usual graph distance, the canonical oriented pseudo-distance, and the Riemannian metric. This notably gives the first mathematical proof of a convergence to the Brownian map for maps endowed with their Riemannian metric. Along the way, we also construct new models of infinite random maps, as local limits of large planar Eulerian triangulations.
\end{fabsmfabstract}

\vfil
\tableofcontents


\section{Introduction}

\subsection{Context}

Eulerian triangulations are face-bicolored triangulations. They can be encountered in several contexts. As their definition is quite straightforward, they are already an object of interest in themselves in enumerative combinatorics (see {\cite{tutte-slicings,bdg, bm-s, albenque-bouttier}}). Moreover, they are in bijection with combinatorial objects such as \textbf{constellations} and \textbf{bipartite maps}, and geometrical objects such as \textbf{Belyi surfaces} (see {\cite{lando-zvonkine}}). They also correspond to the two-dimensional case of \textbf{colored tensor models}, an approach to quantum gravity that generalizes matrix models to any dimension (see Part I of \cite{carrance-th} for an introduction to this topic).

The main aim of this paper is to show that large planar rooted Eulerian triangulations converge to the Brownian map (see \Cref{thm-cv-to-b-map} for a more precise statement). Along the way, we explore uncharted properties of planar Eulerian triangulations. This allows us to construct, in the case of Eulerian triangulations, many random objects and structures whose equivalents already exist for other families of planar maps.

Let us now briefly sketch how this exploration ties in together with proving \Cref{thm-cv-to-b-map}.\\

If one wants to prove that a family of planar maps converges to the Brownian map, the classical method is to use a \textbf{bijection} between this family, and a family of labeled trees, whose labels keep track of the distances in the map.  Obtaining a joint scaling limit for the trees and their label functions is a classical procedure, however, it then remains to deduce from this limit, a scaling limit for the metric space induced by the maps. This was first done independently by Le Gall for triangulations and $2p$-angulations {\cite{legall}}, and Miermont for quadrangulations {\cite{miermont}}, using different technical tools. The list of families amenable to this method has been expanded since then to general maps, general bipartite maps, simple triangulations and odd $p$-angulations \cite{bjm,abraham,addario-berry-albenque, addario-berry-albenque1}\footnote{We stay purposefully vague here, as some of these results rely on bijections with other types of decorated trees.}.

A more recent method applies to local modifications of distances, in families that are already known to converge to the Brownian map. This method, established by Curien and Le Gall in {\cite{curien-legall}} for usual triangulations, uses a \textbf{layer decomposition} of the maps, rather than a bijection with trees. This makes it possible to use an ergodic subbaditivity argument, to obtain that the modified and original distances are asymptotically proportional. This method has recently been extended by Lehéricy to planar quadrangulations (and general maps, \emph{via} Tutte's bijection) in {\cite{lehericy}}, using the layer decomposition of quadrangulations established by Le Gall and Lehéricy in {\cite{legall-lehericy}}. Note that a first notion of layer decomposition was already introduced by Krikun for usual triangulations (without self-loops) {\cite{Krikun2005}} and for quadrangulations {\cite{uipq}}.

In the case of Eulerian triangulations, there exists a bijection with a family of labeled trees, but, as we will explain in the sequel, these labels do not correspond to the usual graph distance from the root, but to an oriented pseudo-distance. This implies that we cannot a priori recover the distances from the labels, so that, while it is still easy to get a scaling limit at the level of labeled trees, we are stuck there without any additional ingredient. This ingredient turns out to be the layer decomposition. Indeed, the usual graph distance can be seen as a local modification of the oriented pseudo-distance, so that the layer decomposition method applies to Eulerian triangulations equipped with these two distances. This method then yields that the oriented pseudo-distance is asymptotically proportional to the usual graph distance, so that the labels do keep track of it up to a small error. This proves to be enough to obtain convergence to the Brownian map.

This is the first time that a combination of these two methods is needed to show such a convergence. It would be interesting to apply this to other families of maps, such as Eulerian quadrangulations.\\

Our layer decomposition of Eulerian triangulations also allows us to prove their convergence to the Brownian map when endowed with the \textbf{Riemannian metric}, which is inherited from the Euclidean geometric realization obtained by gluing equilateral triangles according to the combinatorics of the map. This result is the first of its kind to be proven mathematically, and as such it reinforces the link between random maps and models of $2D$ quantum gravity in theoretical physics, such as Causal Dynamical Triangulations (see for instance {\cite{cdt}}), in which it is the geometric realization itself that is studied.\\

Note that one could want to prove the convergence of planar Eulerian triangulations to the Brownian map using their bijection with bipartite maps, as the convergence for these has already been proven in {\cite{abraham}}. However, this would necessitate to treat the distances on an Eulerian triangulation as a local modification of the distances on the corresponding bipartite map, and thus use a layer decomposition of bipartite maps. As this has not been achieved yet, this route is a priori not easier than the one undertaken here, which has the advantage of uncovering a lot of properties of Eulerian triangulations. However, achieving a layer decomposition of bipartite maps would be interesting in itself.

\subsection{Outline}

In the whole paper, $\mathbf{c}_0$ refers to the constant $\mathbf{c}_0 \in [2/3,1]$ appearing in \Cref{prop-subadd-l} below. The main result of this paper is \Cref{thm-cv-to-b-map}. As the full statement of this theorem necessitates a bit of notation, we postpone it to \Cref{sec cv to b map}. We can however already give a much weaker version of it:
\begin{thm}
\label{thm-weak-cv-to-b-map}
Let $\mathcal{T}_n$ be a uniform random rooted Eulerian planar triangulation with $n$ black faces, equipped with its usual graph distance $d_n$, and let $V(\mathcal{T}_n)$ denote its vertex set. Let $(\mathbf{m}_{\infty},D^*)$ be the Brownian map. The following convergence holds
\[
n^{-1/4}\cdot(V(\mathcal{T}_n),d_n) \xrightarrow[n \to \infty]{(d)} \mathbf{c}_0 \cdot (\mathbf{m}_{\infty},D^*),
\]
for the Gromov-Hausdorff distance on the space of isometry classes of compact metric spaces.
\end{thm}
We give a detailed definition of the Brownian map in \Cref{sec cv to b map}, and we refer to \cite{burago} for a precise definition of the Gromov-Hausdorff distance.

We will see how \Cref{thm-weak-cv-to-b-map} can be obtained from the following result:

\begin{thm}
\label{thm-total-asympt-prop-in-finite-trig}
Let $\mathcal{T}_n$ be a uniform random rooted Eulerian planar triangulation with $n$ black faces, and let $V(\mathcal{T}_n)$ be its vertex set. We denote by $d_n$ its usual graph distance, and by $\vec{d}_n$ its canonical oriented pseudo-distance. For every $\varepsilon > 0$, we have
\[
\Prob{\sup_{x,y \, \in V(\mathcal{T}_n)} \lvert d_n(x,y) - \mathbf{c}_0\vec{d}_n(x,y) \rvert > \varepsilon n^{1/4}} \xrightarrow[n \to \infty]{} 0.
\]
\end{thm}

After giving a precise description of the structure of Eulerian triangulations endowed with their oriented pseudo-distance in \Cref{sec struct}, in \Cref{sec cv to b map} we will give the complete statement of \Cref{thm-cv-to-b-map}, and explain how to prove it using \Cref{thm-total-asympt-prop-in-finite-trig}. \Cref{sec tech prelim,sect-skeleton,sec lhpet,sec bounds,sec subadd} are then devoted to proving \Cref{thm-total-asympt-prop-in-finite-trig}.

Let us sketch the different steps of this proof. After some technical statements in \Cref{sec tech prelim}, pertaining either to asymptotic estimates of $\vec{d}$, or to asymptotics of the enumeration of Eulerian triangulations with a boundary, we detail in \Cref{sect-skeleton} the decomposition of finite rooted planar Eulerian triangulations (possibly with a boundary) into \emph{layers}, determined by the oriented distance from the root. This decomposition makes it possible to describe the random triangulation $\mathcal{T}_n$, defined like in \Cref{thm-weak-cv-to-b-map}, in terms of a branching process whose generations are associated to the layers of $\mathcal{T}_n$. This nice description of $\mathcal{T}_n$ allows us to take the limit $n \to \infty$, to define the Uniform Infinite Planar Eulerian triangulation, $\mathcal{T}_{\infty}$, that is naturally endowed with a decomposition into an infinite number of layers. Now, in \Cref{sec lhpet}, we take a local limit of $\mathcal{T}_{\infty}$ where we view these layers ``from infinity'', which yields the Lower Half-Planar Eulerian Triangulation $\mathcal{L}$. In \Cref{sec subadd}, we explain how the construction of this half-plane model makes it possible to obtain \Cref{thm-total-asympt-prop-in-finite-trig}. First, the layers of $\mathcal{L}$ are i.i.d., which makes it straightforward to apply an ergodic subbadditivity argument to the graph distance $d$ between the root of $\mathcal{L}$ and the $n$-th layer of $\mathcal{L}$. Then, we detail how this result can carry over to finite Eulerian triangulations, first for the graph distance between the root and a random uniform vertex, then between any two vertices, as stated in \Cref{thm-total-asympt-prop-in-finite-trig}. The transfer of the results from $\mathcal{L}$ to finite triangulations necessitates estimates on the distances in $\mathcal{L}$ that are derived in \Cref{sec bounds}.

Finally, \Cref{sec riem} tackles the case of the Riemannian metric, using the same arguments as for the usual graph distance, to show that it is asymptotically proportional to the oriented pseudo-distance, and that, endowed with it, planar Eulerian triangulations still converge to the Brownian map.\\

As our use of the layer decomposition to get the asymptotic proportionality of the oriented and usual distances follows closely the chain of arguments of {\cite{curien-legall}} (albeit with additional difficulties), we purposefully use similar notation, and will omit some details of proofs when they are very similar and do not present any additional subtleties in our case. This is especially the case in \Cref{sect-skeleton}, \Cref{subsec upper bounds} and \Cref{sec subadd}.

\section{Structure of Eulerian triangulations and bijection with trees}
\label{sec struct}

\subsection{Basic definitions}

We start by giving basic definitions related to graphs and maps, that will be needed in the sequel.

\begin{defnt}
Let $G$ be a finite connected graph. A \textbf{map} with underlying graph $G$ is an embedding $f$ of $G$ into an (orientable) surface $S$ such that
\begin{itemize}
\item the images of the (open) edges of $G$ are homeomorphic to (open) segments
\item the images of different edges do not intersect, except at their extremities if they correspond to the same vertex
\item the connected components of $S \setminus f(G)$ are homeomorphic to the open disk; these components are called the \textbf{faces} of the map.
\end{itemize} 
A \textbf{planar} map is a map embedded into the sphere. 

A \textbf{rooted} map is a map equipped with a distinguished oriented edge, called its \textbf{root edge}. The starting vertex of the root edge is called the \textbf{root vertex}.

A \textbf{pointed} map is a map equipped with a distinguished vertex.
\end{defnt}

Maps are usually considered up to orientation-preserving homeomorphisms of the surface $S$. The only automorphism of a map that fixes an oriented edge is the trivial one, so that rooted maps do not have any non-trivial automorphisms. In the sequel, we will consider maps up to isomorphism, unless specified.

Another way to define a map up to isomorphism is to equip its underlying graph with a cyclic ordering of the edges around each vertex.

\begin{defnt}
A \textbf{corner} in a map is an angular sector between two consecutive edges in the cyclic order around a vertex.
\end{defnt}

Two notions that will be useful in the sequel are those of maps with boundaries, and submaps:

\begin{defnt}
A \textbf{map with boundaries} is a map $m$ with a certain number of distinguished faces, that are called its \textbf{external faces}. The other faces of $m$ are naturally called its \textbf{internal faces}. Likewise, the vertices of $m$ that are not incident to any external face are called its \textbf{inner vertices}. We allow two external faces to share vertices, but not edges. We will usually denote by $\partial m$ the boundary cycle of a map $m$ with one boundary.
\end{defnt}

\begin{defnt}
Let $m$ be a rooted map, and let $m'$ be a rooted map with simple boundaries. We say that $m'$ is a \textbf{submap} of $m$, and write $m' \subset m$, if $m$ can be obtained from $m'$, by gluing to each boundary $f_i$ of $m'$ some finite map $u_i$ with a (possibly non-simple) boundary.
\end{defnt}

For any map or graph $G$, we will denote by $V(G)$ its vertex set.\\

In this paper, we will come upon two specific types of maps:

\begin{defnt}
  A \textbf{tree} is a connected graph with no cycle. A \textbf{plane tree} is a map $T$ that, as a graph, is a tree. Since $T$ has no cycle, it is necessarily a planar map.
\end{defnt}

\begin{defnt}
An \textbf{Eulerian triangulation} is a map whose faces have all degree 3, and such that these faces can be properly bicolored, \emph{i.e.}, colored in black and white, such that all white faces are only adjacent to black faces, and \emph{vice versa}.
\end{defnt}

We will also deal with \textbf{Eulerian triangulations with a boundary}, that is, maps with one distinguished face, such that all its other faces have degree 3, and these inner faces can be properly bicolored (\emph{i.e.}, colored in black and white, such that white faces are only adjacent to black faces or to the external face, and similarly for black faces).

By convention, when we root an Eulerian triangulation with a boundary, we do so on an edge adjacent to the external face.

\subsection{Bijection with trees}
\label{subsec bij}

We consider here rooted, planar Eulerian triangulations. Bouttier, Di Francesco and Guitter {\cite{bdg}} have established a bijection between this family of maps and a particular class of labeled trees, whose construction we now briefly recall and extend.\\

Let $A$ be a rooted planar Eulerian triangulation. The orientation of the root edge of $A$ fixes a \textbf{canonical orientation} of all its edges, by requiring that orientations alternate around each vertex. By construction, edges around a given face are necessarily oriented either all clockwise, or all anti-clockwise (with respect to the plane embedding in which the face on the left of the root edge is the infinite one). This fixes the bicoloration of the faces of $A$, by setting for instance that clockwise faces are black, and anti-clockwise faces, white.

From now on, any mention of orientation refers to this canonical orientation.

\begin{figure}[htp]
\centering
\includegraphics[scale=0.8]{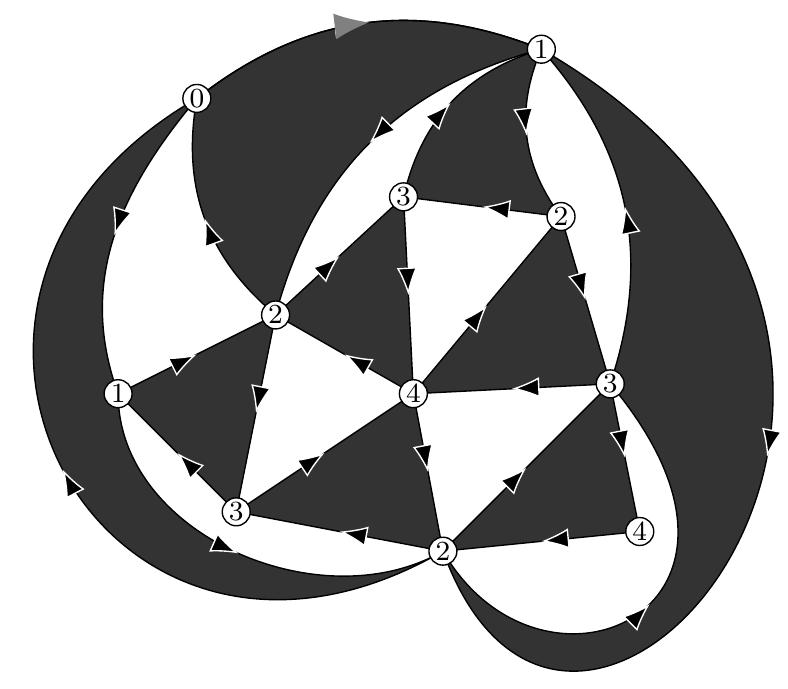}
\caption{A planar Eulerian triangulation with its canonical orientation, bicoloration and oriented geodesic distances.}
\label{EulTrigEx}
\end{figure}

For any pair $(u,v)$ of vertices of $A$, we define the \textbf{oriented distance} $\vec{d}(u,v)$ from $u$ to $v$, as the minimal length of an oriented path from $u$ to $v$. 

Let us state a useful fact. Denoting by $d$ the usual graph distance, in any Eulerian triangulation, we always have:
\begin{equation}
\label{eq bound oriented dist by usual dist}
d \leq \vec{d} \leq 2 d,
\end{equation}
as, in the worst case, the oriented distance forces a path to go through two edges of a triangle instead of just taking the third one.

We also define the \textbf{oriented geodesic distance} to any vertex of $A$, as the oriented distance from the origin (that is, the root vertex $\rho$) to that vertex. This gives a labeling of the vertices of $A$, such that the sequence of labels around any triangle, starting from the minimal label, is of the form $n \to n+1 \to n+2$. 

Let us now introduce a bit of notation that will be of use in the sequel.

\begin{defnt}
In a rooted Eulerian triangulation $A$, a \textbf{vertex of type \boldmath$n$} is a vertex whose canonical labeling by the oriented geodesic distance is \unboldmath$n$. An \textbf{edge of type \boldmath$n \to m$} is an oriented edge that starts at a vertex of type \unboldmath$n$ and ends at a vertex of type $m$. A \textbf{triangle of type \boldmath$n$} is a triangle adjacent to a vertex of type \unboldmath$n-1$, one of type $n$ and one of type $n+1$.
\end{defnt}

By keeping only the edge of type $n+1 \to n+2$ in each black face of type $n+1$, we construct a graph $T$ whose vertices, that correspond to $V(A)\setminus\{\rho\}$, are labeled by integers. Moreover, it is \textbf{well-labeled} in the sense that the labels of adjacent vertices differ by exactly 1, and that the root vertex has label 1. By construction, those labels are all positive, but we do not include it in the definition of being well-labeled, for reasons that will be clear soon.

\begin{lemma}\textnormal{\cite{bdg}}
For any planar rooted Eulerian triangulations $A$, the corresponding labeled graph $T$ is a plane tree.
\end{lemma}

This tree (which is a spanning tree of the subgraph of $A$ induced by $V(A) \setminus \{\rho\}$) is naturally rooted at the corner of a vertex of type 1, that corresponds to the root edge of $A$ (see \Cref{EulTrigToTree}).

\begin{figure}[htp]
\centering
\includegraphics[scale=0.8]{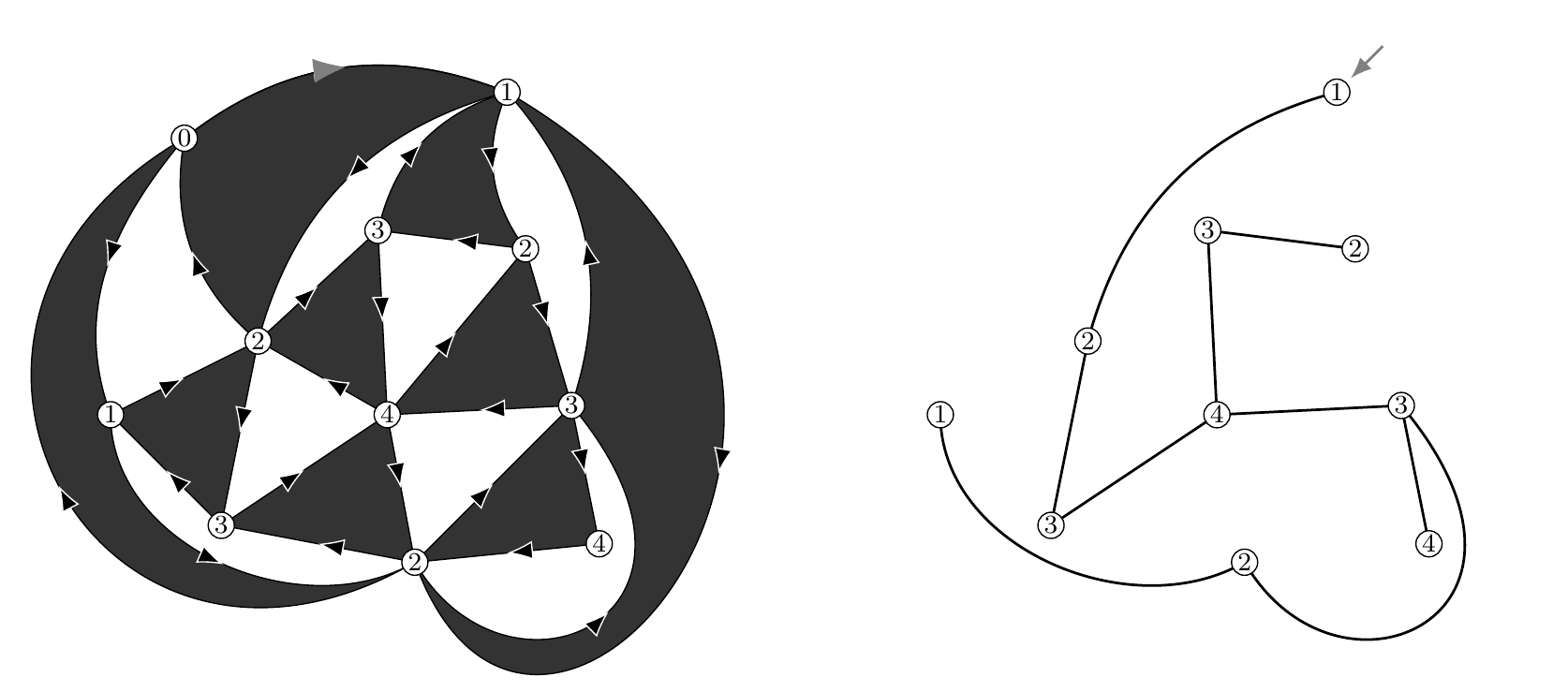}
\caption{The construction of the well-labeled tree associated to the triangulation of \Cref{EulTrigEx}.}
\label{EulTrigToTree}
\end{figure}

\begin{thm}\textnormal{\cite{bdg}}
The above mapping $\varphi_n$ from the set of rooted planar Eulerian triangulations with $n$ faces, to the set of well-labeled trees with $n$ edges with positive labels, is a bijection.
\end{thm}

Bouttier-Di Francesco-Guitter have also detailed the inverse construction from trees to triangulations, that we recall now. 

The inverse construction consists in building iteratively the black triangles of $A$. Starting from a well-labeled rooted plane tree with positive integers, the first step consists in adding an origin (labeled 0). We then create a black triangle of type 1 to the right of each edge of type $1 \to 2$, by adding edges between the origin and the two vertices of the edge. The creation of these black triangles splits the original external face into a number of white faces. By construction, the clockwise sequence of labels around any of these white faces is of the form $0 \to 2 \to \dotsm \to 2 \to 1$, where all the labels between the first and last ``2'' are greater or equal to 2, and all increments but the first are $\pm 1$. For each white face $F$ that is not already a triangle, and for each type-$(2 \to 3)$ edge whose right side is adjacent to $F$, we create a black triangle of type 2 to the right of this edge, by adding edges between its vertices and the unique vertex labeled 1 around $F$. This induces a splitting of $F$ into smaller white faces, and we repeat the procedure again, until all labels are exhausted. This yields an Eulerian triangulation $A$ rooted at the $0 \to 1$ edge linking the origin to the root of $T$ (see \Cref{TreeToEulTrig}). 

\begin{figure}[htp]
\centering
\includegraphics[scale=0.8]{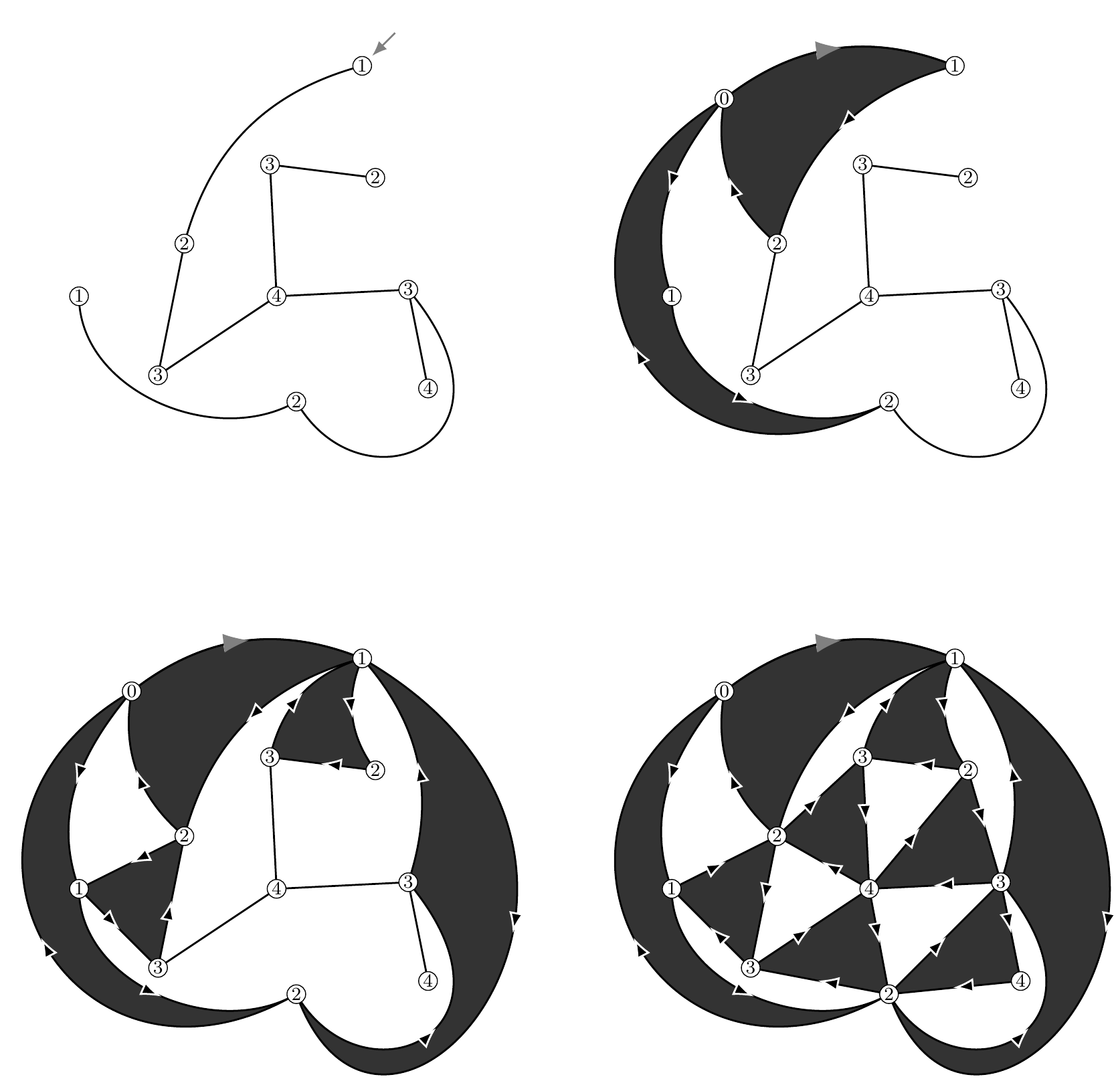}
\caption{The inverse construction of the triangulation of \Cref{EulTrigEx} from the labeled tree.}
\label{TreeToEulTrig}
\end{figure}



In the sequel, it will be more convenient to deal with trees whose labels are not necessarily positive. For that purpose, we point $A$ at some vertex $v_*$, and consider the new labeling $\ell$ on $V(A)$ given by:
\[
\ell(u)=\vec{d}(v_*,u)-\vec{d}(v_*,\rho).
\]
Let us proceed with this new labeling as we did with the oriented geodesic distance: starting from the triangulation $A$, in each black triangle of type $n+1$ (where now the integer $n$ may be nonpositive), we only keep the edge of type $n+1 \to n+2$. This construction now gives a correspondance between pointed planar Eulerian triangulations with $n$ black triangles, and well-labeled trees with $n$ edges, with no constraint on the label signs, that still maps the vertices of the triangulation, minus the distinguished vertex, to those of the tree. Let us now pay attention to the rooting: since the distinguished vertex that we add to the tree is no longer the origin of the triangulation (and, conversely, the root corner of the tree is no longer at a minimal label), we need additional information with either object, to know how to root the other. Thus, in the triangulation-to-tree direction, we start from a couple $(A, \varepsilon)$, with $\varepsilon \in \{0,1\}$: depending on the value of $\varepsilon$, we root the tree $T$ at the edge remaining from the (black) root face of $A$, either with its original direction, or the reverse one, see \Cref{PointedEulTrigToTree}. Note that we have to shift the labels of $T$ by an integer $L(A)$ between -2 and 2, so that the root corner has label 1, so that the labeling of the vertices of the tree is:
\[
l(u)=\vec{d}(v_*,u)-\vec{d}(v_*,\rho)+L(A).
\]

\begin{figure}[htp]
\centering
\includegraphics[scale=0.8]{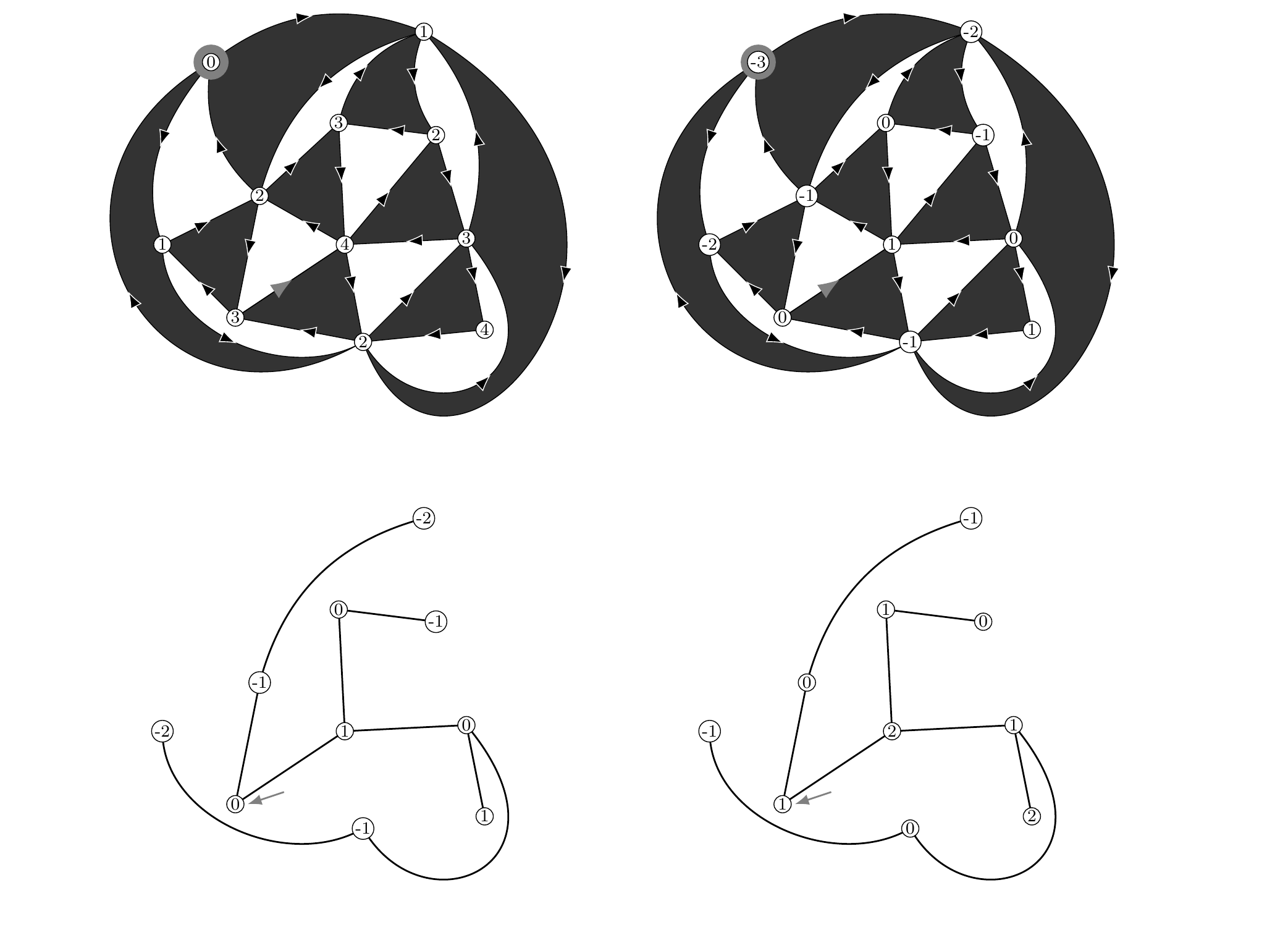}
\caption{Example of the mapping from a pointed, rooted planar Eulerian triangulation to a well-labeled trees: top left, the original triangulation $A$, top right, the triangulation with the shifted labels. To determine the rooting of the associated tree $T$, we need an additional parameter $\varepsilon \in \{0,1\}$. Choosing it to be 0, we root $T$ at the edge remaining from the root face of $A$, with its original direction: bottom left, the tree with the labeling from top right, an bottom right, the shifted labels that respect the condition that the root corner has label 1. Here, the value of the parameter $\delta$ is 1, as the root edge of $A$ has type $3 \to 4$, and its root face has type $3$.}
\label{PointedEulTrigToTree}
\end{figure}

Conversely, in the tree-to-triangulation direction, we start from a couple $(T,\delta)$, with $\delta \in \{0,1,2\}$: depending on the value of $\delta$, the root edge of $A$ is either the type-$n-1 \to n$, $n \to n+1$, or $n+1 \to n-1$ edge of the black face adjacent to the root edge of $T$ (where this face is itself of type $n$).

It is straightforward to prove similarly to the case of $\varphi_n$, that this new mapping is a bijection as well:

\begin{prop}
Let us denote by $\overline{\mathfrak{T}}_n$ the set of pointed, rooted planar Eulerian triangulations with $n$ faces, and by $\mathbb{T}_n$ the set of well-labeled trees with $n$ edges. The mapping $\psi_n$ detailed above from $\overline{\mathfrak{T}}_n \times \{0,1\}$ to $\mathbb{T}_n \times \{0,1,2\}$ is a bijection.
\end{prop}

Consider now the random variable $(\mathcal{T}^{\bullet}_n,\varepsilon_n)$, where $\mathcal{T}^{\bullet}_n$ is picked uniformly at random in $\overline{\mathfrak{T}}_n$, and $\varepsilon_n$ is uniform over $\{0,1\}$. Since these additional decorations of $\psi_n$ only influence the rooting of the obtained map, the image of $(\mathcal{T}^{\bullet}_n,\varepsilon_n)$ can be written as $(\mathscr{T}_n,\delta_n)$, where $\mathscr{T}_n$ is uniform over $\mathbb{T}_n$, and $\delta_n$ is uniform over $\{0,1,2\}$.

Furthermore, note that, if $T=\psi_n(A)$, for every vertex $u$ of $T$, the oriented distance from $v_*$ to $u$ in $A$ is given by:
\begin{equation}
\label{eq:distance-from-root-from-labels}
\vec{d}(v_*,u)=l(u) - \min_{v \in V(T)}l(v) +1.
\end{equation}

To get more general information on the oriented distances in $A$ from the labels of $T$, we need a bit of additional notation.

First observe that, with the construction of $A$ from $T$, a corner $c$ of $T$ is always incident in $A$ to an edge oriented from the first corner $c'$ encountered when going anticlockwise around the unique face of $T$, starting at $c$, and that has label $l(c)-1$ (by convention, we define the label of a corner in $T$ to be the label of the associated vertex). Indeed, either this corner was already adjacent to $c$ in $T$, or we create an edge between them when adding a black triangle to the right of the edge of type $l(c) \to l(c)+1$ that starts at $c$.

We call $c'$, the \textbf{predecessor} of $c$, and denote it by $p(c)$. (The predecessor of a corner of minimal label is naturally the corner of the origin to which it is linked in the first step of the construction.) We also call $p^k(c)$ the $k$-th predecessor of $c$, whenever it is defined.

Let us denote by $l_0$ the minimal label of the vertices of $T$. For a corner $c$ of $T$, the edge $p(c) \to c$ in $A$ is obviously of type $l(c) -l_0  \to l(c)-l_0+1$ (for the oriented distance from $v_*$). This implies that the path from $v_*$ to $c$ going through all its predecessors: $v_* \to p^{(l(c)-1)}(c) \to \dotsm \to p(c) \to c$ is a geodesic for $\vec{d}$ in $A$.

For any pair of corners $c, c'$ in $T$, we denote by $[c,c']$ the set of corners of $T$ encountered when starting from $c$, going anticlockwise around $T$, and stopping at $c'$. The property \eqref{eq:distance-from-root-from-labels} yields the following bound on oriented distances in $A$:
\begin{prop}
\label{prop dists from labels}
Let $c,c'$ be two corners of $T$, with corresponding vertices $u,v$. Then
\[
\vec{d}(u,v) \leq 2\left(l(u)+l(v)-2\min_{c'' \in [c,c']}l(c'')+2\right).
\]
\end{prop}

\begin{proof}
Let $m=\min_{c'' \in [c,c']}l(c'')$, and let $c''$ be the first corner in $[c,c']$ such that $l(c'')=m$. Then $c''$ is the $(l(c)-m)$-th predecessor of $c$. Moreover, by definition, $p(c'')$ does not belong to $[c,c']$, so that it is also the $(l(c')-m+1)$-th predecessor of $c'$. Thus, the predecessor geodesic $p(c'') \to c'' \to \dotsm \to c$, concatenated with the similar geodesic $p(c'') \to \dotsm \to c'$, is a simple path in $A$ made of $l(c) + l(c') -2m+2$ edges. However, part of it is not oriented from $c$ to $c'$, so that we lose a multiplicative factor of 2 when deducing a bound on the distance from $u$ to $v$.
\end{proof}

As will be clearer in the proof of \Cref{thm-cv-to-b-map}, this factor of 2 is really the stumbling block that prevents us from reaching the convergence to the Brownian map using only the bijective approach.

\subsection{Convergence of the labeled trees}
\label{subsec cv-of-trees}

From what precedes, starting from a uniform random rooted, pointed planar Eulerian triangulation with $n$ black faces, we get a uniform random well-labeled tree $\mathscr{T}_n$ with $n$ edges. Let us now explain how we can make sense of taking a continuum scaling limit of the latter. We first define the \textbf{contour process} of $\mathscr{T}_n$: let $e_0, e_1, \dotsm, e_{2n-1}$ be the sequence of oriented edges bounding the unique face of $\mathscr{T}_n$, starting with the root edge, and ordered counterclockwise around this face. Then let $u_i=e_i^{-}$ be the $i$-th visited vertex in this contour exploration, and set the \textbf{contour process} of $\mathscr{T}_n$ at time $i$:
\[
C_{n}(i):=d_{\mathscr{T}_n}(u_0,u_i), \, 0 \leq i \leq 2n-1,
\]
with the convention that $u_{2n}=u_0$ and $C_{n}(2n)=0$. We also extend $C_{n}$ by linear interpolation between integer times: for $0 \leq s \leq 2n$
\[
C_{n}(s)=(1-\{s\})C_{n}(\lfloor s \rfloor) + \{s\}C_{n}(\lfloor s \rfloor+1),
\]
where $\{s\}=s - \lfloor s \rfloor$ is the fractional part of $s$. Thus, the contour process $C_{n}$ is a non-negative path of length $2n$, starting and ending at 0, with increments of 1 between integer times. We will use the \textbf{rescaled contour process} of $\mathscr{T}_n$:
\[
C_{(n)}(t)=\frac{C_n(2nt)}{\sqrt{2n}}, \, 0\leq t \leq 1.
\]

We define similarly the \textbf{rescaled label function} of $\mathscr{T}_n$:
\[
L_{(n)}(t)=\frac{L_n(2nt)}{n^{1/4}}, \, 0 \leq t \leq 1,
\]
where, similarly, we start by defining $L_n(i)$ as the  label of $u_i$ for $i \in \{0, 1, \dots, 2n\}$, then interpolate between integer times.

Finally, for a continuous, non-negative function $f: [0,1] \to \mathbb{R}_+$ such that $f(0)=f(1)=0$, for any $s,t \in [0,1]$, we set
\[
\check{f}(s,t)=\inf\{f(u) \lvert s \wedge t \leq u \leq s \vee t\}.
\]

Then we have the following result:
\begin{thm}\textnormal{\cite{janson-marckert}}
\label{thm-cv-of-labeled-tree}
It holds that
\begin{equation}
\left(C_{(n)},L_{(n)}\right)\xrightarrow[n\to \infty]{(d)}(\mathbbm{e},Z),
\end{equation}
in distribution in $\mathcal{C}([0,1],\mathbb{R})^2$, where $\mathbbm{e}$ is a standard Brownian excursion, and, conditionally on $\mathbbm{e}$, $Z$ is a continuous, centered Gaussian process with covariance
\[
Cov(Z_s,Z_t)=\check{\mathbbm{e}}_{s,t}, \, s,t \in [0,1].
\]
\end{thm}

As this convergence will be crucial to ultimately prove the convergence of Eulerian triangulations to the Brownian map, to describe and analyse these triangulations, we will need to use their \emph{oriented} distances, instead of the usual graph distance.

 \subsection{Structure for oriented distance}
\label{subsec structure oriented dist}

Let us consider a rooted Eulerian triangulation $A$, equipped with its canonical orientation and oriented geodesic distance $\vec{d}$. For each type-$n$ black face $f$ of $A$, there is exactly one white face $f'$ that shares its $n+1 \to n-1$ edge. We call the union of $f$ and $f'$ a \textbf{type-\boldmath$n$ module}. Now, imagine that for each type-\unboldmath$n$ module of $A$, we trace the ``diagonal'' linking its two type-$n$ vertices, and direct it from the black triangle to the white one (see \Cref{EvenCurve}). We will call this, directing the module \textbf{left-to-right}.

\begin{figure}[htp]%
\centering%
\includegraphics[scale=1.5]{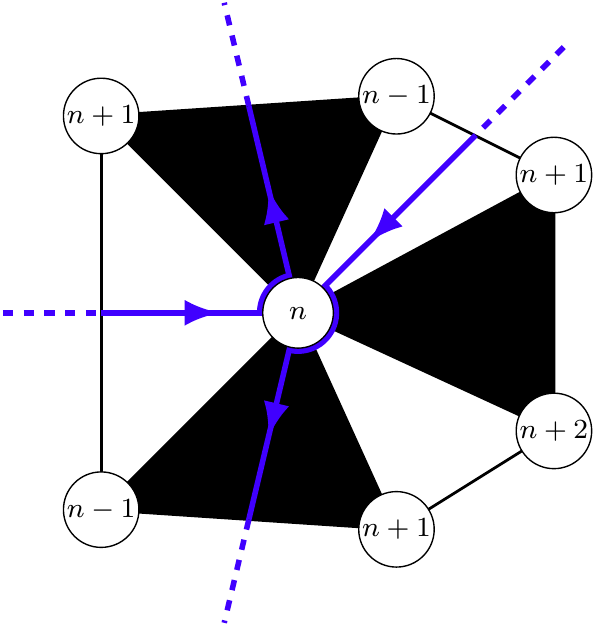}%
\caption{The union of type-$n$ module diagonals can be decomposed into a set of simple closed curves by pairing, at each vertex of type $n$, each ingoing diagonal with the next one clockwise, which is necessarily outgoing.}
\label{EvenCurve}
\end{figure}

We now explain how to describe the union of the diagonals of type-$n$ modules as a set of simple closed curves. First note that, by construction, this union of diagonals only goes through vertices of type $n$. Moreover, around each vertex $u$ of type $n$, these directed diagonals alternate between ingoing and outgoing. Indeed, around $u$, after each black type-$n$ triangle, there is necessarily a white type-$n$ triangle before the next black type-$n$ triangle. This stems from the fact that the triangles around $u$ can only be of type $n-1$, $n$ or $n+1$, and that along each edge, the oriented geodesic distance can only change by 1 or 2 (see \Cref{EvenCurve}). Now, to resolve the intersections at type-$n$ vertices, we can take the convention that if a curve arrives at a vertex $u$ by an ingoing diagonal $\delta$, it will immediately leave $u$ by the first outgoing diagonal that we encounter going clockwise around $u$, starting at $\delta$ (see \Cref{EvenCurve}).

This yields a set of closed curves that we denote by \boldmath$\mathcal{C}_n(A)$. By construction, the curves in \unboldmath$\mathcal{C}_n(A)$ separate vertices at oriented distance $n+1$ or higher from the origin, and they go counter-clockwise around these vertices (\emph{i.e.}, the origin lies on their right).

\begin{figure}[htp]%
\centering%
\includegraphics[scale=1.3]{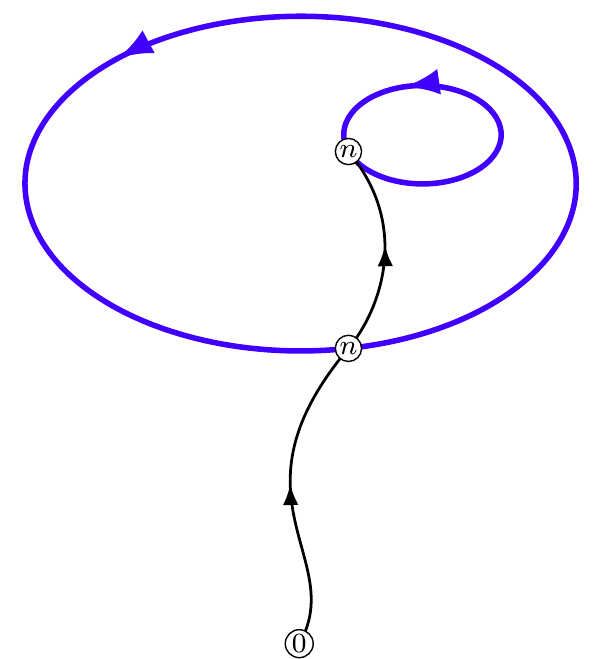}%
\caption{Two disjoint curves in  $\mathcal{C}_n(A)$ cannot encircle one other.}
\label{CurvesCantEncircleEachOther}
\end{figure}

\begin{lemma}
Let \unboldmath$A$ be a planar rooted Eulerian triangulation. For a given vertex $v$ at (oriented) distance at least $n+1$ from the root, there is a unique curve in $\mathcal{C}_n(A)$ that separates $v$ from the root.

Moreover, all curves in  $\mathcal{C}_n(A)$ are simple.
\end{lemma}

\begin{proof}
First consider two disjoint curves in $\mathcal{C}_n(A)$ that separate the same vertex $v$ from the origin. Necessarily, a geodesic path from the root to a vertex belonging to one of them should go through the other, and thus have length at least \unboldmath$n+1$ (see \Cref{CurvesCantEncircleEachOther}).

Now, if two curves of $\mathcal{C}_n(A)$ intersect at a vertex of type $n$, then by our resolution rule, they cannot go counterclockwise around the same region of $A$ (see \Cref{NiceResolutionRules}), and, as explained before, these are precisely the regions they separate from the origin.

This rule also implies that a curve $\mathscr{C}$ in $\mathcal{C}_n(A)$ cannot go twice through the same type-$n$ vertex. Indeed, if that were the case, then $\mathscr{C}$ would separate from the origin vertices of oriented distance $n-1$ and less, so that any oriented geodesic from the origin to these vertices should be of length at least $n+1$ (see \Cref{NiceResolutionRules}).
\end{proof}

\begin{figure}[htp]%
\centering%
\includegraphics[scale=1.3]{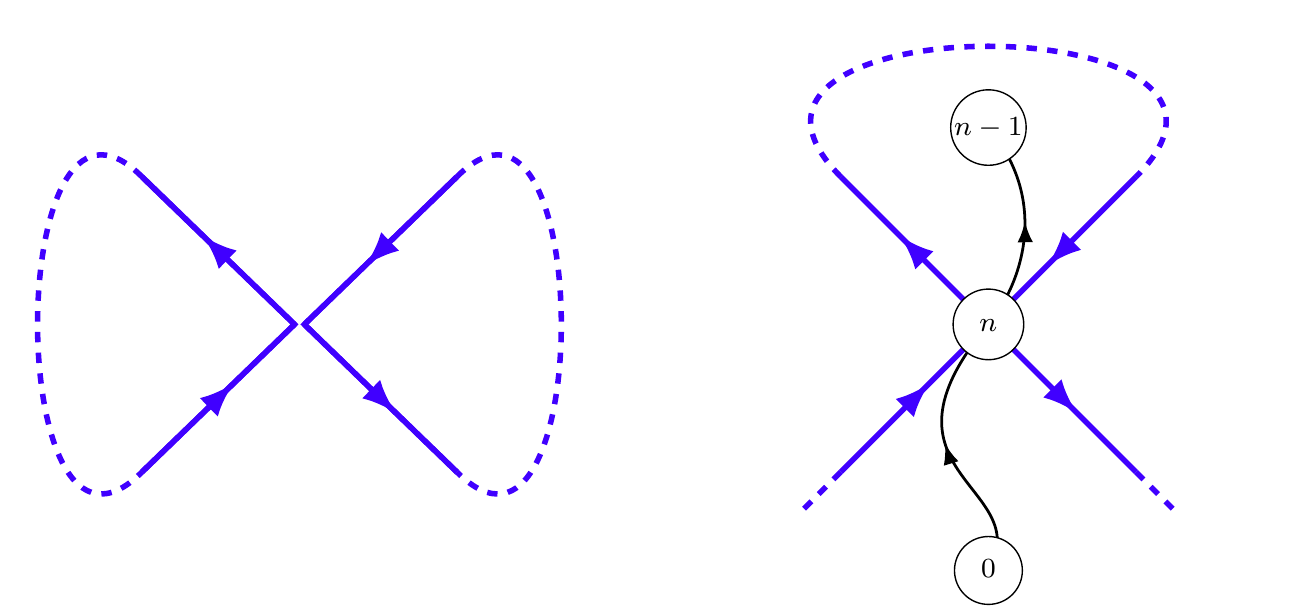}%
\caption{From the rule we have chosen to resolve intersections (see \Cref{EvenCurve}), two curves in $\mathcal{C}_n(A)$ going through the same vertex cannot encircle one other (left), and one curve cannot go through the same vertex twice (right), as it would imply that all oriented paths from the origin to at least one vertex of type $n-1$ or less, would have at least length $n+1$.}
\label{NiceResolutionRules}
\end{figure}

\begin{figure}[htp!]
\centering
\includegraphics[scale=0.7]{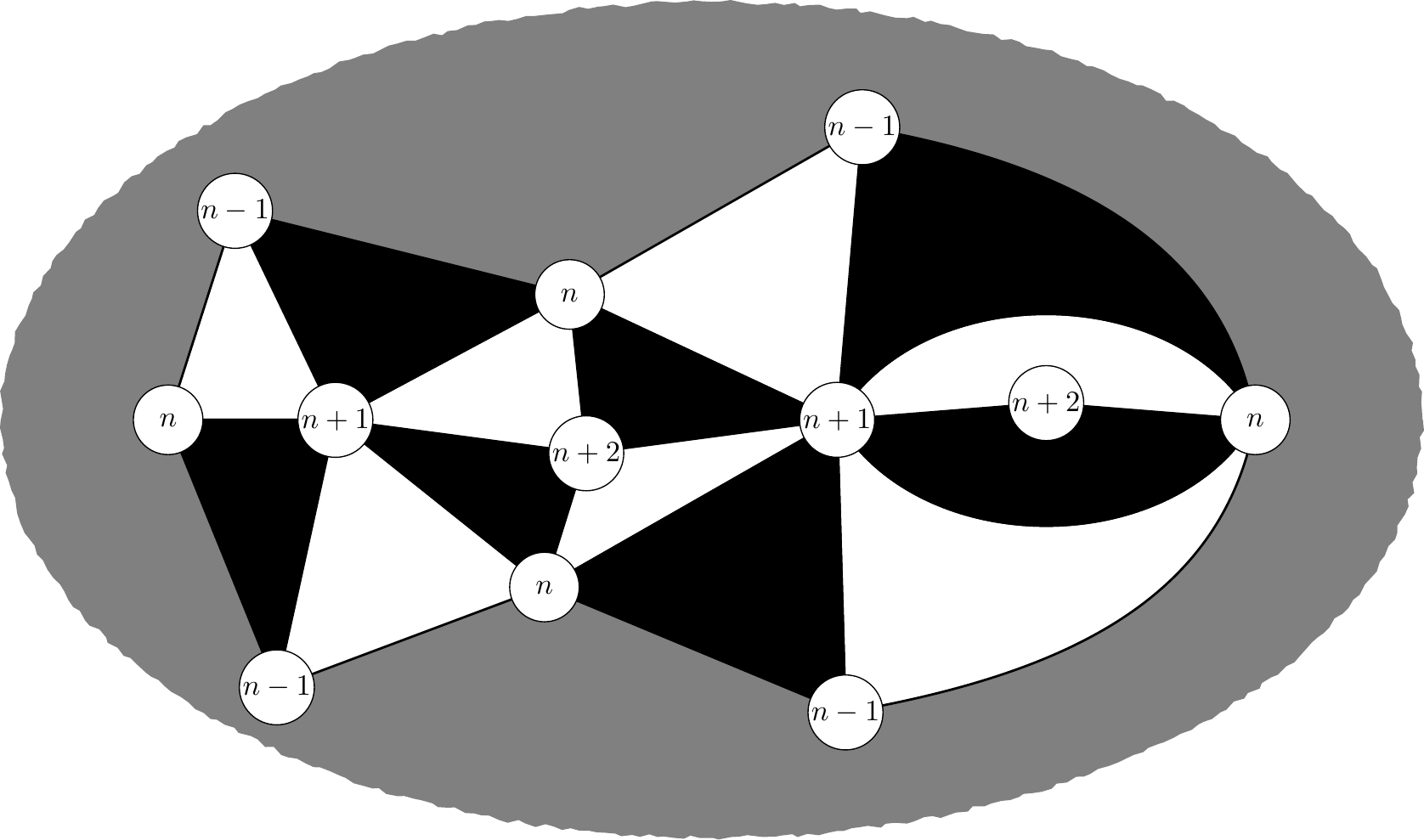}\\
\vspace{2em}
\includegraphics[scale=0.7]{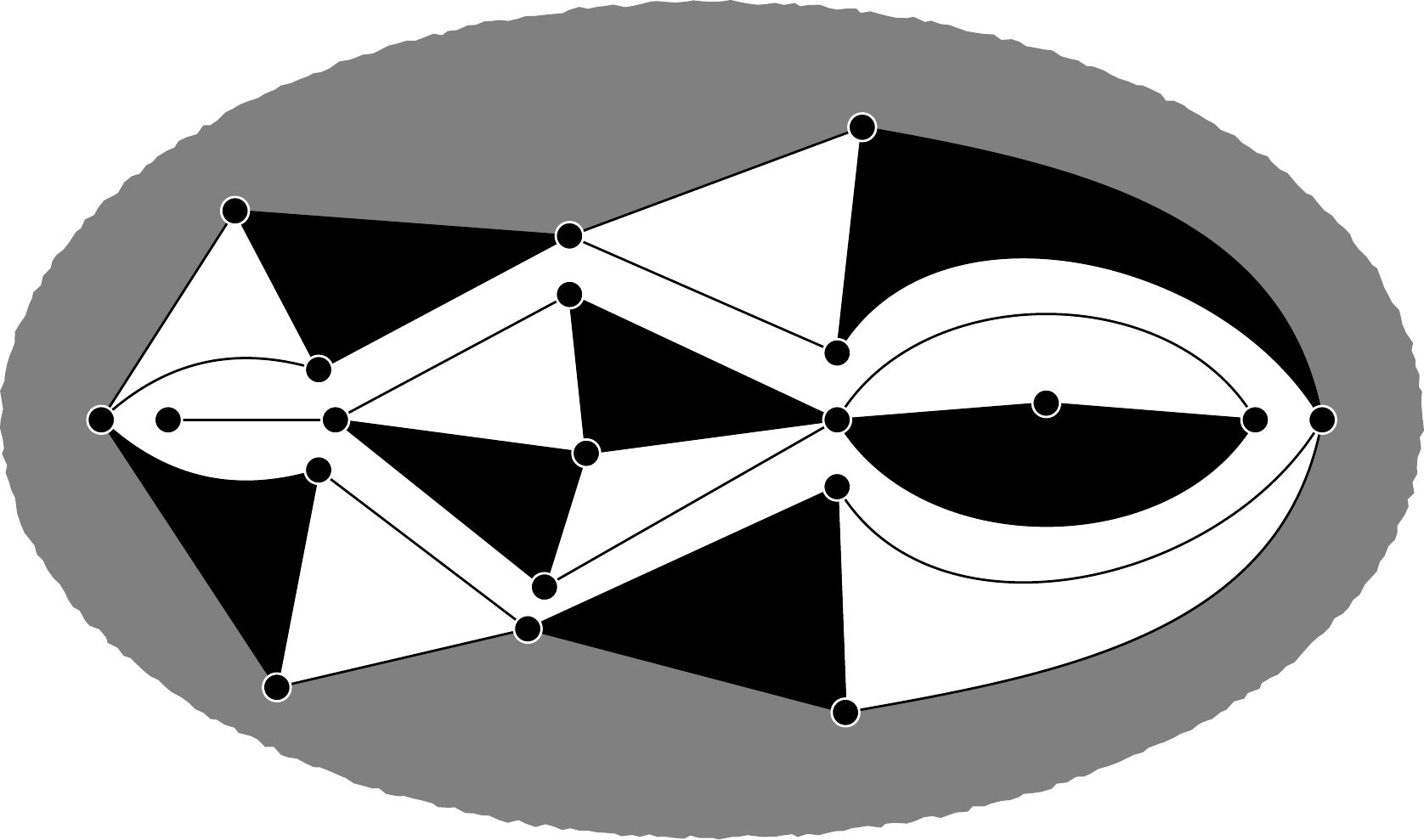}
\caption{In an Eulerian triangulation (top), we cut along the edges of type $n \to n+1$ to separate the ball of radius $n$ from the components of its complement (bottom). This possibly induces the duplication of edges and vertices in the ball.}
\label{bmap LocalBallEulTrig}
\end{figure}

We define the \textbf{ball} \boldmath$\mathcal{B}_{n}(A)$ as the submap of  \unboldmath$A$ obtained by keeping only the faces and edges of $A$ incident to at least a vertex at distance $n-1$ or less from the origin, cutting along the edges of type $n \to n+1$, and filling in the produced holes by simple faces (see \Cref{bmap LocalBallEulTrig} for a local depiction of this procedure). Thus, in $\mathcal{B}_{n}(A)$, for each closed curve $\mathscr{C} \in \mathcal{C}_{n}(A)$, we have replaced all faces that $\mathscr{C}$ separates from the root, by a single, simple face. In particular, if two faces of $A$ of type $n$ share a type-$(n \to n+1)$ edge, in $\mathcal{B}_{n}(A)$ their respective type-$(n \to n+1)$ edges are not identified, so that their common type-$(n+1)$ vertex gives rise to two vertices in $\mathcal{B}_{n}(A)$ (see \Cref{ModulesEx}). Two type-$n$ faces $f, f'$ may also share a type-$(n+1)$ vertex $v$ but no edge: in that case, it means that $v$ is also shared by faces of types $n+1$, so that we would need to add these faces and the type $n \to n+1$ edges they share with $f$ and/or $f'$, in order to identify the type-$(n+1)$ vertices of $f$ and $f'$ into $v$. Note that as $\mathcal{B}_{n}(A)$ contains all the type-$(n-1 \to n)$ edges of $A$, type-$n$ vertices of $A$ are never duplicated in $\mathcal{B}_{n}(A)$.

\begin{figure}[htp]%
\centering%
\includegraphics[scale=0.6]{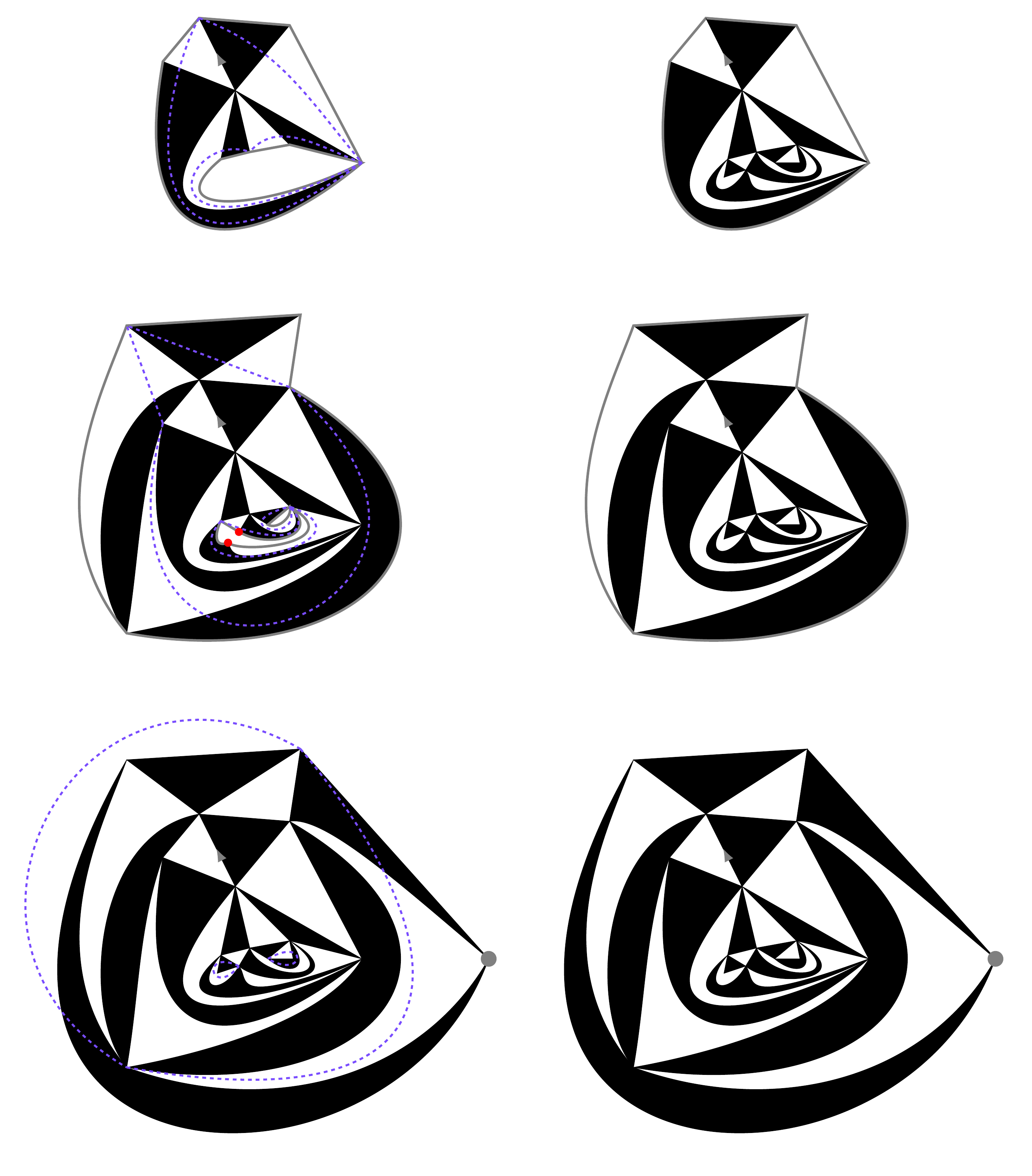}%
\caption{A rooted, pointed Eulerian triangulation (bottom right) and its balls (left) and corresponding hulls (right). The module ``diagonals'' are in dashed purple, and the boundaries of the balls and hulls in solid gray. We can see an example of duplication of vertices in the ball of radius
2, marked in red.}
\label{ModulesEx}
\end{figure}

Thus, $\mathcal{B}_{n}(A)$ is an Eulerian triangulation with simple boundaries\footnote{Note that these boundaries may share a vertex, but not an edge.}, as many as curves in $\mathcal{C}_{n}(A)$. Moreover, the faces of $\mathcal{B}_{n}(A)$ adjacent to these boundaries compose the type-$n$ modules of $A$, so that each part of $\partial\mathcal{B}_{n}(A)$ is \textbf{alternating}, that is, the adjacent faces alternate between black and white.

Let us also formalize the definition of the complement of $\mathcal{B}_{n}(A)$. It is naturally obtained from $A$ by removing the faces and edges of $A$ that are incident to at least a vertex at distance $n-1$ or less from the origin. Note that it is made of as many connected components as there are curves in $\mathcal{C}_{n}(A)$, as it is also the number of boundaries of $\mathcal{B}_{n}(A)$. Consider some $\mathscr{C} \in \mathcal{C}_{n}(A)$, and write $M(A, \mathscr{C})$ for the corresponding connected component of $A \setminus \mathcal{B}_{n}(A)$. $M(A, \mathscr{C})$ is a planar Eulerian triangulation with a boundary, which has the same length as the corresponding one of $\mathcal{B}_{n}(A)$, and is also alternating. However, the boundary of $M(A, \mathscr{C})$ is not necessarily simple. More precisely, a type-$(n+1)$ vertex $v$ of $A$ that sits on the boundary of $M(A, \mathscr{C})$ is attached to the type-$(n \to n+1)$ edges and type-$(n+1)$ faces that are adjacent to $v$ in $A$, as they were excluded from $\mathcal{B}_{n}(A)$, so that $v$ may be a separating vertex in the external face of $M(A, \mathscr{C})$. This is not the case for type-$n$ vertices, as $\mathcal{B}_{n}(A)$ contains all the type-$(n-1 \to n)$ edges and type-$n$ faces of $A$. Thus, the boundary of $M(A, \mathscr{C})$ can have separating vertices, but only on boundary vertices that have, in clockwise order, a black triangle before them and a white one after (as it corresponds to the type-$(n+1)$ vertices of $A$). We call such boundary conditions \textbf{semi-simple}.\\

Let $v$ be a distinguished vertex of $A$ at oriented distance at least $n+2$ from the root. We can now define the \textbf{hull} \boldmath$\mathcal{B}^{\bullet}_{n}(A)$ of \unboldmath$\mathcal{B}_{n}(A)$, as the union of $\mathcal{B}_{n}(A)$ and all the connected components of its complement that do not contain $v$. More precisely, for each curve $\mathscr{C} \in \mathcal{C}_{n}(A)$ that does not separate $v$ from the origin, we glue the boundary of $M(A, \mathscr{C})$ to the corresponding boundary of $\mathcal{B}_{n}(A)$. This operation is well-defined, as the latter is simple, and they both have the same length. The resulting map $\mathcal{B}^{\bullet}_{n}(A)$ has only one boundary, that corresponds to $\mathscr{C}^*$, the unique curve of $\mathcal{C}_{n}(A)$ that separates $v$ from the origin.\\

In the sequel, we will use the notion of \textbf{local distance} between rooted maps. Let $\mathcal{M}$ be the set of finite rooted maps, for  $m, m' \in \mathcal{M}$, we define the \textbf{local distance} between $m$ and $m'$ as
\[
d_{loc}(m,m')=\frac{1}{1+\sup\{R \geq 1 \lvert \mathcal{B}^{d}_{R}(m)=\mathcal{B}^{d}_{R}(m')\}},
\]
where $\mathcal{B}^{d}_{R}(m)$ is defined similarly as before, replacing $\vec{d}$ by the usual graph distance. It is clearly a distance on $\mathcal{M}$, and the completion $(\overline{\mathcal{M}},d_{loc})$ of the space  $(\mathcal{M},d_{loc})$ is a Polish space. The notion of convergence in this space will be called \textbf{local limit}. The elements of $\overline{\mathcal{M}} \setminus \mathcal{M}$ are thus \textbf{infinite} maps that can be defined as the local limit of finite rooted maps.

Note that, from \eqref{eq bound oriented dist by usual dist}, if $A_n$ is a sequence of rooted Eulerian triangulations (possibly with a boundary), and $A$ a rooted planar map, the property that all oriented balls of $A_n$ converge to those of $A$, as $n$ tends to infinity, is equivalent to the same property for non-oriented balls, which is precisely the definition of the convergence of $A_n$ to $A$ in the sense of local limits of rooted planar maps.

As the topology induced by the local distance is what will really matter in the sequel, rather than the actual value of the local distance between two maps, we can forget the general definition of the local distance, and just compare the oriented balls of Eulerian triangulations.

\section{Convergence to the Brownian map}
\label{sec cv to b map}
We will now state and give the proof of the main result of this paper.\\

Before doing so, let us recall the construction of the Brownian map, and introduce some notation. As in \Cref{subsec cv-of-trees}, we write $\mathbbm{e}$ for a standard Brownian excursion, and $Z$ for the ``head'' of the Brownian snake driven by $\mathbbm{e}$, \emph{i.e.}, conditionally on $\mathbbm{e}$, $Z$ a continuous, centered Gaussian process on $[0,1]$ with covariance
\[
Cov(Z_s,Z_t)=\check{\mathbbm{e}}_{s,t}, \, s,t \in [0,1].
\]

The Brownian excursion $\mathbbm{e}$ encodes the \textbf{Continuum Random Tree} $(\mathscr{T}_{\mathbbm{e}},d_{\mathbbm{e}})$, defined by:
\begin{align*}
d_{\mathbbm{e}}(s,t)&=\mathbbm{e}(s) + \mathbbm{e}(t) - 2\check{\mathbbm{e}}_{s,t}  \\
\mathscr{T}_{\mathbbm{e}}&=[0,1]/\{d_{\mathbbm{e}}=0\}.
\end{align*}
The function $d_{\mathbbm{e}}$, which is a pseudo-distance on $[0,1]$, induces a true distance on $\mathscr{T}_{\mathbbm{e}}$ via the canonical projection $p_{\mathbbm{e}}: [0,1] \to \mathscr{T}_{\mathbbm{e}}$, to $\mathscr{T}_{\mathbbm{e}}$.

Almost surely, there is a unique $s \in [0,1]$ such that $Z_s=\inf Z$ {\cite{legall-weill}}. We then denote this point by $s_*$, and $x_*=p_{\mathbbm{e}}(s_*)$ its projection on $\mathscr{T}_{\mathbbm{e}}$.\\

We define, for $s \leq t \in [0,1]$, 
\[
D^{\circ}(s,t)=D^{\circ}(t,s):=Z_s+Z_t-2\max(\min_{r \in [s,t]} Z_r, \min_{r \in [t,1] \cap [0,s]}Z_r).
\]

This function does not satisfy the triangle inequality, which leads us to introduce
\[
D^*(s,t):=\inf\left\{\sum_{i=1}^kD^{\circ}(s_i,t_i) \, \Bigg\lvert \, k\geq 1, \, s_1=s, t_k=t, \, d_{\mathbbm{e}}(t_i,s_{i+1})=0 \ \ \forall \, i \in \{1, 2, \dots, k \}\right\}.
\]
We can now define the \textbf{Brownian map}, by setting $\mathbf{m}_{\infty}=[0,1]/\{D^*=0\}$, and equipping this space with the distance induced by $D^*$, which we still denote by $D^*$.

Let $\mathcal{T}_n$ be a uniform random rooted Eulerian planar triangulation with $n$ black faces, equipped with its usual graph distance $d_n$, and its oriented pseudo-distance $\vec{d}_n$. Let $\overline{\mathcal{T}}_n$ be the triangulation  $\mathcal{T}_n$ together with a distinguished vertex $o_n$ picked uniformly at random. Recall from \Cref{subsec bij} that $\overline{\mathcal{T}}_n$ is the image, by the BDG bijection, of a random labeled tree $\mathscr{T}_n$, uniformly distributed over the set of well-labeled rooted plane trees with $n$ edges. We denote by $l_n$ the labels of the vertices of $\mathscr{T}_n$, and enumerate as in \Cref{subsec bij} the vertices (or rather, the corners) of $\mathscr{T}_n$, by setting $u_i^{(n)}$ to be the $i$-th vertex visited by the contour process of $\mathscr{T}_n$, for $0 \leq i \leq 2n$. As before, we denote by $L_{(n)}$ the rescaled labels of the vertices of $\mathscr{T}_n$.

We define the symmetrization $\overleftrightarrow{d}_{\! \!n}$ of $\vec{d}_n$, by
\[
\overleftrightarrow{d}_{\! \!n}(u,v)=\frac{\vec{d}_n(u,v)+\vec{d}_n(v,u)}{2}.
\] 

We also define a rescaled oriented distance $\vec{D}_{(n)}$ on $[0,1]^2$, by first setting, for $i, j \in \{0,1, \dots, 2n\}$:
\[
\vec{D}_{(n)}\left(\frac{i}{n},\frac{j}{n}\right)=\frac{\vec{d}_{n}(u_i^{(n)},u_j^{(n)})}{n^{1/4}},
\]
then linearly interpolating to extend $\vec{D}_{(n)}$ to $[0,1]^2$.

We define similarly $D_{(n)}$ from $d_n$, as well as $\overleftrightarrow{D}_{(n)}$ from $\overleftrightarrow{d}_n$.

\begin{thm}
\label{thm-cv-to-b-map}
Let $(\mathbf{m}_{\infty},D^*)$ be the Brownian map. There exists some constant  $\mathbf{c}_0 \in [2/3,1]$, such that the following convergence in distribution holds:
\[
\left(C_{(n)},L_{(n)},\vec{D}_{(n)},\overleftrightarrow{D}_{(n)},D_{(n)}\right) \xrightarrow[n \to \infty]{(d)} \left(\mathbbm{e},Z,D^*,D^*,\mathbf{c}_0D^* \right).
\]

Consequently, we have the following joint convergences 
\begin{align*}
n^{-1/4}\cdot(V(\mathcal{T}_n),\overleftrightarrow{d}_{\! \!n}) &\xrightarrow[n \to \infty]{(d)} \phantom{\mathbf{c}_0 }\ \   (\mathbf{m}_{\infty},D^*) \\
n^{-1/4}\cdot(V(\mathcal{T}_n),d_n) &\xrightarrow[n \to \infty]{(d)} \mathbf{c}_0 \cdot (\mathbf{m}_{\infty},D^*), 
\end{align*}
for the Gromov-Hausdorff distance on the space of isometry classes of compact metric spaces.
\end{thm}

Note that we would like to have a statement similar to the one on $\overleftrightarrow{d}_{\! \!n}$ for $\vec{d}_n$. However, as $\vec{d}_n$ is not a proper distance, it does not induce a metric space structure on $V(\mathcal{T}_n)$. Thus, we would need to generalize the Gromov-Hausdorff topology to spaces equipped with a non-symmetric pseudo-distance, to be able to write such a statement.

\begin{proof}
We admit here \Cref{thm-total-asympt-prop-in-finite-trig}, that will be proven later in the paper: for every $\varepsilon > 0$, we have
\begin{equation}
\label{eq asympt prop in proof cv b map}
\Prob{\sup_{x,y \, \in V(\mathcal{T}_n)} \lvert d_n(x,y) - \mathbf{c}_0\vec{d}_n(x,y) \rvert > \varepsilon n^{1/4}} \xrightarrow[n \to \infty]{} 0.
\end{equation}

We proceed similarly to the case of usual triangulations in {\cite{legall}}.

For this whole proof, we work with the pointed triangulation $\overline{\mathcal{T}}_n$, but, as all Eulerian triangulations with $n$ black faces have the same number of vertices, this does not introduce any bias for the underlying, non-pointed triangulation, so that the final statement also holds for $\mathcal{T}_n$.

We have, from \Cref{prop dists from labels}, for any $0 \leq i < j \leq 2n$:
\begin{equation}
\label{eq rescaled distances bounded by labels with 2}
\vec{D}_{(n)}\left(\frac{i}{n}, \frac{j}{n}\right) \leq \frac{2}{n^{1/4}}  \left(l_n(u_i^{(n)}) + l_n(u_j^{(n)}) - 2\max(\min_{k \in \{i, \dots, j\}}l_n(u_k^{(n)}),\min_{k \in \{j, \dots, 2n\}\cup \{0, \dots, i\}}l_n(u_k^{(n)})) + 2 \right).
\end{equation}
As noted before, if we did not have the global multiplicative factor of 2 in \eqref{eq rescaled distances bounded by labels with 2}, we could then proceed as for usual triangulations and other well-known families of planar maps. Thus, the rest of this proof will consist in proving that \Cref{thm-total-asympt-prop-in-finite-trig} makes it possible to ``get rid'' of this cumbersome factor.

We claim that the sequence of the rescaled distances $(\vec{D}_{(n)}(s,t))_{s,t \in [0,1]}$ is tight. First note that, for any $s,s',t,t' \in [0,1]$, we have
\begin{equation}
\label{eq triangle ineq for tightness}
\lvert \vec{D}_{(n)}(s,t) - \vec{D}_{(n)}(s',t') \rvert \leq 2 \left(\vec{D}_{(n)}(s,s') +\vec{D}_{(n)}(t',t)\right).
\end{equation}
Indeed, $\vec{D}_n$, like $\vec{d}_n$, satisfies the triangle inequality, so that
\begin{align*}
 \vec{D}_{(n)}(s,t) - \vec{D}_{(n)}(s',t') &\leq \vec{D}_{(n)}(s,s') +\vec{D}_{(n)}(t',t)\\
 \vec{D}_{(n)}(s',t') - \vec{D}_{(n)}(s,t) &\leq \vec{D}_{(n)}(s',s) +\vec{D}_{(n)}(t,t'),
\end{align*}
which gives \eqref{eq triangle ineq for tightness} when taking into account that, while $\vec{D}_{(n)}$ is not symmetric, we do have, for any $s,t \in [0,1]$, $\vec{D}_{(n)}(s,t) \leq 2 \vec{D}_{(n)}(t,s)$. Then, using \eqref{eq rescaled distances bounded by labels with 2} and \Cref{thm-cv-of-labeled-tree}, we get that, if $\lvert s -s'\rvert \vee \lvert t-t'\rvert \leq \eta$, then, for $n$ large enough, the right-hand side of \eqref{eq triangle ineq for tightness} is smaller than $4 \omega(Z,\eta) + \varepsilon$, where we denote by $\omega(Z,\eta)$ the supremum $\sup_{\lvert I \rvert \leq \eta}\omega(Z,I)$, and $\omega(f,I)$ is the modulus of continuity of $f$ on the interval $I$.

Thus, along a subsequence, we have the joint convergence:
\begin{equation}
\label{eq joint convergence along subsequence}
\left(C_{(n)},L_{(n)},\vec{D}_{(n)}\right) \xrightarrow[n\to \infty]{(d)} (\mathbbm{e},Z,D),
\end{equation}
for some random continuous process $D$ on $[0,1]^2$. In the rest of this proof, we fix a subsequence so that \eqref{eq joint convergence along subsequence} holds, and work along this subsequence.

Note that, from \Cref{thm-total-asympt-prop-in-finite-trig}, we also have the joint convergence of $D_{(n)}$ to $\mathbf{c}_0 D$. This already implies that $D$ is symmetric, and thus is a pseudo-metric. We now want to show that $D=D^*$ a.s., which will conclude the proof, since this will imply the uniqueness of the limit $D$.

First, it is straightforward to get from \eqref{eq:distance-from-root-from-labels} and \eqref{eq joint convergence along subsequence} that, for any $s \in [0,1]$:
\begin{equation}
\label{eq:extracted-lim-of-distance-from-root-equals-brownian-distance}
D(s_*,s)=Z_s - \inf Z.
\end{equation}

We will now show that a.s., for every $s,t \in [0,1]$,
\begin{equation}
\label{eq:extracted-lim-of-distance-smaller-than-brownian-pseudodist}
D(s,t)\leq D^{\circ}(s,t).
\end{equation}
 To prove this claim, let us get back to $\overline{\mathcal{T}}_n$ and $\mathscr{T}_n$. From \eqref{eq asympt prop in proof cv b map}, for any $\varepsilon>0$ and any $\delta \in (0,1)$, for any $n$ large enough, the event
\[
\Big\lvert d_{n}\left(u,v\right) - \mathbf{c}_0\vec{d}_{n}\left(v,u\right) \Big\rvert \leq \varepsilon n^{1/4} \ \ \ \forall u,v \in V(\overline{\mathcal{T}}_n)
\]
holds with probability at least $1 - \delta$.

On that event, we have, for any $u,v,w \in V(\overline{\mathcal{T}}_n)$,
\begin{align}
\label{eq:oriented-dist-almost-trig-ineq}
\vec{d}_n(u,v) &\leq \vec{d}_n(u,w) + \vec{d}_n(w,v) \leq \frac{1}{\mathbf{c}_0}d_n(u,w)+\vec{d}_n(w,v) +\varepsilon n^{1/4} \nonumber\\
&\leq \vec{d}_n(w,u) + \vec{d}_n(w,v)  +2\varepsilon n^{1/4}.
\end{align}

Thus, going back to the proof of \Cref{prop dists from labels}, when estimating oriented distances from the length of the concatenation of two predecessor geodesics, rather than having to multiply this length by 2, we just need to add $2\varepsilon n^{1/4}$.

Therefore, for any $0 \leq i <j \leq 2n$,
\begin{align*}
&\vec{D}_{(n)}\left(\frac{i}{n}, \frac{j}{n}\right) \\
\leq &L_{(n)}\left(\frac{i}{n}\right) + L_{(n)}\left(\frac{j}{n}\right) - 2\max\left(\min_{k \in \{i, \dots, j\}} L_{(n)}\left(\frac{k}{n}\right),\min_{k \in \{j, \dots, 2n\}\cup \{0, \dots, i\}} L_{(n)}\left(\frac{k}{n}\right)\right) + \frac{2}{n^{1/4}} +2\varepsilon.
\end{align*}

Thus, letting $n \to \infty$ (along our subsequence), for any $\varepsilon >0$, we have $D \leq D^{\circ} + 2\varepsilon$ a.s., so that we get the desired inequality \eqref{eq:extracted-lim-of-distance-smaller-than-brownian-pseudodist}.

Moreover, as $D$ satisfies the triangle inequality, we have
\begin{equation}
\label{eq:extracted-lim-of-distance-smaller-than-brownian-true-dist}
D(s,t)\leq D^{*}(s,t) \ \ \ \ \ \ \ \forall s, t \in [0,1] \text{ a.s.}
\end{equation}

To replace this inequality by an equality, it now suffices to show that, for $U,V$ chosen uniformly and independently at random in $[0,1]$, and independently from the rest, $D(U,V)\overset{(d)}{=}D^*(U,V)$. Indeed, this would imply $D=D^*$ a.e., and thus $D=D^*$ since both are continuous. To prove this, from \eqref{eq:extracted-lim-of-distance-from-root-equals-brownian-distance}, it is enough to show that $D(U,V) \overset{(d)}{=} D(s_*,U)$.

To prove this, let us get back to the discrete level for a moment. Let $u_n, v_n$ be two vertices of $\mathcal{T}_n$ chosen independently and uniformly at random. As $\mathcal{T}_n$ re-pointed at $u_n$ has the same law as $\overline{\mathcal{T}}_n$, we have
\begin{equation}
\label{eq rerooting dist}
\vec{d}_n(u_n,v_n)\overset{(d)}{=}\vec{d}_n(o_n,v_n).
\end{equation}

Similarly to the case of usual triangulations in \cite{legall}, this implies the desired equality in distribution $D(U,V) \overset{(d)}{=} D(s_*,U)$. Indeed, set $U_n=\lceil (2n-1)U\rceil$ and $V_n=\lceil (2n-1)V\rceil$, which are both uniformly distributed over $\{1,2, \dots, 2n-1\}$, so that
\[
\frac{U_n}{n} \xrightarrow[n\to \infty]{(P)} U, \ \ \ \frac{V_n}{n} \xrightarrow[n\to \infty]{(P)} V.
\]
Then, from \eqref{eq joint convergence along subsequence}, we have
\[
\vec{D}_{(n)}\left(\frac{U_n}{n},\frac{V_n}{n}\right) \xrightarrow[n\to \infty]{(P)} \vec{D}(U,V).
\]
Now, from \eqref{eq rerooting dist}, we have that the distribution of $\vec{D}(U,V)$ is also the limiting distribution of
\[
L_{(n)}\left(\frac{U_n}{n}\right) - \min L_{(n)} +1,
\]
so that $\vec{D}(U,V)$ has the same distribution as $Z_U -\inf Z$, which is also the distribution of $D(s_*,U)$, from \eqref{eq:extracted-lim-of-distance-from-root-equals-brownian-distance}. This concludes the proof.
\end{proof}

\section{Technical preliminaries}
\label{sec tech prelim}
\subsection{Consequences of the convergence of the rescaled labels}
\label{subsec-consequences-of-cv-of-labels}

We now prove a few technical properties of $\vec{d}$ that stem from the convergence given in \Cref{thm-cv-of-labeled-tree}.\\

For any integer $n\geq 1$, let $\rho_n$ be the root vertex of the random triangulation $\mathcal{T}_n$, uniform over the rooted planar Eulerian triangulations with $n$ black faces. We denote by $\overline{\mathcal{T}}_n$, the triangulation $\mathcal{T}_n$ together with a distinguished vertex $o_n$, picked uniformly at random in $\mathcal{T}_n$. We then have the following result:
\begin{prop}
\label{prop geod cv to sup z}
The following convergence holds:
\[
n^{-1/4} \vec{d}(o_n,\rho_n) \xrightarrow[n\to \infty]{(d)} \sup Z.
\]

Consequently, the sequence $(n^{-1/4} \vec{d}(o_n,\rho_n))_{n \geq 1}$ is bounded in probability and bounded away from zero in probability, as well as the sequence $(n^{-1/4} \vec{d}(\rho_n,o_n))_{n \geq 1}$.
\end{prop}

\begin{proof}
Recall that $\overline{\mathcal{T}}_n$ is in correspondence with a random tree $\mathscr{T}_n$, uniform over the well-labeled plane trees with $n$ edges, whose labelling we denote by $l_n$.
We have, from \eqref{eq:distance-from-root-from-labels}, that
\[
\vec{d}(o_n, \rho_n)= - \min_{v \in V(\mathscr{T}_n)}l(v) +1.
\]

Then, using the convergence of \Cref{thm-cv-of-labeled-tree}, we get that the quantity
\[
n^{-1/4}\left( - \min_{v \in V(\mathscr{T}_n)}l(v) +1\right)
\]
converges in distribution to $(- \inf Z)\overset{(d)}{=} \sup Z$.

This directly implies the bounds in probability for the sequence $(n^{-1/4} \vec{d}(o_n,\rho_n))_{n \geq 1}$. For those pertaining to $(n^{-1/4} \vec{d}(\rho_n,o_n))_{n \geq 1}$, recall that for any two vertices $u,v$ of an Eulerian triangulation, $\frac{1}{2}\vec{d}(u,v)\leq \vec{d}(v,u) \leq 2 \vec{d}(u,v)$.
\end{proof}

For a rooted Eulerian triangulation (possibly with a boundary) $\Delta$, let $N(\Delta)$ be the number of black triangles of $\Delta$. Then:

\begin{prop}
\label{prop-positive-mass-to-any-neigb-of-zero-in-distance-profile}
Let $\alpha>0$, and let us denote by $\mathcal{B}_{r}(\overline{\mathcal{T}}_n, o_n)$ the ball (for the oriented distance) of radius $r$ in $\overline{\mathcal{T}}_n$, centered at $o_n$. For any $\varepsilon \in (0,1)$, there exists some $b \in (0,1)$ such that
\[
\liminf_{n \to \infty}\Prob{N(\mathcal{B}_{\alpha n^{1/4}}(\overline{\mathcal{T}}_n, o_n)) > bn} \geq 1 - \varepsilon.
\]
\end{prop}

\begin{proof}
Let us roughly sketch the idea of the proof. Recall that $\overline{\mathcal{T}}_n$ is in correspondence with a random tree $\mathscr{T}_n$, uniform over the well-labeled plane trees with $n$ edges. Moreover, the black triangles of $\overline{\mathcal{T}}_n$ correspond to the edges of $\mathscr{T}_n$, so that, for any type-$m$ black triangle $t$ of $\overline{\mathcal{T}}_n$:
\[
\text{Leb}\{s \in [0,2n) \Big\lvert u_{\lfloor s \rfloor}^{(n)} \text{is the vertex of type $m$ in $t$}\}=2.  
\]

Thus, using \eqref{eq:distance-from-root-from-labels}, we have
\[
\frac{1}{n} \cdot N(\mathcal{B}_{\alpha n^{1/4}}(\overline{\mathcal{T}}_n,o_n))  \geq \int_0^1 \Indic{l_n(\lfloor 2ns \rfloor) - \min l_n \leq \alpha n^{1/4} - 1}\mathrm{d}s.
\]

Note that we have, for any $s \in [0,1)$:
\[
\lvert l_n(\lfloor 2ns \rfloor) - n^{1/4}L_{(n)}(s) \rvert \leq 1,
\]
so that:
\[
\int_0^1 \Indic{l_n(\lfloor 2ns \rfloor) - \min l_n \leq \alpha n^{1/4} - 1}\mathrm{d}s \geq \int_0^1 \Indic{L_{(n)}(s) - \min L_{(n)} \leq \alpha  - 2/n^{1/4}}\mathrm{d}s.
\]

Therefore:
\[
\Prob{N(\mathcal{B}_{\alpha n^{1/4}}(\overline{\mathcal{T}}_n,o_n)) > bn} \geq
\Prob{ \int_0^1 \Indic{L_{(n)}(s) - \min L_{(n)} \leq \alpha  - 2/n^{1/4}}\mathrm{d}s > b}.
\]

Now, the liminf of the probability on the right-hand side of the previous inequality can be bounded below by
\[
\Prob{\int_{0}^{1}\Indic{Z_s - \inf Z \leq\frac{\alpha}{2}} \mathrm{d}s > b},
\]
which tends to 1 as $b$ tends to 0, as $Z$ is continuous.

This concludes the proof.
\end{proof}

\begin{prop}
\label{prop-sprinkling}
For any $\varepsilon >0$ and any $\delta \in (0,1)$ there exists an integer $k \geq 1$ such that, for any sufficiently large $n$, if $o_n^1, \dots, o_n^k$ are chosen uniformly and independently in $V(\mathcal{T}_n)$, we have
\[
\Prob{\sup_{x \in V(\mathcal{T}_n)}\left(\inf_{1 \leq j \leq k}\vec{d}(x,o_n^j)\right) > \varepsilon n^{1/4}} \leq \delta.
\]
\end{prop}

\begin{proof}
  Let us fix an integer $K \geq 1$. Recall that we write $(u_i^{(n)})_{0 \leq i \leq 2n-1}$ for the vertices of $\mathscr{T}_n$ along its contour exploration. Then, for $k$ large enough, for any sufficiently large $n$,
\begin{equation}
\label{eq:lots-of-unif-vertices-get-into-all-intervals}  
\Prob{\forall i \in \{0, \dots, 2K-1\} \, \exists j \in \{1, \dots, k\} \, \exists m \in \{\lfloor \frac{in}{k} \rfloor, \dots, \lfloor \frac{(i+1)n}{k} \rfloor \}, \, o_n^j=u_m^{(n)}} \geq 1 - \frac{\delta}{2}.
\end{equation}

We will now argue on the event in \eqref{eq:lots-of-unif-vertices-get-into-all-intervals}.

Using \eqref{eq:distance-from-root-from-labels}, we have, for any $i,j \in \{0, 1, \dots, 2n\}$,
\[
\vec{d}_n(u_i^{(n)},u_j^{(n)}) \leq 2(l_n(u_i^{(n)}) + l_n(u_j^{(n)}) - 2\check{l}_n(i,j)+2),
\]
so that, for any $n$ sufficiently large:
\[
\sup_{x \in V(\mathcal{T}_n)}\left(\inf_{1 \leq j \leq k}\vec{d}(x,o_n^j)\right)
\leq 4 \max_{0 \leq i \leq 2K-1} \omega\left(l_n,\left[\lfloor \frac{in}{K} \rfloor, \lfloor \frac{(i+1)n}{K} \rfloor\right]\right) +4
\]
where $\omega(f,I)$ is the modulus of continuity of the function $f$ on the interval $I$.

Therefore, we have
\[
\liminf_{n}\Prob{\sup_{x \in V(\mathcal{T}_n)}\left(\inf_{1 \leq j \leq k}\vec{d}(x,o_n^j)\right) < \varepsilon n^{1/4}} \geq \Prob{\omega(Z,\frac{1}{K}) < \frac{\varepsilon}{5}},
\]
by using once again the convergence of \Cref{thm-cv-of-labeled-tree}. (We denote by $\omega(Z,\eta)$ the supremum $\sup_{\lvert I \rvert \leq \eta}\omega(Z,I)$.)

Now, as $Z$ is a.s. continuous on $[0,1]$, it is uniformly continuous, so that, for any $\varepsilon >0$, for $K$ large enough,
\[
\Prob{\omega(Z,\frac{1}{K}) <\frac{\varepsilon}{5}} \geq 1 - \frac{\delta}{2},
\]

which concludes the proof.
\end{proof}

\subsection{Enumeration results}
\bigskip

We will need some asymptotic results on the generating series $B(t,z)$ of Eulerian triangulations with a semi-simple alternating boundary, as defined in \Cref{subsec structure oriented dist}: 
\begin{equation*}
B(t,z)=\sum_{n,p \geq 0}B_{n,p}t^n z^p,
\end{equation*}
where $B_{n,p}$ is the number of Eulerian triangulations with semi-simple alternating boundary of length $2p$ and with $n$ black triangles. Jérémie Bouttier and the author obtain in {\cite{bouttier-carrance}} a rational parametrization of $B(t,z)$, which yields the following asymptotic result:

\begin{thm}
We have
\begin{equation}
\label{eq:gf-better-asympt}
[t^n]B(t,z) =\sum_{p \geq 0} B_{n,p}z^p \underset{n \to \infty}{\sim} \frac{3}{2}\frac{z}{\sqrt{\pi(z-1)(4z-1)^3}}8^n n^{-5/2} \ \ \ \forall \,z \in[0, \frac{1}{4}).
\end{equation}
\end{thm}

This implies that:
\begin{equation}
\label{gf-expression}
\begin{cases}
B_{n,p}\underset{n \to \infty}{\sim} C(p)8^n n^{-5/2} \ \ \forall \, p\\
C(p)\underset{p \to \infty}{\sim} \frac{\sqrt{3}}{2 \pi}4^p \sqrt{p} \, \text{ and } \sum_{p\geq 1}C(p)z^p=\frac{3}{2}\frac{z}{\sqrt{\pi(z-1)(4z-1)^3}} \ \ \ \forall \,z \in[0, \frac{1}{4})
\end{cases}
\end{equation}

Note that \eqref{eq:gf-better-asympt} is much stronger than \eqref{gf-expression}. Indeed, it states that, for any $\varepsilon >0$, for any $n$ large enough, we have, for all $z \in[0,1/4)$,
\begin{equation}
\label{eq gf bound epsilon}
(1-\varepsilon)8^nn^{-5/2} f(z) \leq g_n(z) \leq (1+\varepsilon)8^nn^{-5/2} f(z),
\end{equation}
where
\[
f(z)=\frac{3}{2}\frac{z}{\sqrt{\pi(z-1)(4z-1)^3}}
\]
and 
\[
g_n(z)=\sum_{p \geq 0} B_{n,p}z^p.
\]
Thus, as both $f$ and $g_n$ are analytic functions on $[0,1/4)$, by taking the successive derivatives of the terms in \eqref{eq gf bound epsilon}, we obtain equivalent bounds for the successive coefficients of $f$ and $g_n$ seen as power series:
\[
(1-\varepsilon)8^nn^{-5/2} ([z^p]f(z)) \leq B_{n,p} \leq (1+\varepsilon)8^nn^{-5/2} ([z^p]f(z)),
\]
for any $n$ large enough and for any $p$.

This yields that, for all $n,p \geq 1$,
\begin{equation}
\label{eq:gf-asympt-bounds}
c \, C(p) 8^n n^{-5/2} \leq B_{n,p} \leq c' \, C(p) 8^n n^{-5/2},
\end{equation}
for some constants $0 < c < c'$ independent of $n$ and $p$.\\

Let us now focus on the coefficients
\begin{equation}
\label{z-coeff-definition}
Z(p):=\sum_{n \geq 0} \left(\frac{1}{8} \right)^nB_{n,p}.
\end{equation}
The calculations of {\cite{bouttier-carrance}} also yield the following exact formula:
\begin{equation}
\label{eq:boltzmann-exact}
\sum_{p \geq 0}Z(p)z^p =\sum_{p \geq 0, \, n\geq 0}\left(\frac{1}{8}\right)^n B_{n,p}z^p=\frac{1+7z-8z^2+\sqrt{(z-1)(4z-1)^3}}{2(1-z)}  \ \ \ \forall \, z \in[0, \frac{1}{4}),
\end{equation}
which gives the asymptotic behavior:
\begin{equation}
\label{eq:boltzmann-asympt}
Z(p)\underset{p \to \infty}{\sim}\frac{1}{4}\sqrt{\frac{3}{\pi}}4^p p^{-5/2} \, \text{ and } Z(0)=1.
\end{equation}

In particular, for any $p \geq 1$, the sum $Z(p)=\sum_nB_{n,p}8^{-n}$ is finite, which makes it possible to define the \textbf{Boltzmann distribution} on Eulerian triangulations of the $2p$-gon (with a semi-simple alternating boundary), that assigns a weight $8^{-n}/Z(p)$ to each such triangulation having $n$ black triangles. A random triangulation sampled according to this measure will be called a \textbf{Boltzmann Eulerian triangulation of perimeter} \boldmath$2p$.\\

Note that there is a natural bijection between Eulerian triangulations of the 2-gon (with an alternating boundary), and rooted planar Eulerian triangulations, which simply consists in ``zipping'' or ``unzipping'' the root edge (see \Cref{DigonToRoot}). This simple observation will be useful in the sequel.

\begin{figure}[htp]
\centering
\includegraphics[scale=1.3]{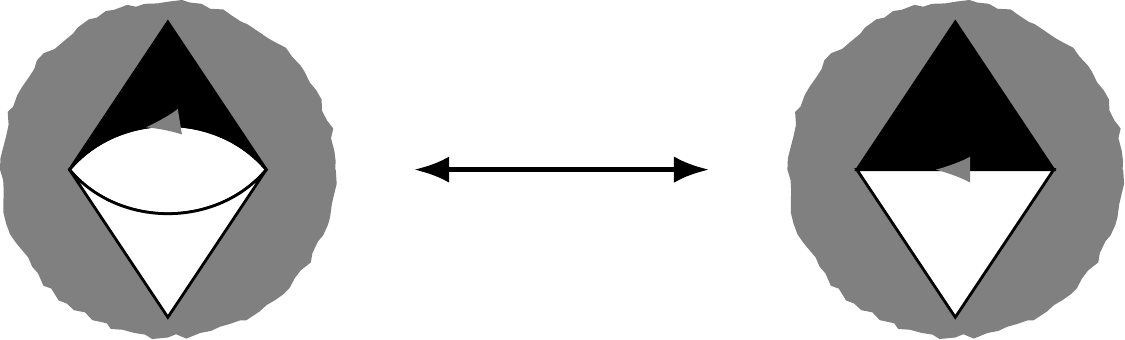}
\caption{The bijection between Eulerian triangulations of the 2-gon with an alternating boundary, and rooted planar Eulerian triangulations.}
\label{DigonToRoot}
\end{figure}

\section{Skeleton decomposition}
\label{sect-skeleton}

We previously considered the balls of planar rooted Eulerian triangulations, and the associated hulls, defined with the oriented distance from the root.\\

To obtain the layer decomposition of finite planar Eulerian triangulations that will be crucial to the rest of this paper, we want to generalize the notions of hulls, to the \emph{layers} of a planar Eulerian triangulation \unboldmath$A$, that lie between the boundaries of two different hulls of $A$ (see \Cref{DiagonalCurvesInSubtrig}). More precisely, we want to have a good definition of oriented distance from the boundary of a rooted planar Eulerian triangulation with an alternating boundary (or, as will be the case for our layers, two disjoint boundaries with one distinguished as the ``bottom'' one): such an oriented distance will once again induce a structure of sets of simple closed curves $\{\mathcal{C}_n\}$, going through type-$n$ modules, and, if we look at different layers making up some triangulation $A$, we want the union of these sets of curves to be $\{\mathcal{C}_n(A)\}$ (see \Cref{DiagonalCurvesInSubtrig}). For that purpose, if $A$ is a planar Eulerian triangulation with a (distinguished) alternating boundary $\partial_0 A$, rooted on this boundary, we define on $V(A)$, the \textbf{oriented distance from the boundary} $\partial_0 A$ as the pullback of the oriented distance on $A'$, where $A'$ is the planar Eulerian triangulation obtained from $A$ by gluing into pairs the edges of $\partial_0A$, as shown in \Cref{BoundaryDistanceConv}. Thus, this distance alernates between 0 and 1 on $\partial_0 A$, and, for an inner vertex $v$ of $A$, is the shortest oriented distance from a 0-labeled outer vertex, to $v$ (see \Cref{BoundaryDistanceConv}).

\begin{figure}[htp]
\centering
\includegraphics[scale=1.2]{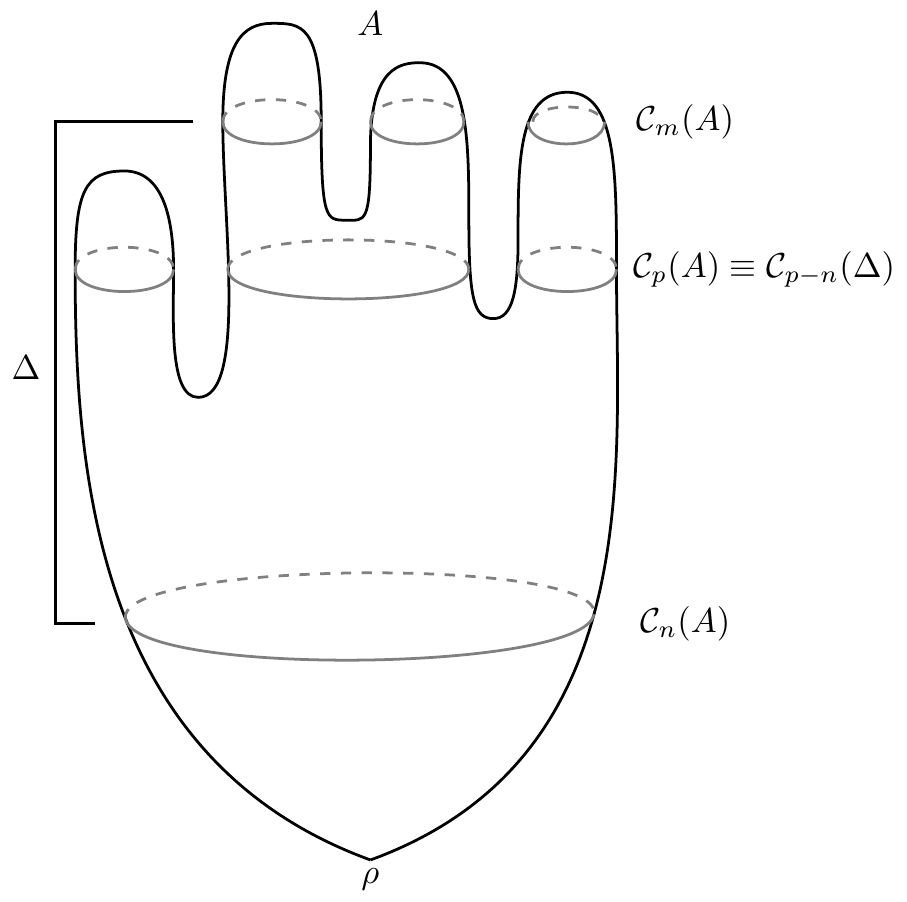}
\caption{The sets of simple closed curves $\{\mathcal{C}_n(A)\}$, cutting up the planar Eulerian triangulation $A$ into layers of increasing distance from the origin $\rho$, also separate the different layers of a subtriangulation $\Delta$ lying between two such curves: this guides us for the good notion of oriented distance from the bottom boundary of $\Delta$.}
\label{DiagonalCurvesInSubtrig}
\end{figure}

\begin{figure}[htp]
\centering
\includegraphics[scale=0.9]{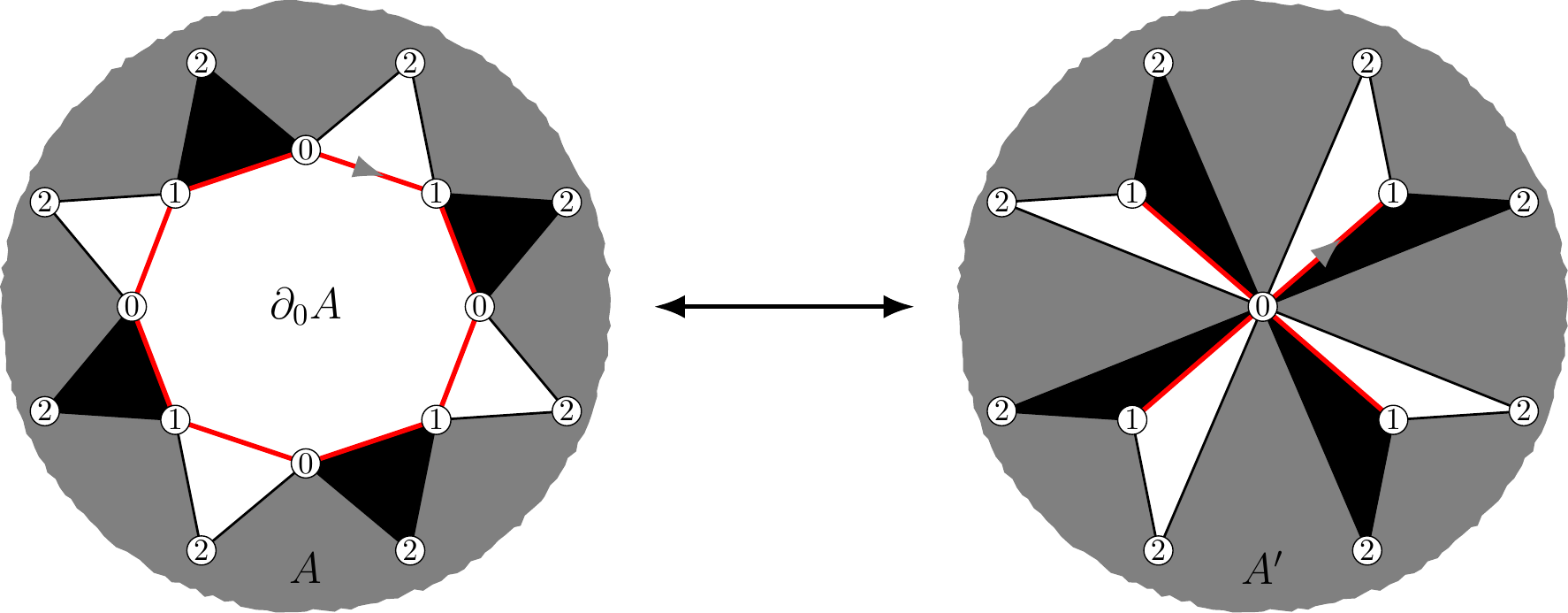}
\caption{The geodesic oriented distances in an Eulerian triangulation $A$ with a (distinguished) alternating boundary $\partial_0A$ (left) are given by the distances from the root in the triangulation $A'$, obtained from $A$ by gluing into pairs the edges of $\partial_0A$ (right). The triangles adjacent to the boundary are depicted in black and white, the rest of triangulation is sketched in gray, and the boundary is in red.}
\label{BoundaryDistanceConv}
\end{figure}

Now that we have a satisfactory notion of distance, as before, we will be interested in the union of faces of $A$ incident to vertices at (oriented) distance less than $n$ from the boundary. We will denote this union $B_{n+1}(A)$. As was the case for usual balls, the faces of $B_{n+1}(A)$ adjacent to its boundary parts, other than the original boundary $\partial A$, will correspond to modules of type $n+1$, for the oriented distance from $\partial_0 A$. Once again, we will have the convention that these boundary parts are simple, and we will glue them to semi-simple boundaries. If $A$ is pointed at a vertex $v$ at oriented distance at least $n+2$ from the boundary, we can also define a notion of hull for $B_{n+1}(A)$, which will be an Eulerian triangulation with two boundaries of specific types. In this section, we will first develop the description of such triangulations, before dealing with random Eulerian triangulations with one boundary, and their hulls.

Note that the chain of arguments and notation of this section follow closely those of {\cite[Section~5]{curien-legall}}: in order to be both concise and precise, we detail in the proofs of this section only the additional subtleties and difficulties arising in our case compared to the equivalent results of {\cite{curien-legall}}.

\subsection{Cylinder triangulations}
\label{subsec cylinder trig}

\begin{defnt}
We call \textbf{Eulerian cylinder triangulation of height \boldmath$r$} \unboldmath $\geq 1$, an Eulerian triangulation with two  boundaries, one (the bottom of the cylinder) being alternating and semi-simple, the other one (the top) being a succession of modules (see \Cref{cylinder-fig}), and such that any module adjacent to the top boundary is of distance type $r$ with respect to the bottom. 

We denote by $\partial \Delta$ its bottom boundary, and by $\partial ^*\Delta$ its top boundary. The root is an edge on $\partial \Delta$ oriented such that the bottom face sits on its right.
\end{defnt}

\begin{figure}[htp]
\centering
\includegraphics[scale=0.8]{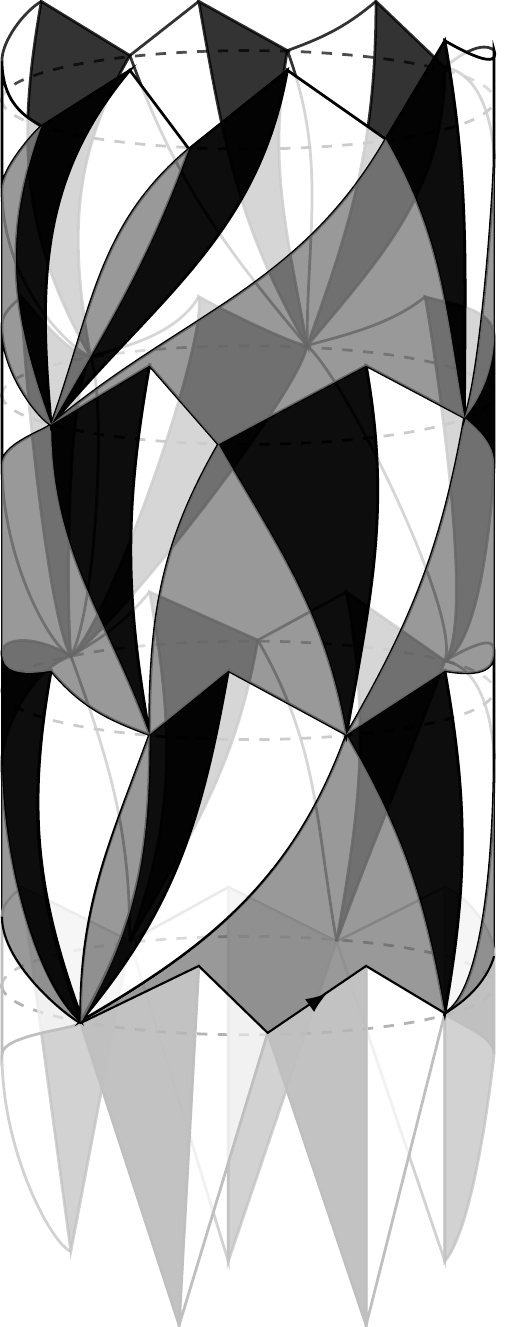}
\hspace{6em}
\includegraphics[scale=0.8]{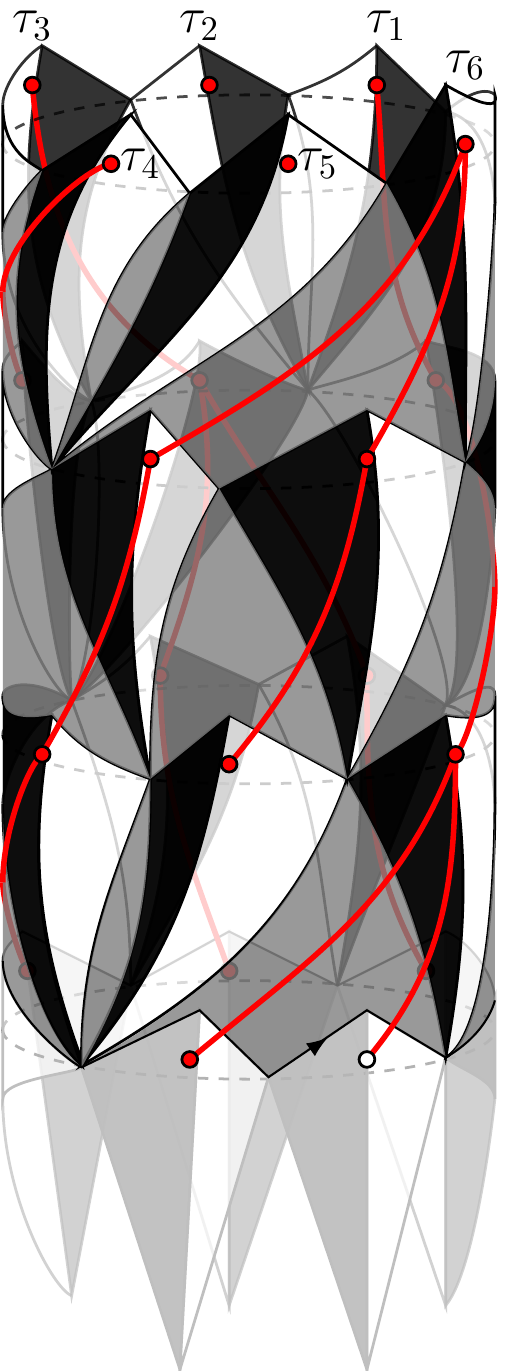}
\caption{Left, a cylinder Eulerian triangulation of height 3, top length 12 and bottom length 10: the foreground parts of the slots are in medium grey while the background ones are left white for legibility, and the ghost modules are in pale grey at the bottom. Right, the construction of the associated forest (with its distinguished vertex at height 3, in white).}
\label{cylinder-fig}
\end{figure}

Let $\Delta$ be an Eulerian cylinder triangulation of height $r$. Let $2p$ be the bottom boundary length, and $2q$ the top boundary length. For $1 \leq j \leq r$, the ball $B_j(\Delta)$ is defined as the union of all edges and faces of $\Delta$ incident to at least a vertex at distance $< j$ from the bottom boundary, and the hull $B^{\bullet}_j(\Delta)$ is obtained from $B_j(\Delta)$ by adding all the connected components of its complement except the one containing the top boundary. Therefore $B^{\bullet}_j(\Delta)$ is a cylinder triangulation of height $j$, and we denote by $\partial_j\Delta$ the set of modules adjacent to its top boundary. Let $\mathcal{M}(\Delta)$ be the set of modules of $\Delta$ belonging to some $\partial_j(\Delta)$, for $0 \leq j \leq r$. (For convenience, we will associate a ``ghost'' module to each pair of successive edges of the bottom boundary respectively adjacent to a white and a black triangle, and the set $\partial_0(\Delta)$ of these $p$ ghost modules will be included in $\mathcal{M}(\Delta)$.)

We define a genealogical order on $\mathcal{M}(\Delta)$: a module $m$ of $\partial_j(\Delta)$ is the parent of a module $m'$ of $\partial_{j-1}(\Delta)$ if $m$ is the first module of $\partial_j(\Delta)$ that we encounter when going left-to-right along the modules of $\partial_{j-1}(\Delta)$, starting by the top vertex of $m'$.
This order yields a forest $\mathcal{F}$ of $q$ plane trees, whose vertices correspond to the modules belonging to $\mathcal{M}(\Delta)$. The maximal height of this forest is $r$, and a vertex of height $r-j$ corresponds to a module of $\partial_{j}(\Delta)$. We denote by $\tau_1, \tau_2, \dots, \tau_q$ the trees of the forest listed clockwise around $\partial_r(\Delta)$, with $\tau_1$ the tree containing the vertex corresponding to the root.  \! \!Therefore, the tree $\tau_1$ has height $r$, with a distinguished vertex (the one corresponding to the root) at height $r$.

Apart from the modules of $\mathcal{M}(\Delta)$, $\Delta$ is composed of triangulations with a semi-simple alternating boundary that fill in the ``slots'' bounded by the modules of $\mathcal{M}(\Delta)$. To a module $m$ in $\partial_j(\Delta)$, we associate the slot bounded by $m$, its  children if any, and the module to the left of $m$ in $\partial_j(\Delta)$. This slot $M_m$ is thus filled in by a triangulation with a semi-simple alternating boundary, of perimeter $2(c_m+1)$, where $c_m$ is the number of children of $m$. We denote by $N(M_m)$ the number of black triangles of this triangulation with a boundary.\\

We will say a forest $\mathcal{F}$ with a distinguished vertex is a \boldmath$(p,q,r)$\textbf{-admissible forest} if it consists of an ordered sequence \unboldmath$(\tau_1, \tau_2, \dots, \tau_q)$ of $q$ rooted plane trees of maximal height $r$, with $p$ vertices at height $r$, with the distinguished vertex at height $r$ in $\tau_1$.

If $\mathcal{F}$  is a $(p,q,r)$-admissible forest, we write $\mathcal{F}^*$ for the set of all vertices of $\mathcal{F}$ at height strictly smaller than $r$.

From the preceding decomposition, we obtain the following result:

\begin{prop}
\label{prop:bij-skeleton-dec}
The Eulerian triangulations of the cylinder $\Delta$ of height $r$ with a bottom boundary length $2p$ and a top boundary length $2q$, are in bijection with pairs consisting of a $(p,q,r)$-admissible forest $\mathcal{F}$ and a collection $(M_v)_{v \in \mathcal{F}^*}$ such that, for every $v \in \mathcal{F}^*$, $M_v$ is an Eulerian triangulation of the $2(c_v+1)$-gon with a semi-simple alternating boundary, with $c_v$ being the number of children of $v$ in $\mathcal{F}$. 
\end{prop}

Note that the bijection of \Cref{prop:bij-skeleton-dec} is an adaptation of similar constructions that have been made for usual triangulations and quadrangulations, starting with Krikun's works {\cite{Krikun2005,uipq}}, with more recent versions by Curien and Le Gall for usual triangulations {\cite{curien-legall}}, and by Le Gall and Lehéricy for quadrangulations {\cite{legall-lehericy}}. Following the vocabulary used in these works, we call this bijection the \textbf{skeleton decomposition}, and say that $\mathcal{F}$ is the \textbf{skeleton} of the triangulation $\Delta$. We will also call \textbf{skeleton modules} the modules of $\mathcal{M}(\Delta)$.

\subsection{Skeleton decomposition of random triangulations}

We will now use the bijection derived in \Cref{subsec cylinder trig} to obtain the asymptotic behavior of the laws of the hulls of random uniform Eulerian triangulations with a boundary.\\

We first need a bit of additional notation.

Consider an Eulerian triangulation with a boundary $\Delta$, pointed at $v$. We can define the hull $B_r^{\bullet}(\Delta)$ of $\Delta$ like for cylinder triangulations, if $\vec{d}(\partial \Delta,v) > r +1$. If $\vec{d}(\partial \Delta,v) \leq r +1$, we can set $B_r^{\bullet}(\Delta)=\Delta$.

Let  $\mathcal{T}^{(p)}_n$ be a uniform random triangulation over the set of Eulerian triangulations with a semi-simple alternating boundary of length $2p$ and with $n$ black triangles. We denote by $\overline{\mathcal{T}}^{(p)}_n$ the pointed triangulation obtained by choosing a uniform random inner vertex of $\mathcal{T}^{(p)}_n$. Let $\Delta$ be a cylinder triangulation of height $r$, of respective bottom and top boundary lengths $2p$ and $2q$, with $N$ black triangles, with $n \geq N$. Using the skeleton decomposition, we associate to $\Delta$ a $(p,q,r)$-admissible forest $\mathcal{F}$, together with triangulations $(M_v)_{v \in \mathcal{F}}$ filling in the ``slots'' between the modules of $\mathcal{M}(\Delta)$. We write $N(M_v)$ for the number of black triangles of $M_v$, for every $v \in \mathcal{F}^*$.

\begin{lemma}
\label{skeleton-proba-limit-lem}
We have 
\begin{equation}
\label{skeleton-proba-limit}
\lim_{n \to \infty}\mathbb{P}\left(B_r^{\bullet}(\overline{\mathcal{T}}^{(p)}_n)=\Delta \right)=\frac{4^{-q}C(q)}{4^{-p}C(p)}\prod_{v \in \mathcal{F}^*}\theta(c_v)\frac{8^{-N(M_v)}}{Z(c_v+1)},
\end{equation}
where
\begin{equation}
\label{eq:theta-related-to-z}
\theta(k)=\frac{1}{8}4^{-k+1}Z(k+1),
\end{equation}
with $Z(k)$ defined as in \eqref{z-coeff-definition}.
\end{lemma}

\begin{proof}
  First note that this result is the equivalent of {\cite[Lemma~2]{curien-legall}}. It is obtained very similarly, though in our case we start from a slightly less explicit expression, as shown in \eqref{eq:probas-for-pointed-cylinder}: this stems from the fact that our triangulations do not necessarily have simple boundaries.\\
  To simplify notation, let us note in this proof $\rho=8$ and $\alpha=4$. The property $B_r^{\bullet}(\overline{\mathcal{T}}^{(p)}_n)=\Delta$ holds if and only if $\mathcal{T}^{(p)}_n$ is obtained from $\Delta$ by gluing to the top boundary\footnote{Note that, as $\Delta$ is rooted, we can fix an arbitrary rule to determine where to glue the root of the other triangulation.} an arbitrary triangulation with a semi-simple alternating boundary of length $2q$, and with $n-N$ black triangles, and if the distinguished vertex is chosen among the inner vertices of the glued triangulation. Thus:
\begin{equation}
\label{eq:probas-for-pointed-cylinder}
\mathbb{P}\left(B_r^{\bullet}(\overline{\mathcal{T}}^{(p)}_n)=\Delta \right) = \frac{B_{n-N,q}}{B_{n,p}}\cdot \frac{\#\text{inner vertices in glued triangulation}}{\#\text{inner vertices in total triangulation}}.
\end{equation}
Therefore:
\begin{equation}
\label{eq:limit-probas-for-tall-pointed-cylinder}
 \lim_{n \to \infty}\mathbb{P}\left(B_r^{\bullet}(\overline{\mathcal{T}}^{(p)}_n)=\Delta \right)=\frac{C(q)}{C(p)}\rho^{-N}.
\end{equation}
As we have
\[
N=\# \mathcal{M}(\Delta) -p +\sum_{v \in \mathcal{F}^*}N(M_v)=\sum_{1\leq i \leq q}\# \tau_i - p +\sum_{v \in \mathcal{F^*}}N(M_v)=q+\sum_{v \in \mathcal{F}^*}(c_v + N(M_v)) -p,
\]
we get
\begin{equation*}
 \lim_{n \to \infty}\mathbb{P}\left(B_r^{\bullet}(\overline{\mathcal{T}}^{(p)}_n)=\Delta \right)=\frac{\rho^{-q}C(q)}{\rho^{-p}C(p)}\prod_{v \in \mathcal{F}^*}\rho^{-c_v}\rho^{-N(M_v)}.
\end{equation*}
Now, since $\sum_{v \in \mathcal{F}^*}(c_v -1)=p-q$, we can multiply the right-hand side by $(\alpha\rho)^{p-q-\sum_{v \in \mathcal{F}^*}(c_v -1)}$, which yields
\begin{equation*}
\lim_{n \to \infty}\mathbb{P}\left(B_r^{\bullet}(\overline{\mathcal{T}}^{(p)}_n)=\Delta \right)=\frac{\alpha^{-q}C(q)}{\alpha^{-p}C(p)}\prod_{v \in \mathcal{F}^*}\rho^{-1}\alpha^{-c_v+1}\rho^{-N(M_v)},
\end{equation*}

that is
\begin{equation*}
\lim_{n \to \infty}\mathbb{P}\left(B_r^{\bullet}(\overline{\mathcal{T}}^{(p)}_n)=\Delta\right)=\frac{\alpha^{-q}C(q)}{\alpha^{-p}C(p)}\prod_{v \in \mathcal{F}^*}\theta(c_v)\frac{\rho^{-N(M_v)}}{Z(c_v+1)},
\end{equation*}
for $\theta(k)=\rho^{-1}\alpha^{-k+1}Z(k+1)$. 
\end{proof}

Let us give a few properties of $\theta$ that will be useful in the sequel. These properties are obtained from the analytic combinatorial work in {\cite{bouttier-carrance}}, rather than explicit enumeration as was the case for usual triangulations in {\cite{curien-legall}.

First, the asymptotics of $Z$ give:
\begin{equation}
\label{theta-asympt}
\theta(k)\underset{k \to \infty}{\sim}\frac{1}{2}\sqrt{\frac{3}{\pi}}k^{-5/2}.
\end{equation}

Moreover, $\theta$ has the following generating function $g_{\theta}$:
\begin{equation}
\label{eq theta generating function}
g_{\theta}(x)=\sum_{k=0}^{\infty}\theta(k)x^k=1 - \frac{3}{\left(\sqrt{\frac{4-x}{1-x}} +1 \right)^2 -1} \ \ \forall x \in [0,1).
\end{equation}
Indeed, the generating function of $\theta$ may be written, for $0 \leq x <1$:
\begin{align*}
\sum_{k \geq 0}\theta(k)x^k&=\sum_{k \geq 0}\rho^{-1}\alpha^{-k+1}Z(k+1)x^k=\frac{\alpha}{\rho}\sum_{k \geq 0}\left(\frac{x}{\alpha}\right)^kZ(k+1)\\
&=\frac{\alpha^2}{x\rho}\sum_{k \geq 1}\left(\frac{x}{\alpha}\right)^kZ(k)=\frac{\alpha^2}{x\rho}\left(\sum_{k \geq 0}\left(\frac{x}{\alpha}\right)^kZ(k) - Z(0) \right)\\
&=\frac{2}{x}\left(\frac{1}{2}\left(\frac{1+\frac{7}{4} x-\frac{x^2}{2}+\sqrt{(\frac{x}{4}-1)(x-1)^3}}{1-\frac{x}{4}}\right) -1\right)\\
&=\frac{-4+9 x-2x^2+2\sqrt{(x-4)(x-1)^3}}{x(4-x)}\\
&=1 - \frac{3}{\left(\sqrt{\frac{4-x}{1-x}} +1 \right)^2 -1}=g_{\theta}(x).
\end{align*}

It is straightforward to obtain from this that $\theta$ is a probability distribution with mean $1$, so that, considered as the offspring distribution of a branching process, it is critical.

Let $Y=(Y_r)_{r\geq 0}$ be a Galton-Watson process with offspring distribution $\theta$, and let us write $\mathcal{P}_k(\cdot)$ for the law of $Y$ given $Y_0=k$, and $\mathcal{E}_k[\, \cdot \,]$ for the corresponding expectation. Then, for every $r \geq 1$, the generating function of $Y_r$ under $\mathcal{P}_1$ is the $r$-th iterate $g_{\theta}^{(r)}$ of $g_{\theta}$. It is easy to show that this iterate has a very nice expression for any positive integer $r$:
\begin{equation}
\label{iterate-expr}
\mathcal{E}_1 \left[ x^{Y_r} \right] =g_{\theta}^{(r)}(x)=1 - \frac{3}{\left(\sqrt{\frac{4-x}{1-x}} +r \right)^2 -1} \ \ \ \forall x \in [0,1).
\end{equation}

Note that a similarly convenient expression for the $r$-th iterate of the generating function also exists for the offspring distributions associated to the skeleton decompositions of usual triangulations {\cite{Krikun2005,curien-legall}} and of quadrangulations {\cite{uipq,legall-lehericy}}.

Using the transfer theorem (see Theorem VI.3 in {\cite{flajolet-sedgewick}}), we deduce from \eqref{iterate-expr} that
\begin{equation}
\label{gw-tail}
\mathcal{P}_1\left(Y_r=k \right)\underset{k \to \infty}{\sim}\sqrt{\frac{3}{\pi}}\frac{r}{2}k^{-5/2}.
\end{equation}

Let us denote by $\mathbb{F}_{p,q,r}$ the set of $(p,q,r)$-admissible forests. We also define the set $\mathbb{F}'_{p,q,r}$ of pointed forests satisfying the same conditions as $(p,q,r)$-admissible forests, except that the tree with a distinguished vertex is not necessarily $\tau_1$, and the set $\mathbb{F}''_{p,q,r}$ of forests which satisfy the same conditions but do not have a distinguished vertex.

We now prove that the ``skeleton part'' of \eqref{skeleton-proba-limit} defines a probability measure on $\mathbb{F}_{p,r}=\cup_{q\geq 1}\mathbb{F}_{p,q,r}$, similarly to {\cite[Lemma~3]{curien-legall}}:
\begin{lemma}
\label{skeleton-is-proba-lem}
For every $p \geq 1$ and $r \geq 1$,
\begin{equation}
\label{skeleton-is-proba-eq}
\sum_{q=1}^{\infty}\sum_{\mathcal{F} \in \mathbb{F}_{p,q,r}}\frac{4^{-q}C(q)}{4^{-p}C(p)}\prod_{v \in \mathcal{F}^{*}}\theta(c_v)=1.
\end{equation}
\end{lemma}

\begin{proof}
Like in the proof of {\cite[Lemma~3]{curien-legall}}, \Cref{skeleton-is-proba-eq} amounts to
\begin{equation}
\label{h-stationary-first-expr}
\sum_{q=1}^{\infty}\frac{h(q)}{h(p)}\mathcal{P}_q\left(Y_r=p\right)=1,
\end{equation}
with
\begin{equation}
h(k)=2\sqrt{\pi}\frac{4^{-k}C(k)}{k}.
\end{equation}

Now, \Cref{h-stationary-first-expr} is equivalent to 
\begin{equation}
\label{h-stationary}
\sum_{q=1}^{\infty}h(q)\mathcal{P}_q\left(Y_r=p\right)={h(p)}
\end{equation}
or, in other words, to the fact that $h$ is an infinite stationary measure for $Y$.

Let $\Pi$ be the generating function of the sequence $(h(k))_{k\geq 1}$:
\[
\Pi(x):=\sum_{k=1}^{\infty}h(k)x^k=\sum_{k=1}^{\infty}\frac{1}{k}C(k)\left(\frac{x}{4}\right)^k.
\]
Contrary to the case of usual triangulations, we do not have an explicit expression for $h$, but, by integrating \eqref{gf-expression}, we obtain one for $\Pi$:
\[
\Pi(x)=\sqrt{\frac{4-x}{1-x}}-2 \ \ \ \forall\, 0<x<1
\]
To prove that  $h$ is an infinite stationary measure for $Y$, it is enough to check that $\Pi\left(g_{\theta}(x)\right) - \Pi\left(g_{\theta}(0)\right) = \Pi(x)$ for every $x \in [0,1)$, which follows from the explicit formulas for $g_{\theta}$ and $\Pi$.
\end{proof}

With \Cref{skeleton-is-proba-lem}, we can define a probability measure ${\bf P}_{p,r}$ on $\mathbb{F}_{p,r}$ by setting, for any $\mathcal{F} \in \mathbb{F}_{p,q,r}$,
\begin{equation}
\label{forest-proba}
{\bf P}_{p,r}(\mathcal{F}):=\frac{4^{-q}C(q)}{4^{-p}C(p)}\prod_{v \in \mathcal{F}^{*}}\theta(c_v).
\end{equation}

Let us note $\mathbb{C}_{p,r}$ the set of Eulerian triangulations of the cylinder of height $r$ and bottom boundary length $2p$. We can define a probability measure $\mathbb{P}_{p,r}$ on $\mathbb{C}_{p,r}$, by first setting the skeleton to be distributed according to $\mathbf{P}_{p,r}$, then, conditionally on the skeleton, filling the slots by independent Boltzmann triangulations (whose boundary lengths are prescribed by the skeleton). Thus, \Cref{skeleton-proba-limit-lem} amounts to stating that, if $\Delta \in \mathbb{C}_{p,r}$,
\begin{equation}
\lim_{n \to \infty}\mathbb{P}\left(B_r^{\bullet}(\overline{\mathcal{T}}^{(p)}_n)=\Delta \right)=\mathbb{P}_{p,r}(\Delta).
\end{equation}

In other words, the law of $B_r^{\bullet}(\overline{\mathcal{T}}^{(p)}_n)$ converges weakly to $\mathbb{P}_{p,r}$ as $n\to \infty$.

Note that the expression \eqref{forest-proba} implies that, if a random cylinder triangulation $A$ is distributed as $\mathbb{P}_{p,r}$, then, for all $1 \leq s \leq r$, its hull $B_s^{\bullet}(A)$ will be distributed as $\mathbb{P}_{p,s}$, or, in other words, the laws $(\mathbb{P}_{p,r})_{r\geq 1}$ are consistent. This implies that the sequence of random maps $(\mathcal{T}^{(p)}_n)_n$ has a local distributional limit. To express this result more precisely, we need to generalize the notion of hulls to some infinite maps. First, for any infinite planar Eulerian triangulation $A$  with a boundary, we can define its ball $B_r(A)$ like in the finite case. Then, if $A$ has a unique end, only one connected component of $A \setminus B_r(A)$ is infinite, so that we can fill all the finite holes, to get the hull $B_r^{\bullet}(A)$.

We then have the following result:
\begin{prop}
For any integer $p \geq 1$, the sequence of random maps $(\mathcal{T}^{(p)}_n)_n$ converges in distribution, in the sense of local limits of rooted maps, to an infinite map that we call the \textbf{uniform infinite Eulerian triangulation of the \boldmath$2p$-gon}, and that we denote by \unboldmath$\mathcal{T}^{(p)}_{\infty}$. It is a random infinite Eulerian triangulation of the plane, with an alternating, semi-simple boundary of length $2p$, that has a unique end almost surely, and such that $B_r^{\bullet}(\mathcal{T}^{(p)}_{\infty})$ has law $\mathbb{P}_{p,r}$, for every integer $r \geq 1$.
\end{prop}

For $p=1$, we can perform the transformation described in \Cref{DigonToRoot}, which yields a random infinite planar Eulerian triangulation, which we denote by $\mathcal{T}_{\infty}$. This random infinite map is the local limit of uniform rooted planar Eulerian triangulations with $n$ black faces when $n \to \infty$, therefore we call it the \textbf{Uniform Infinite Planar Eulerian Triangulation} (UIPET).

The UIPET is the equivalent of well-known models of random infinite planar maps such as the UIPT or the UIPQ (see {\cite{angel,uipq}}), in the case of Eulerian triangulations. Note that this present work gives the first construction of the UIPET.\\

Let  $L^{(p)}_r$ be the length of the top cycle of $B^{\bullet}_r(\mathcal{T}^{(p)}_{\infty})$. When $p=1$, we write $L_r$ for $L^{(1)}_r$ for simplicity.

Let us first note that $\mathcal{T}^{(p)}_{\infty}$ exhibits a spatial Markov property. Let $r, s$ be integers with $1\leq r < s$, and $\Delta \in \mathbb{C}_{p,s}$. Let $2q$ be the length of the boundary $\partial _r\Delta$. We can obtain $\Delta$ by gluing a triangulation $\Delta'' \in \mathbb{C}_{q,s-r}$ on top of a triangulation $\Delta^{\prime} \in \mathbb{C}_{p,r}$, whose top boundary has length $q$. From the explicit formula of \eqref{forest-proba}, we get
\begin{equation}
\mathbb{P}_{p,s}(\Delta)=\mathbb{P}_{p,r}(\Delta') \cdot \mathbb{P}_{q,s-r}(\Delta'').
\end{equation}
Therefore, conditionally on $\{L^{(p)}_r = q\}$, $B^{\bullet}_s(\mathcal{T}^{(p)}_{\infty}) \backslash B^{\bullet}_r(\mathcal{T}^{(p)}_{\infty})$ follows $\mathbb{P}_{q,s-r}$, and is independent of $B^{\bullet}_r(\mathcal{T}^{(p)}_{\infty})$. By letting $s \to \infty$, we obtain that, conditionally on $\{L^{(p)}_r = q\}$, the triangulation $\mathcal{T}^{(p)}_{\infty} \backslash B^{\bullet}_r(\mathcal{T}^{(p)}_{\infty})$ is distributed as $\mathcal{T}^{(q)}_{\infty}$ and is independent of $\mathcal{T}^{(p)}_{\infty}$.

We now give a technical but useful result on the law of $L_r$, which is the equivalent in our case of {\cite[Lemma~4]{curien-legall}}.
\begin{lemma}
\label{top-length-bounds}
There exists a constant $C_0 > 0$ such that for any $\alpha \geq 0$, and for any integers $r,p \geq 1$,
\begin{equation}
\label[ineq]{top-length-bound1}
  \mathbb{P}\left(L_r=p\right) \leq \frac{C_0}{r^2}
\end{equation}
and 
\begin{equation}
\label[ineq]{top-length-bound2}
\mathbb{P}\left(L_r \geq \alpha r^2 \right) \leq C_0 e^{-\alpha/4}.
\end{equation}
\end{lemma}

Let us fix some notation before getting to the proof of \Cref{top-length-bounds}. For $1 \leq r < s$, let $\mathcal{F}^{(1)}_{r,s}$ be the skeleton of $B^{\bullet}_s(\mathcal{T}^{(1)}_{\infty}) \backslash B^{\bullet}_r(\mathcal{T}^{(1)}_{\infty})$. We let $\widetilde{\mathcal{F}}^{(1)}_{r,s}$ by the non-pointed forest obtained by a uniform cyclic permutation of $\mathcal{F}^{(1)}_{r,s}$, and by forgetting the distinguished vertex. Thus, on the event $\{L_r = p\}\cap \{L_s = q\}$, $\widetilde{\mathcal{F}}^{(1)}_{r,s}$ is a random element of $\mathbb{F}''_{p,q,s-r}$.

\begin{proof}
These bounds are obtained very similarly to those of Lemma 4 in {\cite{curien-legall}}, with the (small) additional difficulty that in our case we do not have an explicit expression for $h$, but only asymptotics. We give the full argument here as it consists in a short but rather involved computation.

First observe that
\begin{equation*}
\Prob{L_r=p}=\sum_{\mathcal{F} \in \mathbb{F}''_{1,p,r}}\Prob{\widetilde{\mathcal{F}}^{(1)}_{0,r}=\mathcal{F}}=\sum_{\mathcal{F} \in \mathbb{F}''_{1,p,r}}\frac{h(p)}{h(1)}\prod_{v \in \mathcal{F}^*}\theta(c_v).
\end{equation*}
Thus
\begin{equation*}
\Prob{L_r=p}=\frac{h(p)}{h(1)}\mathcal{P}_p\left(Y_r=1 \right).
\end{equation*}
From the definition of $h$ and the asymptotics of $C(p)$, there exists a constant $C_1$ such that, for every $p \geq 1$,
\begin{equation*}
h(p)\leq \frac{C_1}{\sqrt{p}}.
\end{equation*}
Moreover, from \eqref{iterate-expr}, we have 
\begin{equation}
\label{proba-gen-r-0}
\mathcal{P}_1\left( Y_r=0\right)=1 - \frac{3}{(r+2)^2-1},
\end{equation}
hence
\begin{align*}
\mathcal{P}_p\left( Y_r=1\right)&=\lim_{x \downarrow 0}x^{-1}\left(\mathcal{E}_p\left[x^{Y_r} \right] - \mathcal{P}_p\left( Y_r=0\right)\right)\\
&=\lim_{x \downarrow 0}x^{-1}\left(\left(1 - \frac{3}{\left(\sqrt{\frac{4-x}{1-x}}+r\right)^2-1}\right)^p - \left(1  -\frac{3}{(r+2)^2-1}\right)^p\right)\\
&=\frac{9p(r+2)}{2((r+2)^2-1)((r+2)^2-4)}\left(1 - \frac{3}{(r+2)^2-1} \right)^p.
\end{align*}
Therefore, for some constant $C_3 >0$,
\begin{equation*}
\Prob{L_r=p}\leq \frac{C_1}{h(1)}\sqrt{p}\frac{9(r+2)}{((r+2)^2-1)((r+2)^2-4)}\left(1 - \frac{3}{(r+2)^2-1} \right)^{p-1} \leq \frac{C_3}{r^2}\sqrt{\frac{p}{r^2}}e^{-3p/r^2}.
\end{equation*}
The bound \eqref{top-length-bound1} immediately follows. As for \eqref{top-length-bound2}, since the function $x \mapsto \sqrt{x}e^{-3x}$ is decreasing for $x \geq 1/6$, we have, for $\alpha \geq 1/6$, for some constant $C_4 >0$,
\begin{equation*}
\Prob{L_r> \alpha r^2}\leq \sum_{p=\alpha r^2+1}^{\infty}\frac{C_3}{r^2}\sqrt{\frac{p}{r^2}}e^{-3p/r^2} \leq \frac{C_3}{r^2} \int_{\alpha r^2}^{\infty}\sqrt{\frac{x}{r^2}}e^{-3x/r^2}\mathrm{d}x \leq C_4 e^{-\alpha/4}.
\end{equation*}
\end{proof}

We now fix a positive constant $a \in (0,1)$. For every integer $r \geq 1$, let $N^{(a)}_r$ be uniformly random in $\{\lfloor ar^2\rfloor +1, \dots, \lfloor a^{-1}r^2\rfloor\}$. We also consider a sequence $\tau_1, \tau_2, \dots$ of independent Galton-Watson trees with offspring distribution $\theta$, independent of  $N^{(a)}_r$. For every integer $j \geq 0$, we write $[\tau_i]_j$ for the tree $\tau_i$ truncated at generation $j$. 

Using the same arguments that yield Proposition 5 from Lemma 4 in {\cite{curien-legall}}, the above lemma implies the following bound:
\begin{prop}
\label{prop-comparison-pple}
There exists a constant $C_1$, which only depends on $a$, such that, for every sufficiently large integer $r$, for every choice of $s \in \{r+1, r+2, \dots \}$, for every choice of integers $p$ and $q$ with $ar^2<p,q\leq a^{-1}r^2$, for every forest $\mathcal{F} \in  \mathbb{F}''_{p,q,s-r}$,
\begin{equation}
\Prob{\widetilde{\mathcal{F}}^{(1)}_{r,s}=\mathcal{F}}\leq C_1 \Prob{([\tau_1]_{s-r}, \dots, [\tau_{N^{(a)}_r}]_{s-r})=\mathcal{F}}.
\end{equation}
\end{prop}

\subsection{Leftmost mirror geodesics}

We now define a type of paths in Eulerian cylinder triangulations that will be useful in the sequel.\\

Let $\Delta$ be an Eulerian cylinder triangulation of height $r \geq 1$. Let $x$ be a type-$j$ vertex of $\partial_{j}\Delta$, with $1 \leq j \leq r$. We define the \textbf{leftmost mirror geodesic} from $x$ to the bottom cycle in the following way. Enumerate in clockwise order around $x$ all the half-edges incident to it, starting from the half-edge of $\partial_j{\Delta}$ that is to the right of $x$. The first edge on the leftmost mirror geodesic starting from $x$ is the last edge connecting $x$ to $\partial_{j-1}\Delta$ arising in this order. The path is then continued by induction (see \Cref{LeftmostMirrorGeods}). Note that, taken in the reverse order, this path is an oriented geodesic, hence the name \emph{mirror} geodesic. (Such a precision is not necessary in {\cite{curien-legall}}, that deals with proper, symmetric distances.)

\begin{figure}[htp]
\centering
\includegraphics[scale=1.2]{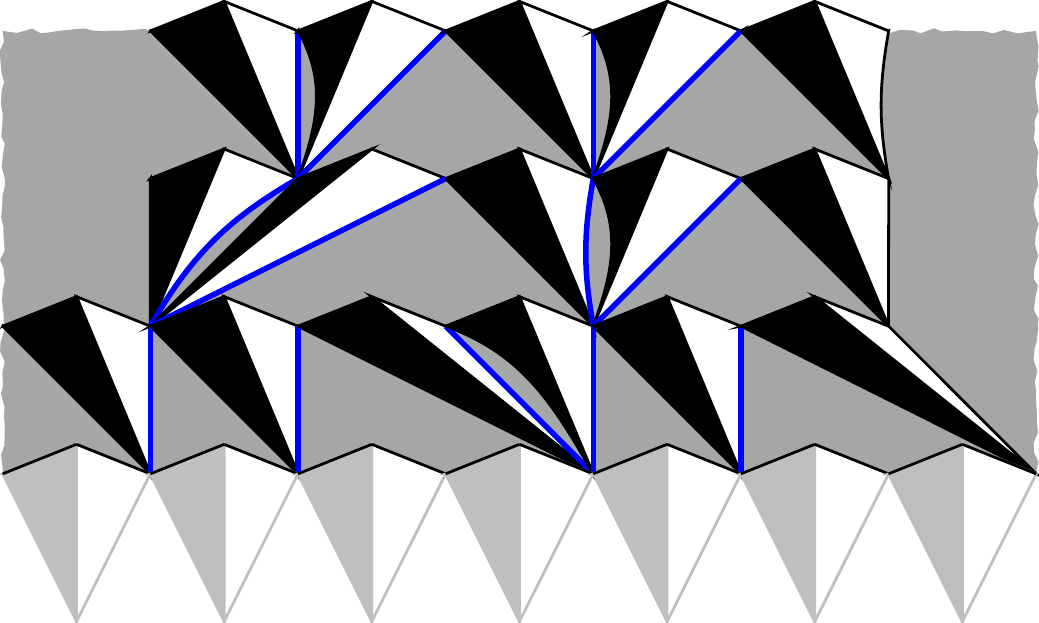}
\caption{Some leftmost mirror geodesics (depicted in blue) in a portion of an Eulerian cylinder triangulation.}
\label{LeftmostMirrorGeods}
\end{figure}

The coalescence of leftmost geodesics from distinct vertices can be characterized by the skeleton of $\Delta$. Indeed, let $u, v$ be two distinct type-$r$ vertices of $\partial^*\Delta$. Let $\mathcal{F}$ be the skeleton of $\Delta$, $\mathcal{F}'$ the subforest of $\mathcal{F}$ consisting of the trees rooted between $u$ and $v$ left-to-right in $\partial^*\Delta$, and $\mathcal{F}''$ be the rest of the trees in $\mathcal{F}$. Then, for any $k \in \{1, 2, \dots, r\}$, the leftmost mirror geodesics from $u$ and $v$ merge before step $k$ (possibly exactly at step $k$) if and only if at least one of the two forests $\mathcal{F}'$ and $\mathcal{F}''$ have height strictly smaller than $k$.

\section{The Lower Half-Plane Eulerian Triangulation}
\label{sec lhpet}

We now construct a triangulation of the lower half-plane $\mathbb{R} \times \mathbb{R}_-$ that will be crucial to prove \Cref{thm-total-asympt-prop-in-finite-trig}, and that also is an object of interest in itself. Note that this construction is very similar that of the LHPT in {\cite[Section~3.2]{curien-legall}}.

We start with a doubly infinite sequence $(\mathscr{T}_i)_{i \in \mathbb{Z}}$ of independent Galton-Watson trees with offspring distribution $\theta$. They are embedded in the lower half-plane so that, for every $i \in \mathbb{Z}$, the root of $\mathscr{T}_i$ is $(\frac{1}{2}+i,0)$, and such that the collection of all vertices of all the $\mathscr{T}_i$ is  $(\frac{1}{2}+\mathbb{Z}) \times \mathbb{Z}_{\leq 0}$, with vertices at height $k$ being of the form $(\frac{1}{2}+i,-k)$. We also assume that the embedding is such that the collection of vertices of the $\mathscr{T}_i$, for $i \geq 0$, is  $(\frac{1}{2}+\mathbb{Z}_{\geq 0}) \times \mathbb{Z}_{\leq 0}$ (see \Cref{sketch-lhpet}).

We can now build the triangulation itself. We start with the ``distinguished'' modules, which will play the role of skeleton modules for our infinite triangulation. They are naturally associated with the vertices of the infinite collection of trees in the following way. To each vertex $(\frac{1}{2}+i,j)$ in one of the trees, we associate a module whose type $n+1$ vertices are $(i,j)$ and $(i+1,j)$. The type $n$ vertex is $(k,j-1)$, where $k$ is the minimal integer such that $(\frac{1}{2}+k,j-1)$ is the child of $(\frac{1}{2}+i',j)$, for some $i'>i$. The last vertex, of type $n+2$, is set to be $(\frac{1}{2}+i,j+\varepsilon)$, for an arbitrary $0<\varepsilon<1$. As for the (outer) edges of these skeleton modules, we draw them such that they are all distinct, and do not cross. Having completely determined the configuration of the skeleton edges from the infinite collection of trees, we fill in the slots bounded by these modules, with independent Boltzmann Eulerian triangulation of appropriate perimeters. (Note that each point of the form $(i, j)$, with $j\geq 1$, is at the top of a slot of perimeter $2(c_{i,j}+1)$, where $c_{i,j}$ is the number of children of $(\frac{1}{2}+i,j)$ in the infinite collection of trees.)

We obtain an Eulerian triangulation of the lower-half plane, which we will note $\mathcal{L}$ and call the \textbf{Lower Half-Plane Eulerian Triangulation} (LHPET). It is rooted at the edge from $(0,0)$ to $(\frac{1}{2}, \varepsilon)$.

We will denote by $\mathcal{L}_{ [0,r]}$ the infinite rooted planar map obtained by keeping only the first $r$ layers of $\mathcal{L}$ (having the skeleton modules at level $r$ as ghost modules), and denote by $\mathcal{L}_r$ the lower boundary of $\mathcal{L}_{[0,r]}$. For integers $0 \leq m < n$, we also define $\mathcal{L}_{[m,n]}$, to be the map obtained by keeping only the layers of $\mathcal{L}$ that lie between the levels $m$ and $n$ (the skeleton modules at level $m$ making up the top boundary of $\mathcal{L}_{[m,n]}$, and the ones at level $n$ being its bottom ghost modules).

\begin{figure}[htp]
\centering
\includegraphics{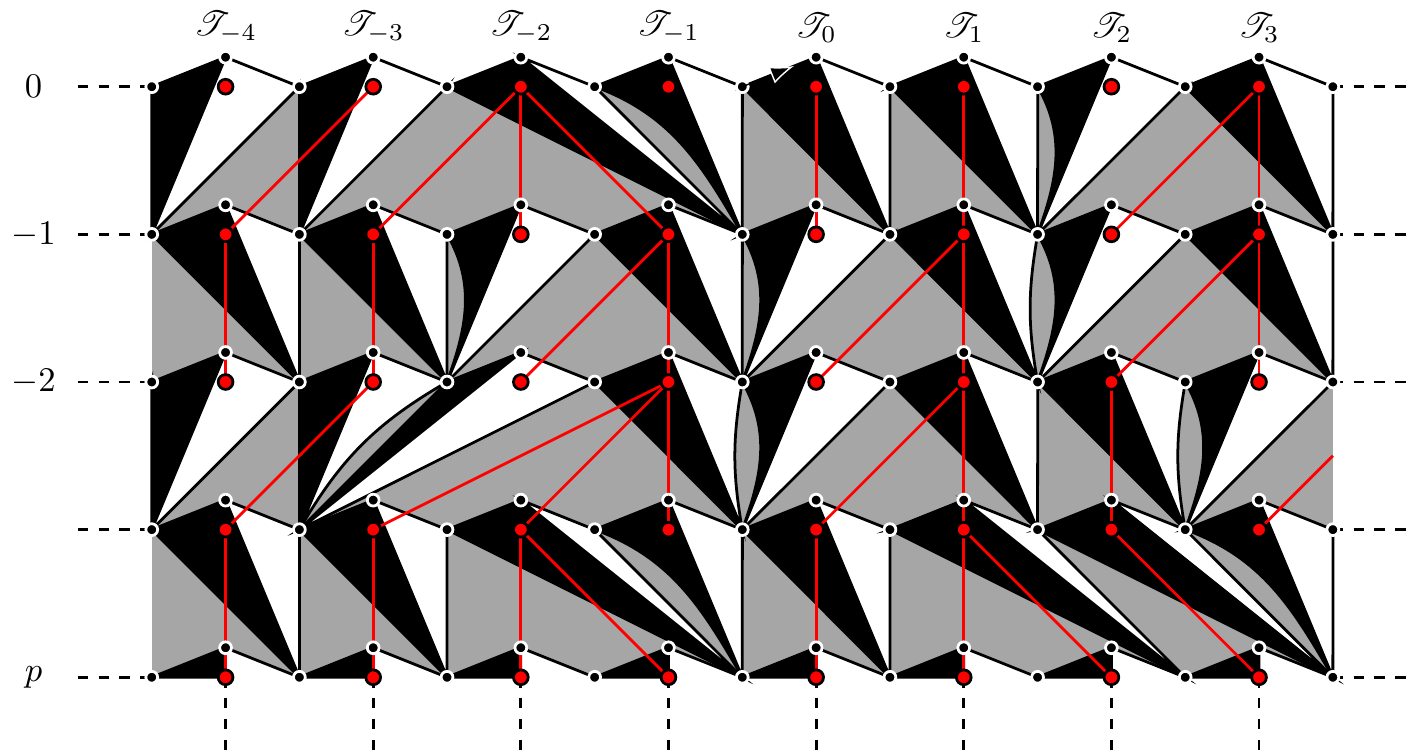}
\caption{Construction of the LHPET.}
\label{sketch-lhpet}
\end{figure}

While we will not use this result in the sequel, note that $\mathcal{L}$ is the local limit of $\mathcal{T}_{\infty}^{(p)}$  ``seen from infinity''. This statement is made more precise in the following proposition:
\begin{prop}
\label{prop lim loc lhpet}
Set $p \geq 1$, and for every $r \geq 1$, define $\widetilde{B}^{\bullet}_{s}\left(\mathcal{T}^{(p)}_{\infty}\right)$ as the hull $B^{\bullet}_{s}\left(\mathcal{T}^{(p)}_{\infty}\right)$ re-rooted at an edge uniform on those of $\partial_s\left(\mathcal{T}^{(p)}_{\infty}\right)$ that are oriented so that the top face is lying on their left. Then
\[
\widetilde{B}^{\bullet}_{s}\left(\mathcal{T}^{(p)}_{\infty}\right)\xrightarrow[s \to \infty]{(d)} \mathcal{L}
\]
in the sense of local limits of rooted planar maps.
\end{prop}

An equivalent result was stated for usual triangulations in {\cite[Proposition~7]{curien-legall}}, but its proof was not detailed, since it is similar to the proof of the equivalent convergence to the \emph{Upper} Half-Plane Triangulation. We give a proof of our result both for the sake of completeness, and because it involves nonetheless a few arguments that are different from the ones for the upper half-plane models. 

\begin{proof}
Recall that, for an Eulerian triangulation $A$ (possibly with a boundary) and an integer $r \geq 1$, we denote by $\mathcal{B}_r(A)$ the ball of radius $r$ of $A$, that is, the union of all edges and faces of $A$ incident to a vertex at (oriented) distance strictly less than $r$ from the root. Proving the proposition amounts to showing that, for every $r \geq 1$, for every rooted planar map $A$,
\begin{equation}
\label{eq:cv-of-displaced-balls-to-l}
\Prob{\mathcal{B}_r\left(\widetilde{B}^{\bullet}_{s}\left(\mathcal{T}^{(p)}_{\infty}\right)\right)=A} \xrightarrow[s \to \infty]{}  \Prob{\mathcal{B}_r(\mathcal{L})=A}.
\end{equation}

To obtain this convergence, we will need a bit of additional notation.  We fix $r \geq 1$, and note $[\mathscr{T}]_r$ for the tree $\mathscr{T}$ truncated at height $r$, and similarly for a forest. For any $s \geq 1$, we write $\mathcal{F}_{0,s}^{(p)}=\left( \mathscr{T}_0^{(p)}, \mathscr{T}_1^{(p)}, \dots, \mathscr{T}_{L_s^{(p)}}^{(p)}\right)$ for the skeleton of $B^{\bullet}_{s}\left(\mathcal{T}^{(p)}_{\infty}\right)$. Let us fix $k \geq 1$. For any $q \geq 1$, for any forest $\mathcal{F}=(\sigma_0, \dots, \sigma_{l-1}) \in \mathbb{F}_{q,l,r}$ with $l \geq 2k+1$, we write $\Phi_k(\mathcal{F})=(\sigma_{i-k}, \dots, \sigma_{i-1}, \sigma_{i}, \dots, \sigma_{i+k})$, where $i$ is a uniform index on $0, \dots, l-1$, and the indices for the $\sigma_j$ are extended to $\mathbb{Z}$ by periodicity.

We will prove that, for every collection $\mathcal{F}_k=(\tau_{-k}, \dots, \tau_0, \dots, \tau_k)$ of $2k+1$ plane trees of maximal height $r$,
\begin{equation}
\label{eq:cv-of-displaced-forest-to-gw}
\Prob{\{\Phi_k\left([\mathcal{F}_{0,s}]_r\right)=\mathcal{F}_k\} \cap \{L_s^{(p)} \geq 2k+1\}} \xrightarrow[s \to \infty]{} \Prob{\left( [\mathscr{T}_{-k}]_r, \dots, [\mathscr{T}_{0}]_r, \dots, [\mathscr{T}_{k}]_r\right)=\mathcal{F}_k}.
\end{equation}

If $k$ is large enough, we can find a set $\mathbf{F}_k$ of forests such that the probability of the event 
\[
\left( [\mathscr{T}_{-k}]_r, \dots, [\mathscr{T}_{0}]_r, \dots, [\mathscr{T}_{k}]_r\right) \in \mathbf{F}_k
\]
is close to 1, and such that, on that event, the ball $\mathcal{B}_r(\mathcal{L})$ is a deterministic function of the truncated trees $[\mathscr{T}_{-k}]_r, \dots, [\mathscr{T}_{0}]_r, \dots, [\mathscr{T}_{k}]_r$ and of the triangulations with a boundary filling in the slots associated with the vertices of these trees. (Note that we need $k$ to be large, so that the $(2k+1)$ central trees of the skeleton of $\mathcal{L}$ and the associated slots are enough to cover the ball $\mathcal{B}_r(\mathcal{L})$, not only vertically, which is a given, but also horizontally.) Likewise, on the event $\{\Phi_k\left([\mathcal{F}_{0,s}]_r\right)\in \mathbf{F}_k\} \cap \{L_s^{(p)} \geq 2k+1\}$, the ball $\mathcal{B}_r\left( \widetilde{B}^{\bullet}_{s}\left(\mathcal{T}^{(p)}_{\infty}\right)\right)$ is given by the same deterministic function of the trees in $\Phi_k\left([\mathcal{F}_{0,s}]_r\right)$ and of the associated triangulations with a boundary.

Moreover, we claim that, for every fixed $p \geq 1$ and $j \geq 1$,
\begin{equation}
\label{eq:high-cycles-are-small}
\Prob{L_s^{(p)} =j} \xrightarrow[s \to \infty]{}0.
\end{equation}
Indeed, we can write
\[
\Prob{L_s^{(p)} =j}=\frac{h(j)}{h(p)}\mathcal{P}_j\left(Y_s=p \right),
\]
and, from $\mathbb{E}_{j}[x^{Y_s}]=(g_{\theta}^{(s)}(x))^j$, we get that $\mathcal{P}_j\left(Y_s=p \right) \xrightarrow[s \to \infty]{}0$.

Thus, the desired convergence of \eqref{eq:cv-of-displaced-balls-to-l} will follow from \eqref{eq:cv-of-displaced-forest-to-gw} and \eqref{eq:high-cycles-are-small}.

It remains to prove \eqref{eq:cv-of-displaced-forest-to-gw}. Let us fix $\mathcal{F}_k$ as above. From the definition of the $\mathscr{T}_i$, we have
\begin{equation}
\label{eq: proba of centerend forest in L}
\Prob{\left( [\mathscr{T}_{-k}]_r, \dots, [\mathscr{T}_{0}]_r, \dots, [\mathscr{T}_{k}]_r\right)=\mathcal{F}_k} = \prod_{v \in (\tau_{-k}, \dots, \tau_k)^*} \theta(c_v),
\end{equation}
where, as before, for a forest $\mathcal{F}$, $\mathcal{F}^*$ denotes the set of vertices in $\mathcal{F}$ that are not at the maximal height, and, for such a vertex $v$, $c_v$ is its number of children.

Now, using the definition of the law $\mathbf{P}_{p,s}$ of $B_s^{\bullet}\left(\mathcal{T}_{\infty}^{(p)}\right)$, the left-hand side of \eqref{eq:cv-of-displaced-forest-to-gw} is equal to 
\begin{align*}
&\sum_{l=2k+1}^{\infty}\sum_{\mathcal{F} \in \mathbb{F}_{p,l,s},\Phi_k\left(\mathcal{F}\right)=\mathcal{F}_{k}} \frac{4^{-l}C(l)}{4^{-p}C(p)} \prod_{v \in \mathcal{F}^*} \theta(c_v) \\
&=\left( \prod_{v \in  (\tau_{-k}, \dots, \tau_k)^*} \theta(c_v)\right) \cdot \left( \sum_{l=2k+1}^{\infty}  \frac{4^{-l}C(l)}{4^{-p}C(p)} \prod_{\substack{v \in (\sigma_0, \dots, \sigma_{l-2k-1})^* \bigcup (\tilde{\sigma}_1, \dots, \tilde{\sigma}_{m_k})^*\\\#\sigma_0(s) + \dots +\#\sigma_{l-2k-1}(s) +\# \tilde{\sigma}_1(s-r) + \dots +\# \tilde{\sigma}_{m_k}(s-r)=p}} \theta(c_v)\right),
\end{align*}
where $m_k$ is the number of vertices at generation $r$ in $\mathcal{F}_k$, while $\sigma_0, \dots, \sigma_{l-2k-1}$ stand for the trees (of maximal height $s$) not selected in $\mathcal{F}_k$, and $\tilde{\sigma}_1, \dots, \tilde{\sigma}_{m_k}$ stand for the trees (of maximal height $s-r$) obtained after truncation of the selected trees.

Let us denote by $A_s$ the second term of the second line of the previous equation. To conclude the proof, it suffices to show that 
\begin{equation}
\label{eq no loss of mass for lhpet}
\liminf_s A_s \geq 1.
\end{equation}
Indeed, in that case the liminf of the quantities in the left-hand side of \eqref{eq:cv-of-displaced-forest-to-gw} are greater than or equal to the right-hand side, for any choice of the forest $\mathcal{F}_k$. As the sum of the quantities on the right-hand side of \eqref{eq:cv-of-displaced-forest-to-gw} over these choices is equal to 1, necessarily the desired convergence holds.

Let us thus show \eqref{eq no loss of mass for lhpet}. Set $\varphi(l):=4^{-l}C(l)$. We have
\[
A_s =\sum_{l=2k+1}^{\infty}\frac{\varphi(l)}{\varphi(p)}\sum_{q=0}^{p} \mathcal{P}_{l-(2k+1)}\left(Y_s=q\right)\mathcal{P}_{m_k}\left(Y_{s-r}=p-q\right).
\]

First, as $\theta$ is a critical offspring distribution, we get from \cite{papangelou} that, for any $q\geq 0$,
\[
\frac{\mathcal{P}_{m_k}\left(Y_{s-r}=p-q\right)}{\mathcal{P}_{m_k}\left(Y_{s}=p-q\right)} \xrightarrow[s \to \infty]{} 1.
\]

Thus, for any $l \geq 2k+1$, for every $\varepsilon >0$, for any sufficiently large $s$,
\begin{align*}
  \sum_{q=0}^p \mathcal{P}_{l-(2k+1)}\left(Y_s=q\right)\mathcal{P}_{m_k}\left(Y_{s-r}=p-q\right) & \geq (1 - \varepsilon) \sum_{q=0}^p \mathcal{P}_{l-(2k+1)}\left(Y_s=q\right)\mathcal{P}_{m_k}\left(Y_{s}=p-q\right) \\
   & \geq (1 - \varepsilon) \mathcal{P}_{l-(2k+1)+m_k}\left(Y_s=p\right).
\end{align*}

This implies that:
\[
A_s \geq (1 - \varepsilon) \sum_{l=2k+1}^{\infty}\frac{\varphi(l)}{\varphi(p)}\mathcal{P}_{l-(2k+1)+m_k}\left(Y_s=p\right).
\]

Now, from the asymptotics of $C(l)$, we have that, for some $l_0 \geq 0$, for any $l \geq l_0$, we have
\[
\varphi(l) \geq (1- \varepsilon)\varphi(l-(2k+1)+m_k),
\]

so that,
\[
A_s \geq (1 - \varepsilon)^2 \sum_{l=m_k \vee l_0}^{\infty}\frac{\varphi(l)}{\varphi(p)}\mathcal{P}_{l}\left(Y_s=p\right).
\]

Recall that $\varphi(l)=lh(l)$, which yields:
\begin{align*}
A_s &\geq (1 - \varepsilon)^2 \sum_{l=m_k \vee l_0 \vee p}^{\infty}\frac{h(l)}{h(p)}\mathcal{P}_{l}\left(Y_s=p\right)\\
&= (1 - \varepsilon)^2 \left( 1 - \sum_{l=0}^{m_k \vee l_0 \vee p-1}\frac{h(l)}{h(p)}\mathcal{P}_{l}\left(Y_s=p\right)\right),
\end{align*}
the last equality stemming from \eqref{h-stationary}.

Finally, we use once again the fact that, for any fixed $l$,
\[
\mathcal{P}_{l}\left(Y_s=p\right) \xrightarrow[s \to \infty]{} 0,
\]

to get that, for any $\varepsilon >0$,
\[
\liminf_{s}A_s \geq (1 - \varepsilon)^2.
\]

As $\varepsilon$ was completely arbitrary in the above chain of arguments, we get that
\[
\liminf_{s}A_s \geq 1.
\]

This completes the proof of the proposition.
\end{proof}

\section{Distances along the half-plane boundary}
\label{sec bounds}

To fulfill our goal of showing the asymptotic equivalence between the oriented and non-oriented distances in uniform Eulerian triangulations, we need as a technical ingredient some estimates on the (oriented) distances along the boundary of $\mathcal{L}$.

Note that the vertices on $\partial \mathcal{L}$ are of two types, those of coordinates $(i,0)$ for some $i \in \mathbb{Z}$, and those of coordinates $(i + 1/2,\varepsilon)$, for some $i \in \mathbb{Z}$. To simplify notation, the results in this section only deal with the distances between vertices of the first type, since we are interested in asymptotic estimates, and including the vertices of the second type only adds 1 or 2 to the considered distances. We will lay the stress on this generalization whenever it arises later in the paper.\\

In the sequel, we will use leftmost mirror geodesics, that were defined in \Cref{sect-skeleton} for finite cylinder triangulations, and that we generalize now to $\mathcal{L}$. For any $i \in \mathbb{Z}$, the leftmost mirror geodesic from $(i,0)$ in $\mathcal{L}$ is an infinite path $\omega$ in $\mathcal{L}$, whose reverse is an oriented geodesic, and that visits a vertex $\omega(n)$ in $\mathcal{L}_n$ at every step $n \geq 0$. It starts at $(i,0)$, and is obtained by choosing at step $n+1$ the leftmost edge between $\omega(n)$ and $\mathcal{L}_{n+1}$. As before, for $i <j$, the leftmost mirror geodesics from $(i,0)$ and $(j,0)$ will coalesce before hitting $\mathcal{L}_r$, if and only if all the trees $\mathscr{T}_{i}, \mathscr{T}_{i+1}, \dots, \mathscr{T}_{j-1}$ all have height strictly smaller than $r$.

\subsection{Block decomposition and lower bounds}

We first want to obtain lower bounds on the distances along the boundary of $\mathcal{L}$. For that purpose, we adapt the \textbf{block decomposition} of causal triangulations {\cite[Section~2.1]{chn}}, to $\mathcal{L}$.

\begin{figure}[htp]
\centering
\includegraphics{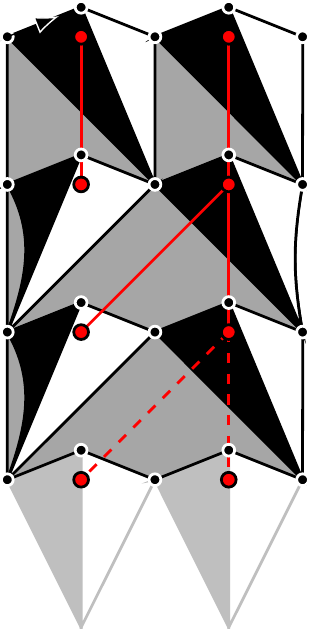}
\caption{The block of height 3 between $\mathscr{T}_{0}$ and $\mathscr{T}_1$ in the triangulation of \Cref{sketch-lhpet}. As before, ghost modules are shown in pale grey.}
\label{sketch-block}
\end{figure}

For $r \geq 1$, we define the random map $\mathcal{G}_r$ to be the planar map obtained from $\mathcal{L}_{[0,r]}$ by keeping only the faces and edges that are between $\mathscr{T}_0$ and $\mathscr{T}_{i_r}$, where $i_r$ is the smallest integer $i >0$ such that $\mathscr{T}_{i}$ has height at least $r$. More precisely, we only keep the skeleton modules that are at height smaller than or equal to $r$, belonging to trees $\mathscr{T}_i$, with $0 \leq i \leq i_r$, and the slots that are to the left of all these skeleton modules (see \Cref{sketch-block} for an example). Thus, $\mathcal{G}_r$ has one boundary that is naturally divided into four parts: the upper and lower parts that it shares with $\mathcal{L}_{[0,r]}$, and the left and right parts.

Note that $\mathcal{L}$ contains a lot of submaps that have the same law as $\mathcal{G}_r$: if $\mathscr{T}_i$, $\mathscr{T}_j$ are two consecutive trees reaching height $r$ in the skeleton of $\mathcal{L}$ (with $i <j$), we can define the submap of $\mathcal{L}_{[0,r]}$ encased between $\mathscr{T}_i$ (strictly) and $\mathscr{T}_j$ (included), which is obtained by keeping only the skeleton modules belonging to trees $\mathscr{T}_k$, with $i < k \leq j$, and the slots that are to the left of all these skeleton modules. Such a map has the same law as $\mathcal{G}_r$.

We call any map that can be a realization of $\mathcal{G}_r$, a \textbf{block of height \boldmath$r$}.\\

We define the \textbf{diameter} of \unboldmath$\mathcal{G}_r$, denoted \boldmath$\textbf{Diam}(\mathcal{G}_r)$ to be the minimal oriented distance from a vertex on its left boundary, to a vertex on its right boundary. Note that this diameter is not uniformly large when \unboldmath$r$ is large. However, we will now show that a long block is also typically wide. To do so, we consider the \textbf{median diameter} of a block.
\begin{defnt}
For any $r \geq 1$, let $f(r)$ be the median diameter of $\mathcal{G}_r$, that is, the largest number such that
\[
\Prob{\text{Diam}(\mathcal{G}_r) \geq f(r)} \geq \frac{1}{2}.
\]
\end{defnt}

We show the following upper bound on the median diameter, which is similar to the first part of {\cite[Theorem~5]{chn}}:
\begin{thm}
\label{thm-median-block-diameter}
There exists $c>0$ such that 
\[
f(r) \geq cr,
\]
for all $r$ sufficently large.
\end{thm}

Let us introduce a bit of notation before explaining how to prove this theorem. For any $m\geq 1$ and $h \geq 0$, consider the layer $\mathcal{L}_{[h,h+m]}$: it is composed of a (bi-infinite) sequence of blocks of height $m$, $(\mathcal{G}_{m}(i,h))_{i \in \Z}$. To avoid ambiguities, we set $\mathcal{G}_{m}(1,h)$ to be the block that has a part of $\mathscr{T}_k$ as its right boundary, where $k$ is the smallest integer $l \geq 1$ such that $\mathscr{T}_l$ has height at least $h+m$. For fixed $h,m$, these blocks are independent and distributed as $\mathcal{G}_m$. For $r \geq h+m$, we denote by $N_r(m,h)$ the maximal index $i$ such that the block $\mathcal{G}_m(i,h)$ is a sub-block of $\mathcal{G}_r$. (Note that, by our convention, the minimal such index is $i=1$, so that $N_r(m,h)$ is also the number of blocks $\mathcal{G}_m(i,h)$ that are sub-block of $\mathcal{G}_r$.)

To prove \Cref{thm-median-block-diameter}, we use, like in {\cite{chn}}, a renormalization scheme, splitting $\mathcal{G}_r$ into smaller blocks. This relies on an estimate of the numbers $N_r(2m,lm)$:

\begin{lemma}
\label{lem-lots-of-subblocks}
There exists $c > 0$ such that, for every $1 \leq m \leq cr$, we have
\[
\Prob{\inf_{0 \leq l \leq (r/m) -2}N_r(2m,lm) \geq c \left( \frac{r}{m}\right)^2} \geq \frac{7}{8}.
\]
\end{lemma}

The proof of this lemma can be adapted straightforwardly from the proof of {\cite[Lemma~1]{chn}}, in the case $\beta=2$. 

\begin{prop}
\label{prop-block-renorm}
There exists $C >0$ such that, for any integer $m$ with $1 \leq m \leq Cr$, we have
\[
f(r) \geq C \cdot \min\{m, \left(\frac{r}{m}\right)^2 f(m)  \}.
\]
\end{prop}

\begin{proof}
Let us give a sketch of the proof of this result, as it is very similar to the one of {\cite[Proposition~1]{chn}}. The idea is to consider the shortest (oriented) path going from a vertex on the left boundary of $\mathcal{G}_r$, to its right boundary, and whether or not it leaves a small horizontal layer of a specific type.

More precisely, we pick a vertex $x$ on the left boundary of $\mathcal{G}_r$, at a height $0 \leq j \leq r$. Then, we can find an integer $l$ such that $x$ is located in the layer $\mathcal{L}_{[lm, (l+2)m]}$, with $0 \leq l \leq (r/m) -2$ and so that $\lvert lm - j \rvert \geq m/3$ and $\lvert (l+2)m - j \rvert \geq m/3$. Consider then the shortest oriented path from $x$ to the right boundary of $\mathcal{G}_r$. Either it stays in that layer, or it leaves it at some point. 

If it leaves the layer, then its length is bounded below by $m/3$.

If it does not leave this layer, then its length is bounded below by 
\[
\sum_{i=1}^{N_r(2m,lm)}\text{Diam}(\mathcal{G}_{2m}(i,lm)).
\]
Then, from \Cref{lem-lots-of-subblocks}, and from the definition of $f$, we get that, for $r/m$ large enough, for some $c'>0$ independent of $r,l,m$,
\[
\Prob{\{N_r(2m,lm) < c\left(r/m\right)^2\} \cup \{\exists i\, \in \{1, \dots, N_r(2m,lm)\} \ \ \text{Diam}(\mathcal{G}_{2m}(i,lm))\leq c'f(2m)\}}
\leq \frac{1}{4},
\]
so that
\[
\Prob{\text{Diam}(\mathcal{G}_r) \leq  m/3 \wedge c \cdot c' (r/m)^2f(2m)} \leq \frac{1}{4},
\]
which implies the desired bound, by the definition of $f$.

The details of the proof can be adapted from the proof of {\cite[Proposition~1]{chn}}.
\end{proof}

\Cref{thm-median-block-diameter} is then a purely analytic consequence of \Cref{prop-block-renorm}, and its proof is a straightforward adaptation of that of {\cite[Theorem~5]{chn}}.\\

We can now use \Cref{thm-median-block-diameter} to obtain the following lower bounds for the distances along the boundary of $\mathcal{L}$:

\begin{prop}
\label{prop--l-lowerbound}
For every $\varepsilon >0$, there exists an integer $K >0$ such that, for every $r \geq 1$,
\[
\Prob{\min_{\lvert j \rvert \geq Kr^2}\vec{d}_{\mathcal{L}}((0,0),(j,0)) \geq r} \geq 1 - \varepsilon.
\]
Consequently, for $K'=9K$, we also have, for every $r \geq 1$,
\[
\Prob{\min_{\lvert j \rvert \geq 2 K'r^2} \ \ \ \min_{-K'r^2 \leq i \leq K'r^2}\vec{d}_{\mathcal{L}}((i,0),(j,0)) \geq r} \geq 1 - 2\varepsilon.
\]
\end{prop}

\begin{proof}
Let us start with the first assertion. Let $\varepsilon >0$. Fix $r \geq 1$, and $K \geq 1$. Then, from \eqref{proba-gen-r-0}, the number $N_{(K,r)}$ of trees that reach height $r$ between $(0,0)$ and $(j,0)$ is bounded below by a binomial variable of parameters $(Kr^2, 3/((r+2)^2-1))$, so that, using Chebyshev's inequality, for any $a >0$,
\[
\Prob{N_{(K,r)} \leq \frac{3}{8}K - a} \leq \frac{3K}{a^2}.
\]
(Note that the binomial variable in question has expectation greater than or equal to $3K/8$, with equality when $r=1$, and a variance smaller than $3K$.)

Taking $a = \sqrt{(6K/\varepsilon)}$, for $K$ large enough that $a \leq (1/8)K +1$, we get
\begin{equation}
\Prob{N_{(K,r)} \leq \frac{1}{4}K +1} \leq \frac{\varepsilon}{2}.
\end{equation}

Now, on the event that $N_{(K,r)} > K/4 $, for any $j \geq Kr^2$, we have
\[
\vec{d}_{\mathcal{L}}((0,0),(j,0)) \geq \sum_{i=1}^{\lfloor \frac{K}{4} \rfloor + 1}\text{Diam}(\mathcal{G}_r(i)) \wedge r,
\]

so that, using \Cref{thm-median-block-diameter},
\[
\Prob{\vec{d}_{\mathcal{L}}((0,0),(j,0)) < cr \frac{K}{4} \wedge r} \leq \frac{1}{2^{K/4}}.
\]

Now, taking $K$ even larger if necessary, we can also have $cK/4 \geq 1$, and $1/2^{K/4} \leq \varepsilon/2$, which does give that, with probability at least $1 - \varepsilon$, for all $j \geq Kr^2$, $\vec{d}_{\mathcal{L}}((0,0),(j,0)) \geq r$. The case of negative $j$ can be treated in the same way.

Let us now turn to the second assertion. Assume that there exist $j \geq 2K'r^2$ and $i \in \{-K'r^2, \dots, K'r^2\}$, such that $\vec{d}_{\mathcal{L}}((i,0),(j,0)) < r$. Then, any geodesic from $(i,0)$ to $(j,0)$ must stay in the layer $\mathcal{L}_{[0,r]}$, and therefore must intersect the leftmost mirror geodesic from $(K'r^2,0)$ to the line $\mathcal{L}_r$, so that
\[
\vec{d}_{\mathcal{L}}((K'r^2,0),(j,0))<3r.
\]
But then, by the first assertion of the proposition, the probability of such an event is bounded above by $\varepsilon$. Considering also the case $j <-K'r^2$, we obtain the desired result.
\end{proof}

An alternative proof of this result, adapting to Eulerian triangulations the method used in {\cite{curien-legall}} for usual triangulations, can be found in Chapter 8 of {\cite{carrance-th}}. Note that {\cite{lehericy}}, that adapts the results of \cite{curien-legall} to planar quadrangulations, and that was written simultaneously to the present work, also uses a block decomposition similar to \cite{chn}.

\subsection{Upper bounds}
\label{subsec upper bounds}
After having proved in the previous subsection lower bounds for the distances along the boundary of $\mathcal{L}$, we now prove upper bounds for these quantities, that will carry to the UIPT of the digon, $\mathcal{T}_{\infty}^{(1)}$, thanks to \Cref{prop-comparison-pple}. For that purpose, we follow closely the chain of arguments leading to Proposition 17 in {\cite[Section~4.3]{curien-legall}}. These bounds are expressed in terms of coalescence of leftmost mirror geodesics:
\begin{prop}
\label{prop-geodesics-coalesce-fast-in-l}
Let $\delta >0$ and $\gamma>0$. We can choose an integer $A \geq 1$ such that, for every sufficiently large $n$, with probability at least $1 - \delta$:

$\forall \, i \in \{-n+1, -n+2, \dots, n\}$, the leftmost mirror geodesic starting from $(i,0)$ coalesces with the one starting from $(-n+\lfloor 2ln/A\rfloor,0)$, for some $0 \leq l \leq A$, before hitting $\mathcal{L}_{\lfloor \gamma \sqrt{n} \rfloor}$.
\end{prop}

 The proof of this proposition can be adapted straightforwardly from that of {\cite[Proposition~16]{curien-legall}}.\\

We now derive a similar result for $\mathcal{T}_{\infty}^{(1)}$. Recall the notation $L_r$ for the number of skeleton modules on $\partial^{*}B_r^{\bullet}(\mathcal{T}_{\infty}^{(1)})$.

For any integer $n \geq 1$, we write $u_0(n)$ for a vertex chosen uniformly at random in the vertices of type $n$ of $\partial^{*}B_n^{\bullet}(\mathcal{T}_{\infty}^{(1)})$, and $u_1(n), \dots, u_{L_n-1}(n)$ for the other type-$n$ vertices of $\partial^{*}B_n^{\bullet}(\mathcal{T}_{\infty}^{(1)})$, enumerated clockwise, starting from $u_0(n)$. We extend the definition of $u_i(n)$ to $i \in \mathbb{Z}$ by periodicity. 

\begin{prop}
\label{prop-geodesics-coalesce-fast-in-uipet}
Let $\gamma \in (0,1/2)$ and $\delta >0$. For every integer $A \geq 1$, let $H_{n,A}$ be the event where any leftmost mirror geodesic to the root starting from a type-$n$ vertex of $\partial^{*}B_n^{\bullet}(\mathcal{T}_{\infty}^{(1)})$ coalesces before time $\lfloor \gamma n \rfloor$ with the leftmost mirror geodesic to the root starting from $u_{\lfloor kn^2/A \rfloor}(n)$, for some $0 \leq k \leq \lfloor n^{-2}L_nA \rfloor$. Then, we can choose $A$ large enough that, for every sufficiently large $n$,
\[
\Prob{H_{n,A}} \leq 1 - \delta.
\]
\end{prop}

\begin{proof}
The idea of the proof is to carry the result of \Cref{prop-geodesics-coalesce-fast-in-l} over to the case of $\mathcal{T}_{\infty}^{(1)}$, using the \emph{comparison principle} of \Cref{prop-comparison-pple}. To apply it, one needs to consider the intersection of $H_{n,A}$ with an event of the form
\begin{equation}
\{ \lfloor an^2 \rfloor < L_n \leq \lfloor a^{-1}n^2 \rfloor \} \cap \{ \lfloor an^2 \rfloor < L_{n- \lfloor \gamma n \rfloor} \leq \lfloor a^{-1}n^2 \rfloor \}.
\end{equation}
\Cref{top-length-bounds} ensures that we can choose an $a >0$ such that this latter event holds with probability at least $1 - \delta/2$.

The details of the proof can be adapted verbatim from the proof of Proposition 17 in {\cite{curien-legall}}.
\end{proof}

\section{Asymptotic equivalence between oriented and non-oriented distances}
\label{sec subadd}

Recall that, on any Eulerian triangulation with a boundary $A$, we write $\vec{d}_A$ for the oriented distance on $A$, and $d_A$ for the usual graph distance. We will show that these two distances are asymptotically proportional, first on the layers of the LHPET $\mathcal{L}$, then on the ones of the UIPET $\mathcal{T}_{\infty}^{(1)}$, and finally in large finite Eulerian triangulations. To do so, we follow the chain of proofs of Sections 5 and 6 in {\cite{curien-legall}}, once again detailing mostly the additional arguments needed in our case.

\subsection{Subadditivity in the LHPET and the UIPET}

Recall that we write $\rho$ for the root vertex $(0,0)$ of the LPHET $\mathcal{L}$, and that $\mathcal{L}_r$ is the lower boundary of the layer $\mathcal{L}_{[0,r]}$. We have the following result:

\begin{prop}
\label{prop-subadd-l}
There exists a constant $\mathbf{c}_0 \in [2/3, 1]$ such that
\[
r^{-1}d_{\mathcal{L}}(\rho,\mathcal{L}_r) \xrightarrow[r\to \infty]{a.s.} \mathbf{c}_0.
\]
\end{prop}

\begin{proof}
The proof of this result, apart from the bounds on $\mathbf{c}_0$, is essentially the same as that of {\cite[Proposition~18]{curien-legall}}. However, as it is a very short argument, but central in this whole work, we write it here in its entirety.

For integers $0 \leq m < n$, recall that $\mathcal{L}_{[m,n]}$ is the infinite planar map obtained by keeping only the layers of $\mathcal{L}$ between the levels $m$ and $n$. The non-oriented distance $d_{\mathcal{L}_{[m,n]}}$ on this strip is defined by considering the shortest non-oriented paths that stay in $\mathcal{L}_{[m,n]}$. Thus, for two vertices $v, v' \in \mathcal{L}_{[m,n]}$, we have $d_{\mathcal{L}_{[m,n]}}(v,v') \geq d_{\mathcal{L}}(v,v')$.

Let then $m,n \geq 1$, and let $x_m$ be the leftmost vertex $x$ of $\mathcal{L}_m$ such that $d_{\mathcal{L}}(\rho,\mathcal{L}_m)=d_{\mathcal{L}}(\rho,x)$. We have
\[
d_{\mathcal{L}}(\rho,\mathcal{L}_{m+n}) \leq d_{\mathcal{L}}(\rho,\mathcal{L}_m) + d_{\mathcal{L}_{[m,m+n]}}(x_m,\mathcal{L}_{m+n}).
\]
As $x_m$ is a function of $\mathcal{L}_{[0,m]}$ only, and the layers in $\mathcal{L}$ are independent, the random variable $d_{\mathcal{L}_{[m,m+n]}}(x_m,\mathcal{L}_{m+n})$ is independent of $\mathcal{L}_{[0,m]}$, and has the same distribution as $d_{\mathcal{L}}(\rho,\mathcal{L}_n)$. 

We can then apply Liggett's version of Kingman's subbadditive theorem {\cite{liggett}}, to get the desired convergence: the fact that the limit is a constant follows from Kolmogorov's zero-one law. As for the bounds for $\mathbf{c}_0$, it is clear from \eqref{eq bound oriented dist by usual dist} that $\mathbf{c}_0 \in [1/2,1]$. Our proof that $\mathbf{c}_0$ must be at least $2/3$ relies on a result of asymptotic proportionality in \emph{finite} Eulerian triangulations, that will be stated further in \Cref{thm-total-asympt-prop-in-finite-trig}. We thus postpone this argument to after \Cref{thm-total-asympt-prop-in-finite-trig}.
\end{proof}

To carry this asymptotic proportionality over to large finite Eulerian triangulations, we will make a stop at the UIPET of the digon $\mathcal{T}_{\infty}^{(1)}$. In the remainder of this subsection, we write $d$ for the non-oriented distance on $\mathcal{T}_{\infty}^{(1)}$, $B_n^{\bullet}$ for $B_n^{\bullet}(\mathcal{T}_{\infty}^{(1)})$ and $\partial^{*}B_n^{\bullet}$ for $\partial^{*}B_n^{\bullet}(\mathcal{T}_{\infty}^{(1)})$, to simplify notation.

\begin{prop}
\label{prop-almost-proportional-in-close-layers-uipet}
Let $\varepsilon, \delta \in (0,1)$. We can find $\eta \in (0,1/2)$ such that, for every sufficiently large $n$, the property
\[
(1-\varepsilon)\mathbf{c}_0 \eta n \leq d(v,\partial^{*}B_{n-\lfloor \eta n \rfloor}^{\bullet} )\leq (1+\varepsilon)\mathbf{c}_0 \eta n \ \ \ \ \forall \, v \in \partial^{*}B_n^{\bullet}
\]
holds with probability at least $1-\delta$.
\end{prop}

\begin{proof}
Let us give a sketch of the proof, as it is very similar to the proof of Proposition 19 in {\cite{curien-legall}}. Recall the notation $u_j^{(n)}$ for the type-$n$ vertices of $\partial^{*}B_n^{\bullet}$. The first key step is to use \Cref{prop--l-lowerbound} to get that a non-oriented shortest path from some $u_j^{(n)}$ to $\partial^{*}B_{n-\lfloor \eta n \rfloor}^{\bullet}$ that stays in $B_n^{\bullet}$ cannot meander too much in the layer $B_n^{\bullet} \setminus B_{n-\lfloor \eta n \rfloor}^{\bullet}$, and, more precisely, that it must stay in the region bounded by the leftmost mirror geodesics starting at $u_{j-\lfloor cn^2 \rfloor}^{(n)}$ and $u_{j+\lfloor cn^2 \rfloor}^{(n)}$ respectively, for some $c >0$. Then, to bound probabilities of events on that sector of $B_n^{\bullet} \setminus B_{n-\lfloor \eta n \rfloor}^{\bullet}$, \Cref{prop-comparison-pple} together with \Cref{top-length-bounds} allows us to replace the skeleton of $B_n^{\bullet} \setminus B_{n-\lfloor \eta n \rfloor}^{\bullet}$ by independent Galton-Watson trees. We can therefore transfer the property of \Cref{prop-subadd-l} from $\mathcal{L}$ to $B_n^{\bullet} \setminus B_{n-\lfloor \eta n \rfloor}^{\bullet}$. Finally, to consider all vertices of $\partial^{*}B_n^{\bullet}$, we use the coalescence property obtained in \Cref{prop-geodesics-coalesce-fast-in-uipet}, which amounts to saying that it suffices to consider for the values of $j$ a fixed number $C$, large but independent of $n$.

The details of the proof can be adapted verbatim from the proof of Proposition 20 in {\cite{curien-legall}} (replacing $d_{\text{gr}}$ by $\vec{d}$, and $d_{\text{fpp}}$ by $d$), with a small caveat.

Indeed, when using the coalescence property of \Cref{prop-geodesics-coalesce-fast-in-uipet} (which corresponds to (57) in {\cite{curien-legall}}), one must pay attention to two things. 

First, \Cref{prop-geodesics-coalesce-fast-in-uipet} only gives an upper bound on the distances between vertices of $\partial^{*}B_n^{\bullet}$ \emph{of type $n$}, and the $C$ chosen $u_j^{(n)}$. To also include the vertices of type $n+1$, one must add an additional margin of 1 to the bounds, which, for any fixed $\varepsilon$, can be smaller than $\varepsilon \mathbf{c}_0\eta n/2$, for $n$ large enough.

A second restriction of the application of \Cref{prop-geodesics-coalesce-fast-in-uipet} is that it ensures that \emph{oriented} geodesics \emph{from the root} to a type-$n$ vertex $v$ of $\partial^{*}B_n^{\bullet}$ and to one of the chosen $u_j^{(n)}$, are merged up to a level $\gamma n$. Thus, the upper bound on the oriented distance between $v$ and $u_j^{(n)}$ is not $2 \gamma n$ but $3 \gamma n$.

Thus, rather than $\gamma=\varepsilon \mathbf{c}_0 \eta /2$, we take $\gamma=\varepsilon \mathbf{c}_0 \eta /6$, to obtain the equivalent of (57) in {\cite{curien-legall}} for all vertices of $\partial^{*}B_n^{\bullet}$.
\end{proof}

We then derive a more global result from the one of \Cref{prop-almost-proportional-in-close-layers-uipet}:

\begin{prop}
\label{prop-almost-proportional-from-root-uipet}
For every $\varepsilon \in (0,1)$,
\[
\Prob{(\mathbf{c}_0-\varepsilon)n \leq d(\rho,v) \leq (\mathbf{c}_0+\varepsilon)n \, \text{ for every vertex } v \in \partial^*B_{n}^{\bullet}} \xrightarrow[n \to \infty]{} 1.
\]
\end{prop}
The proof of this result is straightforwardly adapted from the proof Proposition 20 of {\cite{curien-legall}}, replacing $d_{\text{gr}}$ by $\vec{d}$, and $d_{\text{fpp}}$ by $d$.

\subsection{Asymptotic proportionality of distances in finite triangulations}

We now turn to finite triangulations. More precisely, we consider $\mathcal{T}_{n}^{(1)}$, uniform on the Eulerian triangulations of the digon with $n$ black triangles. Recall that such triangulations are in bijection with (rooted) Eulerian triangulations with $n$ black faces, from \Cref{DigonToRoot}. We write $\rho_n$ for the root of $\mathcal{T}_{n}^{(1)}$, and $d$ for the non-oriented distance on $\mathcal{T}_{n}^{(1)}$.

\begin{prop}
\label{prop-asympt-proportionality-in-finite-trig-from-root}
Let $o_n$ be uniform over the inner vertices of $\mathcal{T}_{n}^{(1)}$. Then, for every $\varepsilon > 0$,
\[
\Prob{\lvert d(\rho_n,o_n) - \mathbf{c}_0 \vec{d}(\rho_n,o_n) \rvert > \varepsilon n^{1/4}} \xrightarrow[n \to \infty]{} 0.
\]
\end{prop}

To derive this from the previous results on $\mathcal{T}_{\infty}^{(1)}$, we will first establish an absolute continuity relation between finite triangulations and this infinite model.

Recall that $\mathbb{C}_{1,r}$ is the set of Eulerian triangulations of the cylinder of height $r$ and bottom boundary length $2$. For $\Delta \in \mathbb{C}_{1,r}$, we denote by $N(\Delta)$ the number of black triangles in $\Delta$.
Finally, we write $\overline{\mathcal{T}}_n^{(1)}$ for the triangulation $\mathcal{T}_n^{(1)}$ together with a distinguished vertex $o_n$. The hull $B_r^{\bullet}(\overline{\mathcal{T}}_n^{(1)})$ is well-defined when $\vec{d}(\rho_n,o_n) >r +1$, otherwise we set it to be $\overline{\mathcal{T}}_n^{(1)}$.

\begin{lemma}
\label{lem-bound-cylinder-in-finite-by-in-uipet}
There exists a constant $\bar{c} >0$ such that, for every $n,r,p \geq 1$ and every $\Delta \in \mathbb{C}_{1,r}$ with top boundary half-length $p$, such that $n > N(\Delta)+p$,
\begin{equation}
\Prob{B_r^{\bullet}(\overline{\mathcal{T}}_n^{(1)}) = \Delta} \leq \bar{c} \left( \frac{n}{n-N(\Delta)+1} \right)^{3/2} \cdot \Prob{B_r^{\bullet}(\mathcal{T}_{\infty}^{(1)}) = \Delta}.
\end{equation}
\end{lemma}

\begin{proof}
The proof of this lemma is very similar to that of {\cite[Lemma~22]{curien-legall}}, with the additional subtlety that, like for \Cref{skeleton-proba-limit-lem}, we do not start with explicit expressions for probabilities in finite triangulations, as shown in \eqref{eq:bound-inner-vertices}.

Fix $r \geq 1$ and $\Delta \in \mathbb{C}_{1,r}$ with top boundary half-length $p$. We will write $V$ for $\#V(\Delta)$ to simplify notation. Using \eqref{eq:limit-probas-for-tall-pointed-cylinder} and the fact that $\mathcal{T}_{\infty}^{(1)}$ is the local limit of $\mathcal{T}_n^{(1)}$, we have
\begin{equation}
\label{eq:cylinder-in-uipet}
\Prob{B_r^{\bullet}(\mathcal{T}_{\infty}^{(1)})=\Delta} = \frac{C(p)}{C(1)}8^{-N(\Delta)}.
\end{equation}

On the other hand, \eqref{eq:probas-for-pointed-cylinder} gives the formula
\begin{align}
\label{eq:bound-inner-vertices}
\mathbb{P}\left(B_r^{\bullet}(\overline{\mathcal{T}}^{(1)}_n)=\Delta \right) &= \frac{B_{n-N,p}}{B_{n,1}}\cdot \frac{\#\text{inner vertices in $\mathcal{T}^{(1)}_n \setminus \Delta$}}{\#\text{inner vertices in $\mathcal{T}^{(1)}_n$}}\\
&\leq \frac{B_{n-N,p}}{B_{n,1}}\cdot \frac{n-V}{n},
\end{align}
where the last inequality is given by Euler's formula and the fact that at most $p$ vertices of $\partial^*\Delta$ are identified together in $\mathcal{T}^{(1)}_n$. (We still need $n > N+p$ since $\mathcal{T}^{(1)}_n \setminus \Delta$ will have $n-N-p$ inner vertices if none of these identifications occur.)

Then, using the bounds of \eqref{eq:gf-asympt-bounds} and the asymptotics of \eqref{gf-expression}, we get that
\[
\mathbb{P}\left(B_r^{\bullet}(\overline{\mathcal{T}}^{(1)}_n)=\Delta \right) \leq c^* C(p) \left( \frac{n}{n-N}\right)^{3/2}8^{-N}
\]
for some constant $c^*$. Comparing the last bound with \eqref{eq:cylinder-in-uipet} gives the desired result.
\end{proof}

\begin{proof}[of \Cref{prop-asympt-proportionality-in-finite-trig-from-root}]
Fix $\varepsilon >0$ and $\nu >0$. Its suffices to prove that, for all $n$ sufficiently large, we have
\begin{equation}
\label{eq asympt prop origin to uniform as ratio}
\Prob{\left\lvert \frac{d(\rho_n,o_n)}{\vec{d}(\rho_n,o_n)} - \mathbf{c}_0\right \rvert >2\varepsilon} < \nu.
\end{equation}

Indeed, as detailed in \Cref{prop geod cv to sup z}, the sequence $n^{-1/4}\vec{d}(\rho_n,o_n)$ is bounded in probability, so that the statement of the proposition will follow from \eqref{eq asympt prop origin to uniform as ratio}. Note that, in {\cite{curien-legall}}, the equivalent tightness is obtained as a consequence of the convergence of usual planar triangulations to the Brownian map: in our case, we had to use the weaker result of \Cref{thm-cv-of-labeled-tree}, as we obviously do not have a convergence at the level of maps yet.

To obtain \eqref{eq asympt prop origin to uniform as ratio}, we want to transfer the results of \Cref{prop-almost-proportional-from-root-uipet} on the UIPET to large finite triangulations. This necessitates the bounds of \Cref{lem-bound-cylinder-in-finite-by-in-uipet}, together with the statement of \Cref{prop-positive-mass-to-any-neigb-of-zero-in-distance-profile} on the profile of distances in large finite triangulations (which is once again a consequence of \Cref{thm-cv-of-labeled-tree}, as we cannot rely on a convergence to the Brownian map).

We omit the details of the proof of \eqref{eq asympt prop origin to uniform as ratio}, as they can be straightforwardly adapted from the equivalent statement in the proof of Proposition 21 in {\cite{curien-legall}}, replacing once again $d_{\text{gr}}$ by $\vec{d}$, and $d_{\text{fpp}}$ by $d$.
\end{proof}

We will now derive our final result of asymptotic proportionality between the oriented and non-oriented distances, \Cref{thm-total-asympt-prop-in-finite-trig}. This one is in the context of $\mathcal{T}_n$, the uniform rooted plane Eulerian triangulation with $n$ black faces, that is in correspondence with $\mathcal{T}_n^{(1)}$ as shown in \Cref{DigonToRoot}. As previously, we use $d$ to denote the non-oriented graph distance.

\begin{proof}[of \Cref{thm-total-asympt-prop-in-finite-trig}]
Let us give an idea of the proof of this theorem, which follows the arguments of the proof of Theorem 1 in {\cite{curien-legall}}.

From \Cref{prop-asympt-proportionality-in-finite-trig-from-root} and the correspondence between $\mathcal{T}_n^{(1)}$ and $\mathcal{T}_n$, we get that, if $o'_n$ is a uniform vertex of $\mathcal{T}_n$, we have
\begin{equation}
\label{eq:asympt-proportionality-from-root-without-boundary}
\Prob{\lvert d(\rho_n,o'_n) - \mathbf{c}_0\vec{d}(\rho_n,o'_n) \rvert > \varepsilon n^{1/4}} \xrightarrow[n \to \infty]{} 0.
\end{equation} 

Observe now that $\overline{\mathcal{T}}_n$, re-rooted at $\rho'_n$, the origin vertex of a random uniform edge $e_n$ (remember that all edges of $\mathcal{T}_n$ have a canonical orientation), still pointed at $o'_n$, has the same distribution as $\overline{\mathcal{T}}_n$. This implies that the statement of \eqref{eq:asympt-proportionality-from-root-without-boundary} also holds for the distances from $\rho'_n$, that is sampled according to its degree:
\begin{equation*}
\Prob{\lvert d(\rho'_n,o'_n) - \mathbf{c}_0\vec{d}(\rho'_n,o'_n) \rvert > \varepsilon n^{1/4}} \xrightarrow[n \to \infty]{} 0.
\end{equation*} 

 As the numbers of edges and vertices of $\mathcal{T}_n$ are fixed, this allows us to deduce a similar statement on distances between two random uniform vertices $o'_n, o^{\prime\prime}_n$ of $\mathcal{T}_n$:
\begin{equation}
\label{eq:asympt-proportionality-between-two-random-verts-without-boundary}
\Prob{\lvert d(o'_n,o^{\prime\prime}_n) - \mathbf{c}_0\vec{d}(o'_n,o^{\prime\prime}_n)  \rvert > \varepsilon n^{1/4}} \xrightarrow[n \to \infty]{} 0.
\end{equation}

We now want to make this statement into a global one on all the vertices of $\mathcal{T}_n$.

Let us fix $\delta \in (0,1/2)$. We can choose an integer $k \geq 1$ such that, for every $n$ sufficiently large, we can pick $k$ random vertices $(o_n^1, \dots, o_n^k)$ uniformly in $\mathcal{T}_n$ and independently from one another, satisfying
\begin{equation}
\label[ineq]{eq:sprinkling-gets-close-to-every-vertex}
\Prob{\sup_{x \in V(\mathcal{T}_n)} \left( \inf_{1 \leq j \leq k}\vec{d}(x,o^j_n)\right) < \varepsilon n^{1/4}} > 1- \delta.
\end{equation}
This follows from \Cref{prop-sprinkling}. Note that, once again, the equivalent property in {\cite{curien-legall}} was obtained as a consequence of the convergence of usual planar triangulations to the Brownian map, whereas here we had to obtain it from the convergence of the rescaled oriented distances from $o_n$ to a Brownian snake, which is a weaker result.

Then, \eqref{eq:asympt-proportionality-between-two-random-verts-without-boundary} implies that we also have, for all sufficiently large $n$, 
\[
\Prob{\bigcap_{1 \leq i \leq j \leq k} \{ \lvert d(o^i_n,o^j_n) - \mathbf{c}_0\vec{d}(o^i_n,o^j_n)  \rvert \leq \varepsilon n^{1/4} \}}> 1- \delta.
\]

Observe now that
\[
\sup_{x,y \, \in V(\mathcal{T}_n)}\lvert d(x,y) - \mathbf{c}_0\vec{d}(x,y) \rvert \leq \sup_{1 \leq i,j \leq N}\lvert d(o^i_n,o^j_n) - \mathbf{c}_0\vec{d}(o^i_n,o^j_n)  \rvert + 5 \sup_{x \in V(\mathcal{T}_n)}\left( \inf_{1 \leq j \leq N} \vec{d}(x,o^j_n)\right).
\]
Using the previous two bounds, the right-hand side of this inequality can be bounded by $6 \varepsilon$ outside a set of probability at least $2\delta$ for all sufficiently large $n$, which concludes the proof.
\end{proof}
\bigskip

Let us finally give a short proof of why $\mathbf{c}_0\geq 2/3$. Consider $\mathcal{T}_n$, the uniform rooted plane Eulerian triangulation with $n$ black faces. From \Cref{thm-total-asympt-prop-in-finite-trig}, for any $\varepsilon, \delta \in (0,1/2)$, for any $n$ large enough, 
\begin{equation}
\label{eq asympt prop last time}
\lvert d_n(x,y) - \mathbf{c}_0\vec{d}_n(x,y) \rvert \leq \varepsilon n^{1/4}, \ \ \ \ \forall x,y \in V(\mathcal{T}_n),
\end{equation}
outside an event of probability less than $\delta$.

Suppose that $\mathbf{c}_0< 2/3$. Let us fix $n \geq 1$, and consider some $c \in (0,1) $. Then, on the event of \eqref{eq asympt prop last time}, for any $x,y \in V(\mathcal{T}_n)$ such that 
\begin{equation}
\label{eq dist scaling for excluding 1/2}
d_n(x,y) \geq cn^{1/4},
\end{equation} we have:
\[
\vec{d}_n(x,y) \geq \left(\frac{1}{\mathbf{c}_0} - \frac{\varepsilon}{c}\right)d_n(x,y).
\]
This means that, for any geodesic $\gamma$ for the distance $d_n$ from $x$ to $y$ in $\mathcal{T}_n$, a fraction larger than or equal to $(1/\mathbf{c}_0-1-\varepsilon/c)$ of the edges of $\gamma$ are oriented from $y$ to $x$. But, as the above bound also applies when we exchange $x$ and $y$, a same fraction of edges of $\gamma$ must be oriented from $x$ to $y$, which is not possible if $(1/\mathbf{c}_0-1-\varepsilon/c)>1/2$, that is, $\varepsilon/c < 1/\mathbf{c}_0 -3/2$.

Since, for any $\delta \in (0,1/2)$, there exists a $c(\delta) \in (0,1)$ such that, if $n$ is large enough, outside of an event of probability less than $\delta$, a positive proportion of pairs of vertices of $\mathcal{T}_n$ satisfy \eqref{eq dist scaling for excluding 1/2}, we deduce that \eqref{eq asympt prop last time} cannot have a high probability for large $n$, if $\mathbf{c}_0<2/3$.\\

It would be interesting to also refine the upper bound on $\mathbf{c}_0$. However, this seems to necessitate deeper arguments than our refinement of the lower bound.

\section{Convergence for the Riemannian distance}
\label{sec riem}

We now turn our attention to another distance that can be defined on Eulerian triangulations, the \textbf{Riemannian distance \boldmath\(d^R\)}. To define this distance, we start by assigning to a triangulation \unboldmath$A$, the piecewise-linear metric space $S(A)$, obtained by gluing equilateral, Euclidean triangles with sides of unit length, according to the combinatorics of $A$. We call this space the \textbf{Euclidean geometric realization} of $A$. It naturally comes endowed with a metric, that we denote by $d^R$, and, by a slight abuse of notation, we also denote by $d^R$ the induced distance on the vertices of $A$.

We want to show that, like the usual graph distance $d$, the Riemannian distance $d^R$ is asymptotically proportional to the oriented distance $\vec{d}$, so that, endowed with $d^R$, the uniform Eulerian triangulation $\mathcal{T}_n$ still converges to the Brownian map. This can be once again proven using the layer decomposition of finite and infinite Eulerian triangulations with respect to $\vec{d}$, together with an ergodic subadditivity argument. Once this argument gives the desired asymptotic proportionality on $\mathcal{L}$, the results of \Cref{sec subadd} can be directly adapted to $d^R$, to obtain the new convergence to the Brownian map.

However, the subbadditivity argument presents here a hurdle that was not present in the case of $d$: indeed, while we still have the immediate upper bound $d^R \leq \vec{d}$ (and even: $d^R \leq d$), we have no obvious way to bound $d^R$ from below with $\vec{d}$. Such a bound is crucial, since, without it, the proportionality constant given by the ergodic subbadditivity theorem could very well be zero. We therefore prove the following result:

\begin{prop}
\label{prop lower bound on riem dist}
Let $A$ be a triangulation, endowed with its graph distance $d$, canonical oriented pseudo-distance $\vec{d}$ and Riemannian distance $d^R$. Then,  
\[
d^R \geq \frac{\sqrt{3}}{4} d.
\]
Consequently, 
\[
d^R \geq \left(\frac{\sqrt{3}}{8}\right) \vec{d}.
\]
\end{prop}

\begin{proof}
Let us prove the first bound, as the second is a direct consequence of it, together with \eqref{eq bound oriented dist by usual dist}.

Let $A$ be a triangulation, and $S(A)$ its Euclidean geometric realization. For any continuous path $\gamma: [0,1] \to S(A)$, we will construct an edge path $\gamma_E$ in $A$, such that, if $\gamma$ is a geodesic, then its length $l(\gamma)$ can be bounded from below by $\sqrt{3}/4$ times the number of edges in $\gamma_E$, which gives the desired inequality.

Let thus $\gamma$ be a continuous path in $S(A)$. We construct a sequence $(u_0, u_1, \dots, u_k)$ of vertices of $A$, in the following way. We start by setting $u_0$ to be the vertex closest to $\gamma(0)$ (if there is an ambiguity, we just pick one of the possible vertices in an arbitrary way). Let $f_0$ be the first triangle that $\gamma$ crosses (\emph{i.e.}, gets out of after having spent a positive time in it). Let then $u_1$ be the vertex closest to the point where $\gamma$ leaves for the last time any of the triangles incident to $u_0$. We then define $u_2$, etc. similarly. This yields a \emph{finite} sequence of vertices $u_0, u_1, \dots, u_k$: indeed, $\gamma$ cannot get close to an infinite number of distinct vertices of $S(A)$. Note also that $u_k$ is necessarily the closest vertex to $\gamma(1)$.

By construction, for any $0 \leq i \leq k -1$, $u_i$ and $u_{i+1}$ are neighbors in $A$, so that the sequence does induce a path $\gamma_E$ of $k$ edges.

\begin{figure}[htp]
\centering
\includegraphics[scale=2]{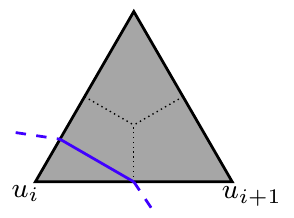}
\caption{The shortest distance that $\gamma$ can cross between the vicinity $u_i$ and the vicinity of $u_{i+1}$ corresponds to the altitude of an equilateral triangle of side length $1/2$, which is equal to $\sqrt{3}/4$ (the boundaries of the Voronoi cells associated to the vertices of the triangle are dashed, and $\gamma$ is in purple).}
\label{RiemDistMin}
\end{figure}
Consider now a geodesic $\gamma$ in $S(A)$ (with respect to $d^R$), going from a vertex $v$ to a different vertex $w$. Then, the restriction of $\gamma$ to any triangle it crosses is necessarily a straight line (since this is a geodesic on a Euclidean triangle). Therefore, the portion of $\gamma$ going from the first moment that $\gamma$ is closest to $u_i$, to the last moment it is closest to $u_{i+1}$, crosses at least one triangle, from one edge to another, and also crossing the Voronoi cell of one of the vertices. This implies that this portion of $\gamma$ has a length of at least $\sqrt{3}/4$ (see \Cref{RiemDistMin}). Thus,
\[
l(\gamma) \geq\frac{\sqrt{3}}{4} \cdot k = \frac{\sqrt{3}}{4} \cdot l(\gamma_e).
\]
This concludes the proof.
\end{proof}

Recall that we write $\rho$ for the root vertex $(0,0)$ of the LHPET $\mathcal{L}$, and that $\mathcal{L}_r$ is the lower boundary of the layer $\mathcal{L}_{[0,r]}$. We have the following result:

\begin{prop}
\label{prop-subadd-riem}
There exists a constant $\mathbf{c}_1 \in [(\sqrt{3}/4)\mathbf{c}_0, 1]$ such that 
\[
r^{-1}d^R_{\mathcal{L}}(\rho, \mathcal{L}_r) \xrightarrow[r \to \infty]{a.s.} \mathbf{c}_1.
\]
\end{prop}

In the sequel, $\mathbf{c}_1$ will refer to the constant of \Cref{prop-subadd-riem}.

\begin{proof}
We proceed like for \Cref{prop-subadd-l}, by considering the layers $\mathcal{L}_{[m,n]}$, for integers $0 \leq m \leq n$. Such a layer corresponds to a strip in the Euclidean geometrical realization $S(\mathcal{L})$ of $\mathcal{L}$: we can define the Riemannian distance $d^R_{\mathcal{L}_{[m,n]}}$ on the vertices of $\mathcal{L}_{[m,n]}$ by considering the shortest paths (\emph{starting and ending at vertices}) in $S(\mathcal{L})$ that stay in this strip. Then, we have, for any two vertices $v,v' \in \mathcal{L}{[m,n]}$, we have $d^R_{\mathcal{L}_{[m,n]}}(v,v') \geq d^R_{\mathcal{L}}(v,v')$.

Thus, as for the graph distance, if $m,n \geq 1$, and $x_m$ is the leftmost vertex $x$ of $\mathcal{L}_m$ such that $d^R_{\mathcal{L}}(\rho, \mathcal{L}_m)=d^R_{\mathcal{L}}(\rho,x)$, we have
\[
d^R_{\mathcal{L}}(\rho,\mathcal{L}_{m+n}) \leq d^R_{\mathcal{L}}(\rho,\mathcal{L}_{m}) + d^R_{\mathcal{L}_{[m,n]}}(x_m,\mathcal{L}_{m+n}).
\]
As in the case of the graph distance $d$, since $x_m$ is a function of $\mathcal{L}_{[0,m]}$ only, and since the layers in $\mathcal{L}$ are i.i.d., this yields the desired convergence. 

The upper bound on $\mathbf{c}_1$ is immediate; let us briefly explain how we obtain the lower bound. Fix $\varepsilon >0$ and $\delta \in (0,1/2)$. We have that, for $n$ large enough, outside of an event of probability less than $\delta$,
\[
\lvert d_n(x,y) - \mathbf{c}_0\vec{d}_n(x,y) \rvert \leq \varepsilon n^{1/4} \ \ \ \forall \, x,y \in V(\mathcal{T}_n),
\]
so that we get from \Cref{prop lower bound on riem dist}:
\[
 d_n^R(x,y) \geq \frac{\sqrt{3}}{4}\mathbf{c}_0\vec{d}_n(x,y) - \varepsilon n^{1/4}.
\]
Now, there exists a constant $0<C(\delta)<1$, that depends only on $\delta$, such that, for $n$ large enough, outside of an event of probability less than $\delta$, a positive proportion of pairs of vertices of $\mathcal{T}_n$ satisfy
\[
\vec{d}_n(x,y) \geq C(\delta)n^{1/4}.
\]
Therefore, for all such pairs, we have
\[
d_n^R(x,y) \geq \left(\frac{\sqrt{3}}{4}\mathbf{c}_0-\frac{\varepsilon}{C(\delta)}\right)\vec{d}_n(x,y),
\]
so that, necessarily, $\mathbf{c}_1 \geq \frac{\sqrt{3}}{4}\mathbf{c}_0$.
\end{proof}

Retracing for $d^R$ the same arguments as the ones used for $d$ in \Cref{sec subadd}, we deduce from \Cref{prop-subadd-riem} the following result:

\begin{thm}
\label{thm-proportionality-for-riem-dist}
Let $\mathcal{T}_n$ be a uniform random rooted Eulerian planar triangulation with $n$ black faces, and let $V(\mathcal{T}_n)$ be its vertex set. For every $\varepsilon >0$, we have 
\[
  \Prob{\sup_{x,y \in V(\mathcal{T}_n)}\lvert d^R_n(x,y) - \mathbf{c}_1\vec{d}_n(x,y) \rvert > \varepsilon n^{1/4}} \xrightarrow[n \to \infty]{} 0.
\]
\end{thm}

This allows us to add a third scaling limit to the joint convergence of \Cref{thm-cv-to-b-map}:
\begin{cor}
\label{coro cv Riem to B map}
Let $(\mathbf{m}_{\infty},D^*)$ be the Brownian map. We have the following joint convergences
\begin{align*}
n^{-1/4}\cdot (V(\mathcal{T}_n),\overleftrightarrow{d}_{\! \!n}) &\xrightarrow[n \to \infty]{(d)}  \phantom{\mathbf{c}_0 }\ \ (\mathbf{m}_{\infty},D^*)\\
n^{-1/4}\cdot (V(\mathcal{T}_n),d_n) &\xrightarrow[n \to \infty]{(d)}  \mathbf{c}_0 \cdot (\mathbf{m}_{\infty},D^*)\\
n^{-1/4}\cdot (S(\mathcal{T}_n),d^R_n) &\xrightarrow[n \to \infty]{(d)}  \mathbf{c}_1 \cdot (\mathbf{m}_{\infty},D^*),
\end{align*}
for the Gromov-Hausdorff distance on the space of isometry classes of compact metric spaces.
\end{cor}

Let us sketch very quickly the proof of \Cref{coro cv Riem to B map}: following the same steps as the ones we made for $d$ in \Cref{sec subadd}, the result of \Cref{thm-proportionality-for-riem-dist} implies that the Brownian map is the scaling limit of the vertex set $V(\mathcal{T}_n)$ endowed with the distance induced by $S(\mathcal{T}_n)$, and not $S(\mathcal{T}_n)$ itself. However, the Gromov-Hausdorff distance between $(S(\mathcal{T}_n),d^R)$ and $(V(\mathcal{T}_n),d^R)$ is at most $\sqrt{3}/4$ (considering $V(\mathcal{T}_n)$ as embedded into $S(\mathcal{T}_n)$), so that this convergence does extend to $S(\mathcal{T}_n)$.\\

As explained in the introduction, the result of \Cref{coro cv Riem to B map} allows us to make a more direct comparison between models of random maps as studied by probabilists, and models of $2D$ quantum gravity studied by theoretical physicists, such as Causal Dynamical Triangulations, as the latter models focus on the Euclidean geometric realization associated to some combinatorial maps. \\

Note that the geometric argument in the proof of \Cref{prop lower bound on riem dist} works for any triangulation, and not just an Eulerian one. Thus, relying on the layer decomposition of usual triangulations of {\cite{curien-legall}}, we can prove in the same way as here that usual triangulations, equipped with the Riemannian metric, also converge to the Brownian map. A similar geometric argument should also work for quadrangulations, along with the layer decomposition of {\cite{legall-lehericy}}.

\paragraph{Acknowledgements}

I warmly thank Grégory Miermont for his crucial help and advice, as well as Nicolas Curien and Christina Goldschmidt for their insightful remarks. I also thank the referees for their very helpful feedback.

\printbibliography

\contactrule
\contactACarrance

\end{document}